\documentclass[a4paper]{amsart}

\usepackage{setspace}
\usepackage[totalwidth=14cm,totalheight=22cm]{geometry}
\usepackage[T1]{fontenc}
\usepackage[dvips]{graphicx}
\usepackage{epsfig}
\usepackage{amsmath,amssymb,amscd,amsthm}%,extarrows}
\usepackage{subfigure}
\usepackage[latin1]{inputenc}
\usepackage{latexsym}
\usepackage{graphics}
\usepackage{colortbl}
\usepackage{color}
\usepackage{ae}
\usepackage{enumitem}
\usepackage[all]{xy}

\makeatletter
\newtheorem*{rep@theorem}{\rep@title}
\newcommand{\newreptheorem}[2]{%
\newenvironment{rep#1}[1]{%
 \def\rep@title{#2 \ref{##1}}%
 \begin{rep@theorem}}%
 {\end{rep@theorem}}}
\makeatother

\makeatletter
\newtheorem*{rep@cor}{\rep@title}
\newcommand{\newrepcor}[2]{%
\newenvironment{rep#1}[1]{%
 \def\rep@title{#2 \ref{##1}}%
 \begin{rep@cor}}%
 {\end{rep@cor}}}
\makeatother

\makeatletter
\newtheorem*{rep@prop}{\rep@title}
\newcommand{\newrepprop}[2]{%
\newenvironment{rep#1}[1]{%
 \def\rep@title{#2 \ref{##1}}%
 \begin{rep@prop}}%
 {\end{rep@prop}}}
\makeatother

\newtheorem{cor}{Corollary}[section]
\newtheorem{corx}{Corollary}

\newtheorem{theorem}[cor]{Theorem}
\newtheorem{thmx}[corx]{Theorem}

\newtheorem{prop}[cor]{Proposition}

\newrepcor{cor}{Corollary}
\newreptheorem{theorem}{Theorem}
\newrepprop{prop}{Proposition}

\newtheorem{lemma}[cor]{Lemma}

\theoremstyle{definition}
\newtheorem{defi}[cor]{Definition}
\theoremstyle{remark}
\newtheorem{remark}[cor]{Remark}
\newtheorem*{remark*}{Remark}
\newtheorem{example}[cor]{Example}

\DeclareRobustCommand{\bfs}[1]{
  \ifcat#1\relax
    \boldsymbol{#1}
  \else
    \mathbf{#1}
  \fi}

\newcommand{\cU}{{\mathcal U}}
\newcommand{\cJ}{{\mathcal J}}
\newcommand{\cH}{{\mathcal H}}
\newcommand{\cB}{{\mathcal B}}
\newcommand{\cR}{{\mathcal R}}
\newcommand{\cC}{{\mathcal C}}

\newcommand{\cF}{{\mathcal F}}
\newcommand{\cL}{{\mathcal L}}
\newcommand{\cM}{{\mathcal M}}

\newcommand{\cS}{{\mathcal S}}
\newcommand{\cT}{{\mathcal T}}

\newcommand{\D}{{\mathbb D}}

\newcommand{\Hh}{{\mathbb H}}

\newcommand{\R}{{\mathbb R}}

\newcommand{\E}{{\mathbb E}}

\newcommand{\dev}{\mathrm{dev}}

\newcommand{\hess}{\mbox{Hess}}

\newcommand{\Hyp}{\mathbb{H}}

\newcommand{\psl}{\mathfrak{sl}}

\newcommand{\Ad}{\mbox{Ad}}
\newcommand{\hol}{\mbox{hol}}
\newcommand{\Hol}{\mbox{\bf{hol}}}

\newcommand{\cun}{{\mathrm{C}}^\infty}

\newcommand{\dn}{d^\nabla}

\newcommand{\Diffeo}{\mbox{Diffeo}}
\newcommand{\Homeo}{\mbox{Homeo}}

\newcommand{\tr}{\mbox{\rm tr}}

\newcommand{\sh}{\mathrm{sinh}\,}
\newcommand{\ch}{\mathrm{cosh}\,}
\newcommand{\grad}{\operatorname{grad}}%{\mbox{grad}}
\newcommand{\isom}{\mathrm{Isom}}

\newcommand{\rar}{\rightarrow}

\newcommand{\ev}{\mbox{ev}}
\newcommand{\da}{\,\mathrm{dA}}
\newcommand{\Id}{\mathrm{I}}
\newcommand{\id}{\mathrm{I}}
\newcommand{\SO}{\mathrm{SO}}
\newcommand{\so}{\mathfrak{so}}

\newcommand{\cA}{\mathcal{A}}
\newcommand{\WP}{W\!P}

\def\Hess{\mathrm{Hess}\,}
\def\Teich{\mathcal{T}}

\def\ol#1{\overline{#1}}

\def\hyp{\mathbb H}
\def\Div{\mbox{div}}
\def\puct{\mathfrak{p}}
\begin{document}
%\begin{document}

\setcounter{secnumdepth}{3}
\setcounter{tocdepth}{2}

\title{On Codazzi tensors on a hyperbolic surface and flat Lorentzian geometry}

\author[Francesco Bonsante]{Francesco Bonsante}
\address{Francesco Bonsante: Dipartimento di Matematica ``Felice Casorati", Universit\`{a} degli Studi di Pavia, Via Ferrata 1, 27100, Pavia, Italy.} \email{bonfra07@unipv.it} 
\author[Andrea Seppi]{Andrea Seppi}
\address{Andrea Seppi: Dipartimento di Matematica ``Felice Casorati", Universit\`{a} degli Studi di Pavia, Via Ferrata 1, 27100, Pavia, Italy.} \email{andrea.seppi01@ateneopv.it}
%\date{\today}

\thanks{The authors were partially supported by FIRB 2010 project ``Low dimensional geometry and topology'' (RBFR10GHHH003). The first author was partially supported by
PRIN 2012 project ``Moduli strutture algebriche e loro applicazioni''.
The authors are members of the national research group GNSAGA}

%Partially supported by the FIRB 2011-2014 grant ``Geometry and topology of low-dimensional manifolds''.}

%\tableofcontents

\begin{abstract}
Using global considerations, Mess proved that the moduli space of globally hyperbolic flat Lorentzian structures on $S\times\R$ is the tangent bundle
of the Teichm\"uller space of $S$, if $S$ is a closed surface.
%By some global analysis, Mess proved that the relevant moduli space for 3D gravity is the tangent bundle
%of the Teichm\"uller space of the spatial part of the space-time.
One of the goals of this paper is to deepen this surprising occurrence and to make explicit the relation between the Mess parameters
and the embedding data of any Cauchy surface.
This relation is pointed out by using some specific properties of Codazzi tensors on hyperbolic surfaces.
As a by-product we get a new \emph{Lorentzian} proof of Goldman's celebrated result about the
coincidence of the Weil-Petersson symplectic form and the Goldman pairing.

In the second part of the paper we use this machinery to get a classification of globally hyperbolic flat
space-times with particles of angles in $(0,2\pi)$ containing a uniformly convex Cauchy surface. The analogue of Mess' result is achieved showing 
that the corresponding moduli space is the tangent bundle of the Teichm\"uller space of a punctured surface.
To generalize the theory in the case of particles, we deepen the study of Codazzi tensors on hyperbolic surfaces with cone singularities,
proving that the well-known decomposition of a Codazzi tensor in a harmonic part and a trivial part can be generalized in the context
of hyperbolic metrics with cone singularities.
\end{abstract}

\maketitle

%\bibstyle{plain}
%\cite{oliker}
%----------------------------------------------sezione 1---------------------------------
%\input bs-sec1.tex

\section{Introduction}
In 1990, in his pioneering work \cite{Mess} G. Mess showed that gravity in dimension $3$ is strictly related to Teichm\"uller theory.
Assuming the cosmological constant equal to $0$ (that in dimension $3$ corresponds to the flatness of the space), Mess showed that 
the relevant moduli space is the tangent bundle of the Teichm\"uller space of the spatial part of the universe.
In the Anti-de Sitter case  (negative cosmological constant) the moduli space is instead the Cartesian product of two copies of Teichm\"uller 
space. Since Mess' work, connections between $3$-dimensional gravity and Teichm\"uller theory and more generally hyperbolic geometry
 have been intensively exploited by a number of authors, see for instance \cite{notes, barbot, bebo,Schlenker-Krasnov, scannell}.

The main problem of gravity is to classify solutions of Einstein equation, that in dimension $3$ are Lorentzian metrics
of constant curvature.
Usually some   global causal conditions are required; the most important one  is the so called global hyperbolicity, that 
prescribes the existence of a Cauchy surface -- namely a surface which meets every inextensible causal path exactly once. 
Globally hyperbolic space-times are topologically a product $M=S\times\R$, where $S$ is the Cauchy surface. 
%Moreover the product structure can 
%be fixed so that each slice $S\times\{\bullet\}$ is a Cauchy surface.
In \cite{choquet},  Choquet-Bruhat proved that the embedding data of  any Cauchy surface $S\subset M$ -- 
that is, the first and the second fundamental form -- determine
a maximal extension of the space. This is an Einstein space-time $M_{\tiny\mathrm{max}}$
containing  $M$,  so that $S$ is a Cauchy surface of $M_{\tiny\mathrm{max}}$, and it contains all the globally hyperbolic space-times
where $S$ is realized as a Cauchy surface.
%The space-time  $M_{\tiny\mathrm{max}}$ is maximal in the sense that any immersion of $M_{\tiny\mathrm{max}}$ into an Einstein
%space-time $M'$ sending $S$ to a Cauchy surface of $M'$ is surjective.
This fact has a simple physical meaning: if the shape of the space and its first order variation is known at some time, one can
predict the whole geometry of the space-time.

So from this point of view, the problem of gravity can be reduced to a study of the embedding data of Cauchy surfaces.
Those are pairs of a metric $I$ and a shape operator $s$ which satisfy the so called constraint equations.
In dimension $3$, assuming the cosmological constant $0$,
 those equations correspond to the classical Gauss-Codazzi equations:
\[
   K_I=-\det s\,,\qquad \nabla_i s_{kj}-\nabla_j s_{ki}=0\,.
\]
However, in general it is not simple to understand whether two different embedding data $(I, s)$, $(I', s')$ correspond
to Cauchy surfaces of the same space-time. 

Assuming $S$ closed, Mess overcame this difficulty by a careful analysis of the holonomy of a globally hyperbolic 
space-time of constant curvature. In the flat case, he showed that the linear part of the holonomy 
is a discrete and faithful representation $\rho:\pi_1(S)\to \SO_0(2,1)$, providing an element $X_\rho=[\mathbb H^2/\rho]$ of
Teichm\"uller space $\Teich(S)$. On the other hand, the translation part is a cocycle $t\in H^1_\rho(\pi_1(S),\mathbb R^{2,1})$.
Using the $\SO_0(2,1)$-equivariant identification between $\mathbb R^{2,1}$ and $\mathfrak{so}(2,1)$ given by the Lorentzian 
cross product, $t$ can be directly regarded as an element of the cohomology group $H^1_{Ad\circ \rho}(\pi_1(S), \mathfrak{so}(2,1))$
which is canonically identified to the tangent space $T_{\rho}(\cR(\pi_1(S), \SO_0(2,1))/\!\!/\SO_0(2,1))$ of the character variety,   \cite{Goldman}.

By a celebrated result of Goldman (see for instance \cite{goldmanthesis})
  the map which gives the holonomy of the uniformized surface
\[
\Hol:\Teich(S)\to\cR(\pi_1(S), \SO_0(2,1))/\!\!/\SO_0(2,1)\,,
\]
is  a diffeomorphism of $\Teich(S)$ over a connected component of $\cR(\pi_1(S), \SO_0(2,1))/\!\!/\SO_0(2,1)$. Through this map we identify
$H^1_{Ad\circ \rho}(\pi_1(S), \mathfrak{so}(2,1))$ and $T_{X_\rho}\Teich(S)$, and consider  $t$ as a tangent vector of Teichm\"uller 
space.

Although Mess  pointed out a complete description of the moduli space of gravity, it remained somehow mysterious how to recover the
space-time from the embedding data of any Cauchy surface. In the Anti-de Sitter case Krasnov and Schlenker provided  some simple
formulae to get the two hyperbolic metrics associated to an Anti-de Sitter space-time $M$ 
in terms of the embedding data of any Cauchy surface $S$ of
$M$, \cite{Schlenker-Krasnov}. Those formulae work provided  the intrinsic curvature of the Cauchy surface is negative.
In the flat case Benedetti and Guadagnini \cite{begua}
studied the particular case of level set of the cosmological time, showing how they are directly related to Mess parameterization. 

 One of the aims of this paper is to  point out an explicit relation between the embedding data $(I, s)$ of any Cauchy surface $S$ in a
 flat $3$-dimensional space-time $M$ and the Mess parameters of $M$. 
 As in \cite{Schlenker-Krasnov}, 
 we will also   work under the assumption that $S$ inherits from $M$ a space-like metric $I$ of negative curvature.
 By the Gauss equation for space-like surfaces in Minkowski space, this is equivalent to the fact that the shape operator $s$ has positive determinant and corresponds to a local convexity of $S$. 

The reason why we assume the local convexity of $S$ is that it permits a  convenient \emph{change of variables}.
Instead of the pair $(I, s)$, one can in fact consider the pair $(h, b)$, where $h$ is the third fundamental form $h=I(s\bullet, s\bullet)$
and $b=s^{-1}$. The fact that $(I,s)$ solves Gauss-Codazzi equations
corresponds to the conditions that $h$ is a hyperbolic metric and $b$ is a self-adjoint solution
of Codazzi equation for $h$.

It is simple to check that the holonomy of the hyperbolic surface $(S,h)$ is the linear part of the holonomy of $M$, so the isotopy class of $h$
does not depend on the choice of a Cauchy surface in $M$ and corresponds to the element $X_\rho$ of Mess parameterization.
Recovering the translation part of the holonomy of $M$ in terms of $(h, b)$ is subtler.
This is based on the fact that  $b$ solves the Codazzi equation for the hyperbolic metric $h$.
Oliker and Simon in \cite{oliker} proved that any $h$-self-adjoint operator on the hyperbolic surface $(S, h)$ which solves the Codazzi equation
can be locally expressed as $\hess u- u\id$ for some smooth function $u$.
Using this result we  construct a short sequence of sheaves
\begin{equation}\label{eq:short}
  0\to\mathcal F\to\cun\to\mathcal C\to 0\,,
\end{equation}
where $\mathcal C$ is the sheaf of self-adjoint Codazzi operators on $S$ and
 $\mathcal F$ is the sheaf of flat sections of the $\mathbb R^{2,1}$-valued flat bundle associated to the holonomy of $h$.
Passing to cohomology, this gives a connecting homomorphism
\[
\delta: \mathcal C(S, h) \to H^1(S, \mathcal F)\,.
\]
It is  a standard fact that  $H^1(S, \mathcal F)$ is canonically identified with $H^1_{\tiny{\hol}}(\pi_1(S),\mathbb R^{2,1})$.
Under this identification we prove the following result.

\begin{reptheorem}{theorem holonomy}
%\label{thm:main1}
Let $M$ be a globally hyperbolic space-time and $S$ be a uniformly convex Cauchy surface with embedding data $(I, s)$.
Let $h$ be the third fundamental form of $S$ and $b=s^{-1}$.
Then
\begin{itemize}
\item the linear holonomy of $M$ coincides with the holonomy of $h$;
\item the translation part of the holonomy of $M$ coincides with $\delta b$.
\end{itemize}
\end{reptheorem}

It should be remarked that the construction of the short exact sequence \eqref{eq:short} and the proof of Theorem \ref{theorem holonomy}
are carried out just by local computations, so they hold for any uniformly convex space-like surface in any flat globally hyperbolic
 space-time without any assumption on the compactness or completeness of the surface.

In the case where $S$ is closed, we also provide a $2$-dimensional  geometric interpretation of $\delta b$.
This is based on the simple remark that $b$ can also be regarded as a first variation of the metric $h$.
As any Riemannian metric determines a complex structure over $S$, $b$ determines an infinitesimal variation
of the complex structure $X$ underlying the metric $h$, giving in this way an element $\Psi(b)\in T_{[X]}\Teich(S)$.

\begin{reptheorem}{theorem parametrization closed surfaces}
\label{thm:main2}
Let $h$ be a hyperbolic metric on a closed surface $S$, $X$ denote the complex structure underlying $h$ and $\cC(S, h)$ be the space of
self-adjoint $h$-Codazzi tensors. Then the following diagram is commutative 
\begin{equation}\label{eq:main3}
\begin{CD}
\cC(S, h) @>\Lambda\circ \delta>>H^1_{\tiny\Ad\circ\tiny\hol}(\pi_1(S),\so(2,1))\\
@V\Psi VV                                       @A d\Hol AA\\
T_{[X]}\cT(S) @>>\mathcal J > T_{[X]}\cT(S)
\end{CD}
\end{equation}
where $\Lambda:H^1_{\tiny\hol}(\pi_1(S), \R^{2,1})\to H^1_{\tiny\Ad\circ\tiny\hol}(\pi_1(S),\so(2,1))$ is the natural isomorphism, and
$\cJ$ is the complex structure on $\cT(S)$.
\end{reptheorem} 

As a consequence we get the following corollary

\begin{repcor}{corollary same spacetimes}
 Two embedding data $(I,s)$ and $(I', s')$ correspond to Cauchy surfaces contained in the same space-time
if and only if 
\begin{itemize}
\item the third fundamental forms $h$ and $h'$ are isotopic;
\item the infinitesimal variation of $h$ induced by $b$ is Teichm\"uller equivalent to the infinitesimal variation of $h'$ induced by $b'$.
\end{itemize}
\end{repcor}

The compactness of $S$ is important for the proof of Theorem \ref{theorem parametrization closed surfaces} for several reasons.
First, in the non closed case the uniformization is more problematic and one should add some extra hypothesis on $h$, like
completeness. 
A more important reason to restrict this theorem to the closed case 
is that the proof highly depends on a decomposition of 
$\mathcal C(S,h)$ as the direct sum of the space of Codazzi tensors that can be
globally expressed as $\Hess (u)-u\id$  (trivial Codazzi tensors)
and the space of Codazzi tensors which are the real part of quadratic differential (harmonic Codazzi tensors).
This decomposition, proved by Oliker and Simon in \cite{oliker}, is peculiar of the closed case.

The key point to prove Theorem \ref{theorem parametrization closed surfaces} is to relate $\delta b$ to  the first-order  variation of the holonomy of the family of hyperbolic metrics
 $h_t$ obtained by uniformizing the metrics $h((\Id+tb)\bullet, (\Id+tb)\bullet)$.
 This computation can be made explicit in the case $b=b_q$ is a harmonic Codazzi tensor. 
 In fact in this case $b_q$ is the first variation of a family of hyperbolic metrics, so the uniformization is not strictly necessary.
 It turns out that the variation of the holonomy for $\Psi(b_q)$ is $\Lambda\delta(b_{iq})$.
 
However in general $b$ is not tangent to a deformation of $h$ through hyperbolic metrics, so the uniformization necessarily enters in the proof.
To get rid of the conformal factor one uses the decomposition $b=b_q+\hess u-u\Id$. Heuristically $u\Id$ is a conformal variation, whereas $\hess u$ is trivial
in the sense that correspond to the Lie derivative of the metric through the gradient field $\grad u$. So one has $\Psi(b)=\Psi(b_q)$ and the commutativity of the diagram
\eqref{eq:main3} follows by the computation on harmonic differentials.

We give an application of the commutativity of the diagram \eqref{eq:main3} to hyperbolic geometry.
Goldman proved in \cite{Goldman} that the Weil-Petersson symplectic form on $\cT(S)$ coincides up to
a factor with the Goldman pairing on the character variety through the map $\Hol$.
We give a  new \emph{Lorentzian proof}  of this fact. It 
directly follows  by the commutativity of \eqref{eq:main3}: 
 we show   by
an explicit   computation that the pull-back of those forms through the maps $\Lambda\circ\delta$ and $\Psi$ coincide (up to a factor) on $\cC(S, h)$. 

The computation of the Weil-Petersson metric is quite similar to that obtained by Fischer and Tromba \cite{trombawp} and the result is completely analogous.
The computation of the Goldman pairing follows in a  simple way using  
 a different characterization of self-adjoint Codazzi tensors.
The inclusion of $\Hyp^2\to\R^{2,1}$ projects to a section $\iota$ of the flat $\R^{2,1}$-bundle $F$ associated with $\hol:\pi_1(S)\rar \SO_0(2,1)$.
The differential of this map provides an inclusion $\iota_*$ of $TS$ into $F$
corresponding to the standard inclusion of $T\Hyp^2$ into $\R^{2,1}$.
Thus any operator $b$ on $TS$  corresponds to an $F$-valued one-form $\iota_*b$. We prove that
$b$ is Codazzi and self-adjoint for $h$ if and only if the form $\iota_*b$ is closed.
From this point of view the connecting homomorphism $\delta:\cC(S, h)\to H^1(S, \cF)$ associated to the short exact sequence in \eqref{eq:short} can be expressed as
$\delta(b)=[\iota_*b]$, where we are implicitly using the canonical identification between $H^1(S, \cF)$ and the de Rham
cohomology group $H^1_{dR}(S, F)$.
The fact that the Goldman pairing 
coincides with the cup product in the de Rham cohomology proves immediately the coincidence of the two forms.

While Goldman's proof highly relies on the complex analytical theory of Teichm\"uller space, our proof is basically only differential geometric.

\subsection*{Space-times with particles}

In the second part of the paper we apply this machinery  to study globally hyperbolic space-times containing particles.
Mathematically speaking particles are cone singularities along time-like lines with angle in $(0,2\pi)$.
In order to develop a reasonable study of Cauchy surfaces in a space-time with particles, some assumption are needed
about the behavior of the surface around a particle. Here the  assumption we consider is very weak: 
 we only assume that the shape operator of the surface is bounded and uniformly  positive 
(meaning that the principal curvatures are uniformly far from $0$ and $+\infty$). We will briefly say that 
the Cauchy surface is bounded and uniformly convex.

Under this assumption we prove that the surface is necessarily orthogonal to the singular locus and intrinsically carries a Riemannian metric
with cone angles equal to the cone singularities of the particle (here we use the definition given by Troyanov \cite{Troyanov}
of metrics with cone angles on a surface with variable curvature).
It turns out that the third fundamental form of such a surface is a hyperbolic surface with the same cone angles
and $b=s^{-1}$ is a bounded and uniformly positive Codazzi operator for $(S,h)$.
More precisely we prove the following statement.

\begin{reptheorem}{corrispondenza coppie caso singolare}
%\label{thm:main3}
Let us fix a \emph{divisor} $\bfs{\beta}=\sum\beta_ip_i$ on a surface with $\beta_i\in(-1,0)$ and consider the following sets:
\begin{itemize}
\item $\E_{\bfs{\beta}}$ is the set of embedding data $(I,s)$ of bounded and uniformly convex 
Cauchy surfaces on flat space-times with particles so
that for each $i=1,\ldots , k$ a particle of angle $2\pi(1+\beta_i)$ passes through $p_i$.
\item $\D_{\bfs{\beta}}$ is the set of pairs $(h,b)$, where $h$ is a hyperbolic metric on $S$ with a cone singularity
of angle $2\pi(1+\beta_i)$ at each $p_i$ and $b$ is a self-adjoint solution of Codazzi equation for $h$, bounded and  uniformly positive.
\end{itemize}
Then the correspondence $(I,s)\to (h,b)$ induces a bijection between $\E_{\bfs{\beta}}$ and $\D_{\bfs{\beta}}$.
\end{reptheorem}

\begin{remark*}
By Gauss-Bonnet formula, in order to have $\D_{\bfs{\beta}}$ (and consequently $\E_{\bfs{\beta}}$)  non empty one has to require that
$\chi(S, \bfs\beta):=\chi(S)+\sum\beta_i$ is negative. We will always make this assumption in the paper. 
\end{remark*}

The difficult part of the proof  is to show that for $(h,b)\in\E_{\bfs{\beta}}$, the corresponding pair
$(I,s)$ is the embedding data of some Cauchy surface in a  flat space-time with particles.
Clearly  the regular part of $S$ can be realized as a Cauchy surface in some flat space-time $M$; it remains to prove that
$M$ can be embedded in a space-time with particles.
The problem is local around the punctures.
Indeed it is sufficient to prove that a neighborhood of a puncture $p_i$ can be realized as a surface in a small flat cylinder
with particle and then use some standard \emph{cut and paste} procedure to construct the thickening of $M$.

The first point is to prove that the translation part of the holonomy of $M$ is trivial for peripheral loops, preventing
singularities with spin as those studied in \cite{barbot-meusburger}. This is achieved by showing that
 $b$ can be expressed as $\Hess u-u\id$ in a neighborhood of a puncture. In fact we prove that this fact holds 
  for the larger class of  Codazzi operators defined on the 
regular part of $S$ whose squared norm is integrable  with respect to the area form of $h$. 

The construction of the small cylinders where neighborhoods of particles in $S$ can be embedded is somehow technical.
Here we try to sketch  the main ingredients we use.
A key point is that if $U$ is a neighborhood of a puncture $p$, then it carries a natural Euclidean metric with the 
same cone angle: such a metric is obtained by composing the developing map of $U$ with the radial projection of
the hyperboloid $\mathbb H^2$ onto  the affine plane tangent in $\R^{2,1}$ at the fixed point of the holonomy of $U$.

Now express $b$ around the cone point as $\hess u-u\id$ and consider the function $\bar u=(\ch r)^{-1} u$, where $r$ is the distance from
the singular point.
In \cite{bonfill} it was proved that   the developing map corresponding to the embedding data $(I, s)$ is  a  map of the form
\begin{equation}\label{eq:sigma}
        \sigma(x)= D_*(\grad \bar u(x))+ \bar u(x) e_0
\end{equation}
where $D:\widetilde{U\setminus\{p\}}\to\R^2$ is the developing map of the Euclidean metric over $U$ and $e_0$ is a unit vector fixed
by the holonomy of the peripheral loop around $p$. We are implicitly using the orthogonal splitting of  Minkowski space 
$\R^{2,1}=\R^2\oplus\mathrm{Span}(e_0)$.

We then prove that the map  $\varphi:\widetilde{U\setminus\{p\}}\to\R^2$, defined by $\varphi(x)=D_*(\grad \bar u(x))$ is the developing
map of a Euclidean structure with a cone point on $U$. This fact allows to construct a flat cylinder with a particle as
the orthogonal product of the flat Euclidean structure over $U$ and the standard negative metric over $\R$. Formula
\eqref{eq:sigma} shows that the graph map $x\to(x, \bar u(x))$ is the isometric embedding of $(U, I, s)$ in this cylinder.

The fact that $\varphi$ is a developing map for  a Euclidean structure with a cone point   
relies on 
the fact that the Euclidean Hessian of $\bar u$ is bounded and uniformly  positive
over $U$, as it is directly related to $b$ by a result of \cite{bonfill}.

%\begin{repprop}{lemma metrica euclidea}
%\label{pr:main}
%Let $(U,g)$ be a Euclidean disc with a cone point  at $p\in U$.
%Let $f$ be a function over $U$ such that $\hess_g(f)$ is bounded and uniformly positive.
%Then the metric
%\[
   %g'=g(\hess_g(f)\bullet, \hess_g(f)\bullet)
%\]
%is a Euclidean metric with the same cone angle at $p$.
%\end{repprop}
Moreover, the above construction allows to prove the following result concerning Riemannian metrics with cone points, which might have an interest on its own:

\begin{reptheorem}{hbb cone metric}
%\label{amazzonia}
Let $h$ be a hyperbolic metric with cone singularities and let $b$ be a Codazzi, self-adjoint operator for $h$, bounded and uniformly positive. Then $I=h(b\bullet,b\bullet)$ defines a singular metric with the same cone angles as $h$.
\end{reptheorem}

Another goal of this part is to 
show  the analogue of Theorem \ref{theorem parametrization closed surfaces}  in the context of cone singularities, proving 
that the relevant moduli space is the tangent bundle of the Teichm\"uller space of the punctured surface, and  in particular
it is  independent of the cone angles.

To give a precise statement we use the Troyanov uniformization result \cite{Troyanov} which ensures that, given a
conformal structure on $S$, there is a unique conformal hyperbolic metric with prescribed cone angles at the points $p_i$ (notice we are assuming $\chi(S, \bfs\beta)<0$).
So once the divisor $\bfs{\beta}$ is chosen we have a holonomy map
\[
   \Hol:\cT(S,\puct)\to\cR(\pi_1(S\setminus\puct), \SO_0(2,1))/\!\!/\SO_0(2,1)\,,
 \]
where $\puct=\{p_1,\ldots, p_k\}$ is the support of $\bfs{\beta}$,
and $\cT(S,\puct)$ is the Teichm\"uller space of the punctured surface.

As in the closed case fix a hyperbolic metric $h$  on $S$ with cone angles  $2\pi(1+\beta_i)$ at $p_i$.
Let $X$ denote the complex structure underlying $h$.
 Any  Codazzi operator $b$ on $(S, h)$ can be regarded as
an infinitesimal deformation of the metric on the regular part of $S$. If $b$ is bounded this deformation
is quasi-conformal so it extends to an infinitesimal deformation of the underlying conformal structure at the punctures,
providing an element $\Psi(b)$ in $T_{[X]}\cT(S,\puct)$.

\begin{reptheorem}{thm:coneteich}
%\label{thm:main4}
Let $\cC_\infty(S, h)$ be the space of bounded Codazzi tensors on $(S, h)$.
The following diagram is commutative 
\begin{equation}\label{eq:main4}
\begin{CD}
\cC_\infty(S, h) @>\Lambda\circ \delta>>H^1_{\tiny\Ad\circ\tiny\hol}(\pi_1(S\setminus\puct),\so(2,1))\\
@V\Psi VV                                       @A d\Hol AA\\
T_{[X]}\cT(S,\puct) @>>\mathcal J > T_{[X]}\cT(S,\puct)
\end{CD}\,,
\end{equation}
where $\Lambda:H^1_{\tiny\hol}(\pi_1(S\setminus\puct), \R^{2,1})\to H^1_{\tiny\Ad\circ\tiny\hol}(\pi_1(S\setminus\puct),\so(2,1))$ 
is the natural isomorphism, and $\cJ$ is the complex structure on $\cT(S,\puct)$.
\end{reptheorem} 

In order to repeat the argument used in the closed case, we show that also in this context bounded Codazzi tensors can be
split as the sum of a trivial part and a harmonic part.
More precisely we prove that  any square-integrable Codazzi tensor on a surfaces with cone angles in $(0,2\pi)$
can be expressed as the sum of a trivial Codazzi tensor and a Codazzi 
tensor corresponding to a holomorphic quadratic differential with at worst simple poles at the punctures.
As a consequence we have the following corollary.

\begin{repcor}{cor:main2}
Two embedding data $(I,s)$ and $(I', s')$ in $\E_{\bfs{\beta}}$
 correspond to Cauchy surfaces contained in the same space-time with particles
if and only if 
\begin{itemize}
\item the third fundamental forms $h$ and $h'$ are isotopic;
\item the infinitesimal variation of $h$ induced by $b$ is Teichm\"uller equivalent to the infinitesimal variation of $h'$ induced by $b'$.
\end{itemize}
\end{repcor}

It should be remarked that in this context, at least if the cone angles are in $[\pi,2\pi)$, the holonomy does not distinguish the structures,
so this corollary is not a direct consequence of Theorem \ref{thm:coneteich}, but some argument is required.
Indeed we believe that for the same reason, the direct application of Mess' arguments to this context is not immediate.
It can be remarked that in Anti-de Sitter case in \cite{bbsads} a generalization of Mess' techniques has been achieved, at least if cone angles are in $(0, \pi)$.
In that context the existence of a convex core can be proved by some general arguments which work also in the case with particles
(with small angles). On the other hand  in the flat case it seems more difficult to adapt Mess' argument.
However, if cone angles are less than $\pi$, the arguments used in this paper indicate that one can recover the analogous picture of the closed
case, showing that the initial singularity of a maximal globally hyperbolic flat space-time  with particles 
containing a uniformly convex surface  is a real tree. We will not develop this point in this work as it seems not strictly related to the focus
of the paper.

It remains open to understand the extent  to which  the condition of containing a uniformly convex surface is restrictive.
It is not difficult to point out some counterexamples: it is sufficient to double a cylinder in Minkowski space based on some polygon on $\R^2$.
However the space-times obtained in this way have the property that Euler characteristic $\chi(S, \bfs\beta)$ of its Cauchy surfaces is $0$.
%to have Euler characteristic equal to $0$, where the Euler characteristic 
%of  a space-time based on $S$ with divisor $\bfs{\beta}$ is defined as $\chi(S,\bfs\beta)=\chi(S)+ \sum\beta_i$.
%On the other hand, if a space-time contains a uniformly convex surface, applying Gauss-Bonnet formula to $(S, h)$ we see that
%$\chi(S,\bfs\beta)<0$.

So a natural question is whether there are globally hyperbolic space-times with particles with negative characteristic
which do not contain uniformly convex surfaces.
In the last section we give some counterexamples in this direction, based on some simple surgery idea.
Similar problems regarding the existence of space-times with certain properties on the Cauchy surfaces have been tackled in \cite{beguacone}.
%Those examples are similar to the ones pointed out in \cite{beguacone}.

We remark that in all those exotic examples at least one particle must have cone angle in $[\pi,2\pi)$.
Indeed we believe that if cone angles are less than $\pi$ the classification given in this paper is complete, that is
maximal globally hyperbolic flat space-times with small cone angles must contain a uniformly convex surface.
However we leave this question for a further investigation.

 We finally address the question of the coincidence of the Weil-Petersson metric and the Goldman pairing in this context
 of structures with cone singularities.
 Once a divisor $\bfs\beta$ is fixed, the hyperbolic metrics with prescribed cone angles   allow to determine
 a Weil-Petersson product on $\cT(S, \puct)$, as it has been studied in \cite{schumacher}.
 In \cite{mondelloPoisson},
 Mondello showed 
 that also in this singular case the Weil-Petersson product corresponds
 to an intersection form on the subspace of $H^1_{\tiny\Ad\circ\tiny\hol}(\pi_1(S\setminus\puct), \so(2,1))$ corresponding to cocycles trivial around the punctures.
Actually  Mondello's proof is based on a careful generalization of Goldman argument in the case with singularity.
Like  in the closed case, we  give a  substantially different proof of this coincidence by
using  the commutativity of \eqref{eq:main4}.

\subsection*{Outline of the paper}
In Section \ref{sec:cauco} we study Codazzi tensors on a hyperbolic surface $(S, h)$.
In this section  $S$ is not assumed to be closed. The only assumption we will make is that
the holonomy of $S$ is not elementary.
We construct the short exact sequence \eqref{eq:short}
%Using results of \cite{oliker} we will construct the short exact sequence of sheaves
%\[
   %0\to\mathcal F\to\cun\to\cC\to 0
%\]
%where $\cF$ is the sheaf of flat sections of the $\R^{2,1}$-valued flat bundle $F$ associated with the holonomy of $S$, and
%$\cC$ is the sheaf of self-adjoint Codazzi tensors over $S$.
and then we prove Theorem \ref{theorem holonomy}.
In the last part of this section we prove the characterization of self-adjoint Codazzi tensors in terms of closed forms and the relation to de Rham cohomology.
%The inclusion of $\Hyp^2\to\R^{2,1}$ projects to a section $\iota:S\to F$.
%The differential of this map provides an inclusion $\iota_*$ of $TS$ into $F$
%corresponding to the standard inclusion of $T\Hyp^2$ into $\R^{2,1}$.
%Thus any operator $b$ on $TS$  corresponds to an $F$-valued one-form $\iota_*b$. We prove that
%$b$ is Codazzi and self-adjoint for $h$ if and only if the form $\iota_*b$ is closed.
%From this point of view the connecting homomorphism $\delta:\cC(h)\to H^1(\cF)$ can be expressed as
%$\delta(b)=[\iota_*b]$, where we are implicitly using the canonical identification between $H^1(\cF)$ and the de Rham
%cohomology group $H^1_{dR}(S, F)$.

Section \ref{closed surfaces} is mostly devoted to the proof of Theorem \ref{theorem parametrization closed surfaces}. The first part of the section shows the relationship between the variation of a family of hyperbolic metrics and the variation of the corresponing holonomies. Then we conclude the proof, by showing the decisive equality $d\Hol(\Psi(b_q))=\delta(b_{iq})$, where $b_q$ is the harmonic Codazzi tensor associated with the quadratic differential $q$.
In the final part of the section we perform the computations of the Weil-Petersson and Goldman symplectic forms and obtain Goldman theorem as a corollary.
%Finally we compute the pull-back of the Weil-Petersson pairing and the Goldman pairing through the maps $\Psi$ and $\Lambda\circ\delta$ and notice that they coincide up to
%a factor. The computation of the Weil-Petersson metric is quite similar to the one get by Fischer and Tromba \cite{trombawp} and the result is completely analogous.
%The computation of the Goldman pairing is indeed very simple using the expression of the connecting homomorphism $\delta b=[\iota_* b]$: the fact that the Goldman pairing 
%coincides with the cup product in the de-Rham cohomology setting makes this computation immediate.

Section \ref{punctured surfaces} develops the theory for hyperbolic surfaces with cone angles and flat space-times with particles.
In the first part of this section we study Codazzi tensors on hyperbolic surfaces with cone angles.
Here there are two remarkable results: first we prove the decomposition of $L^2$-Codazzi tensors as the sum of a trivial tensor
and a harmonic tensor corresponding to a quadratic differential with at worst simple poles at punctures. Second, we prove that
$L^2$-Codazzi tensors are trivial around punctures. 
Basically results of this part are achieved by a local analysis of the Codazzi tensors around the punctures.

In the second part of Section \ref{punctured surfaces}  we apply  the theory we have developed about Codazzi tensors on hyperbolic singular surfaces  
to the Lorentzian setting. Here  we prove Theorem \ref{corrispondenza coppie caso singolare}. 
%Actually, first we prove Proposition \ref{lemma metrica euclidea} and then we use this proposition to 
%get the result.
Again the arguments of this section are of local nature around punctures. As a byproduct we prove Theorem \ref{hbb cone metric}.

In Section \ref{sectionteichsing} we finally prove Theorem \ref{thm:coneteich} and Corollary \ref{cor:main2}.
We use the same ingredients as in the closed case. There is some more technical details to get rid of
conformal factors. Once the commutativity of the diagram \eqref{eq:main4} is achieved, the computation of the Weil-Petersson metric and 
the Goldman pairing is done exactly as in the closed case.

Finally in Section \ref{exotic} we construct an example of a flat globally hyperbolic space-time  with particles whose  Euler characteristic is negative,
 but such that it does not admit any uniformly convex Cauchy surface.

\subsection*{Acknowledgements}
The first author is grateful to Roberto Frigerio and Gabriele Mondello for useful conversations about
the geometry of hyperbolic surfaces with cone angles bigger than $\pi$. The second author would like to thank Thierry Barbot and Francois Fillastre for several valuable discussions.

%----------------------------------------------sezione 2---------------------------------

%\input bs-sec2.tex

\section{Cauchy surfaces and Codazzi operators}\label{sec:cauco}

\subsection{Codazzi operators on a hyperbolic surface} \label{sezione cociclo R21}

We denote with $\R^{2,1}$ Minkowski space, namely the vector space $\R^3$ endowed with the bilinear form 
 $$\langle v,w\rangle=v_1 w_1+v_2 w_2-v_3 w_3\,.$$
In this paper we will implicitly consider $\Hyp^2$ as the hyperboloid model in $\R^{2,1}$, that is,
$$\Hyp^2=\left\{v\in\R^{2,1}\mid \langle v,v\rangle=-1,\ v_3>0 \right\}\,.$$
Recall that $\isom^+(\R^{2,1})=\SO(2,1)\ltimes \R^{2,1}$ is the group of orientation-preserving isometries of Minkowski space, 
and $\isom^+(\Hyp^2)=\SO_0(2,1)$ is the identity component of $\SO(2,1)$, 
containing orientation-preserving orthochronus  linear isometries of $\R^{2,1}$.

Let  $(S,h)$ be any hyperbolic (possibly open and non-complete)  surface.
Denote by $$\hol:\pi_1(S)\rightarrow\SO_0(2,1)$$  the corresponding holonomy.
The only assumption  we will make on $h$ is that  $\hol$ is not elementary.

%In this section, we will identify the universal cover of the closed surface $S$ with $\Hyp^2$, therefore the action of $\pi_1(S)$ is identified with the action of its image under the holonomy homomorphism. We will need the following property.

We will consider the Codazzi operator on the space of linear maps on $TS$
\[
d^\nabla:\Gamma(T^*S\otimes TS)\rightarrow\Gamma(\Lambda^2T^*S\otimes TS)
\]
defined in this way. Given a linear map $b:TS\rightarrow TS$, 
and given $v_1,v_2\in T_xM$  we have
\[
   d^\nabla b(v_1,v_2)=(\nabla_{v_1} b)(v_2)-(\nabla_{v_2} b)(v_1)=\nabla_{v_1}(b(\hat v_2))-\nabla_{v_2}(b(\hat v_1))-b([\hat v_1, \hat v_2])\,,
\]
where $\nabla$ is the Levi Civita connection of $h$ and
 $\hat v_1$ and $\hat v_2$ are local extensions of $v_1$ and $v_2$ in a neighborhood of $x$.

Given a representation $\rho:\pi_1(S)\rightarrow\SO_0(2,1)$, the first cohomology group $H^1_{\rho}(\pi_1(S), \R^{2,1})$ is the vector space obtained as a quotient of cocycles over coboundaries. A cocycle is a map $t:\pi_1(S)\rar\R^{2,1}$ satisfying $t_{\alpha\beta}=\rho(\alpha)t_\beta+t_\alpha$. Such a cocycle $t$ is a coboundary if $t_\alpha=\rho(\alpha)t-t$ for some $t\in\R^{2,1}$.
In this section we will show that self-adjoint operators $b$ satisfying the Codazzi equation $d^\nabla b=0$
naturally describe the elements of the cohomology group $H^1_{\tiny{\hol}}(\pi_1(S), \R^{2,1})$. 

Let us recall that associated with $h$
there is a natural flat $\R^{2,1}$-bundle $F\to S$, whose holonomy is $\hol$.
Basically, $F$ is the quotient of $\tilde S\times\R^{2,1}$  by the product action of $\pi_1(S)$ as deck transformation
on the first component and through the representation $\hol$ on the second one.
Since $\tilde S$ is contractible, the group $H^1_{\tiny{\hol}}(\pi_1(S), \R^{2,1})$ can be canonically identified with the
first cohomology group $H^1(S,\cF)$ of the sheaf $\cF$ of  flat sections of  $F$.

The relation between Codazzi tensors and the cohomology group  $H^1_{\tiny{\hol}}(\pi_1(S), \R^{2,1})=H^1(S,\cF)$ 
relies on the construction of a short exact sequence of sheaves
\[
   0\to\cF\to\cun\to\cC\to 0 \, ,
\]
where $\cC$ is the sheaf of self-adjoint Codazzi tensors on $S$.

First we construct a map $H:\cun\to\cC$. This is is based on the following simple remark:

\begin{lemma}
Let $U$ be any hyperbolic surface.
For every $u\in\cun(S)$, $b=\Hess u-u\id$ is a self-adjoint Codazzi operator with respect to $h$. Here $\id$ denotes the identity operator and $\Hess u=\nabla \grad u$ is considered as an operator on $TS$.
\end{lemma}
\begin{proof}
The fact that $b$ is self-adjoint is clear.
Let us prove that $b$ satisfies the Codazzi equation.
By a simple computation $d^\nabla(u\id)=du\wedge\id$. On the other hand, since by definition $\Hess u=\nabla (\grad u)$ we get
$d^\nabla\Hess u=R(\cdot, \cdot)\grad u$. 
As for a hyperbolic surface $R(v_1,v_2)v_3=h(v_3,v_1)v_2-h(v_3,v_2)v_1$, 
we  get $d^\nabla\Hess u=du\wedge\id$, so the result
follows.
\end{proof}

Hence we define $$H(u)=\Hess u-u\id\,.$$
The second step is to construct a map $V:\cF\to\cun$ whose image is the kernel of $H$.
Notice that the developing map $\dev:\tilde S\rightarrow\Hh^2\subset\R^{2,1}$ induces to the quotient a section
$\bfs{\iota}:S\rightarrow F$ called the developing section.
Using the natural Minkowski product on $F$, for any section $\sigma$ of $F$  the smooth function $V(\sigma)$ is defined by 
taking the product of $\sigma$ with $\bfs{\iota}$:
\[
  V(\sigma)=\langle \sigma, \bfs{\iota}\rangle~.
\]

\begin{theorem}
The short sequence of sheaves
\begin{equation}\label{eq:exsq}
\begin{CD}
 0@>>>\cF@>V>>\cun @>H>>\cC @>>>0
\end{CD}
\end{equation}
 is exact.
\end{theorem}

Since  the statement is of local nature, it suffices to check exactness  on an open convex subset $U$ of $\Hh^2$.
The surjectivity of the map $H:\cun(U)\rightarrow\cC(U,h)$ follows by the general results  in \cite{oliker}.

Notice that  $\cF(U)$ is naturally identified with $\R^{2,1}$. On the other hand, by identifying $\Hh^2$ with a subset of 
$\R^{2,1}$, the developing section is the standard inclusion. So the exactness of the first part of the sequence
(\ref{eq:exsq}) is proved by the following Proposition.

\begin{prop} \label{existence of function}
Let $U$ be a convex neighborhood of $\Hh^2$. 
For any vector $t_0\in\R^{2,1}$ the corresponding function $v=V(t_0)$
\[
    v(x)=\langle t_0, x\rangle
\]
satisfies the equation $H(v)=0$.
Conversely, 
if $u$ is a smooth function on $U$ such that $H(u)=0$, there exists a unique vector
$t\in\R^{2,1}$ such that $u(x)=\langle t, x\rangle$ for any $x\in U$.
\end{prop}

\begin{proof}
We start by showing that for any fixed $t_0\in\R^{2,1}$ the function $v(x)=\langle t_0,x\rangle$ satisfies $H(v)=0$. 
Note that, for $w\in T_x\Hyp^2$, 
\begin{equation} \label{relazione gradiente 1}
d v_x(w)=\langle t_0,w\rangle=\langle t_0^{T_x},w\rangle
\end{equation} 
where $t_0^{T_x}$ is the projection of $t_0$ to $T_x\Hyp^2$. 
Thus $
\grad v(x)=t_0^{T_x}$.
Note that $T_x\Hyp^2$ coincides with the orthogonal plane to $x$. So
$t_0=t_0^{T_x}-\langle t_0,x\rangle x=\grad v(x)-v(x)x$ %t_0^{T_x}-\langle t_0,x\rangle x$
and, upon covariant differentiation in $\R^{2,1}$ (we will denote by $\ol\nabla$ the covariant derivative in the ambient $\R^{2,1}$ and by $\nabla$ that of $\Hyp^2$ and use the fact that the second fundamental form of $\Hyp^2$ coincides with the metric),
$$0=\nabla_w\grad v(x)+\langle w,\grad v(x)\rangle x-d v_x(w)x- v(x)\,w=(\Hess v- v\id)(w)\,.$$

Since elements of $U$, regarded as vectors of $\R^{2,1}$, generate the whole Minkowski space, the map
\[
    V:\R^{2,1}\rightarrow \cun(U)
\]
is injective and the image is a subspace of dimension $3$ contained in the kernel of $H$.
In order to conclude it is sufficient to prove that the dimension of $\ker H$ is $3$.

To this aim  it will suffice to show that any $u$ satisfying $\Hess u-u\id =0$ such that $u(x_i)=0$ on three non-collinear points $x_1,x_2,x_3$ vanishes everywhere. 

Let $\gamma:\R\rar\Hyp^2$ be a unit-speed geodesic in $U$ connecting two points $\gamma(s_1),\gamma(s_2)$ where $u(\gamma(s_1))=u(\gamma(s_2))=0$. We claim that $u\circ\gamma\equiv 0$. Using that $\Hess u-u\id=0$, one gets that $y=u\circ \gamma$ satisfies the linear differential equation $y''=y$. Since $y(s_1)=y(s_2)=0$, for a standard maximum argument, $y\equiv 0$ on the interval $[s_1,s_2]$ and, by uniqueness, $y\equiv 0$ on $\R$.

Then $u\equiv 0$ on any geodesic connecting two points where $u$ takes the value $0$. By hypothesis, $u$ takes the value $0$ on three non-collinear points of $U$. By convexity of $U$, it is easy to see that the geodesics on which $u\equiv 0$ exaust the whole $U$, and this concludes the proof.
\end{proof}
 
% \begin{figure}[htbp]
%\centering
%\begin{minipage}[c]{.40\textwidth}
%\centering
%\includegraphics[height=4cm]{geodesics1.jpg}
%\caption{If $u$ vanishes  on three non-collinear points, then it vanishes on the three geodesics connecting them.} \label{geodesics1}
%\end{minipage}%
%\hspace{5mm}
%\begin{minipage}[c]{.40\textwidth}
%\centering
%\includegraphics[height=4cm]{geodesics2.jpg}
%\caption{It is then easy to see that any other point of $\Hyp^2$ lies on a geodesic connecting two points where $u=0$, and thus $u\equiv 0$} \label{geodesics2}
%\end{minipage}
%\end{figure}

Let us stress that in general  the sequence (\ref{eq:exsq}) is not globally exact.
The following example shows a family of Codazzi tensors which cannot be expressed as $\Hess(u)-u\id$.

\begin{example} \label{Codazzi traceless real part hqd}
Suppose $S$ is a closed surface. 
As observed by Hopf, see also \cite{Schlenker-Krasnov}, a self-adjoint Codazzi  operator $b:TS\rar TS$ is traceless 
 if and only if the symmetric form $g(v,w)=h(b(v),w)$ on $S$ is the real part of a holomorphic quadratic differential $q$ on $S$. 
 This gives an isomorphism of real vector spaces between the space of holomorphic quadratic differentials on $S$ and the space of traceless Codazzi tensors. We denote the image of $q$ under this isomorphism by $b_q$. So traceless Codazzi tensors form a vector space of finite dimension $6g-6$ where $g$ is the genus of $S$.

On the other hand, if $\Hess u-u\id$ is traceless, then $u$ satisfies the equation $\Delta u-2u=0$. A simple application of the maximum principle shows that the only solution of that equation is $u\equiv 0$. It follows that non-trivial traceless Codazzi tensors on $S$ cannot 
be expressed as $\Hess u-u\id$.
\end{example}

The next Proposition shows that however the examples above are in a sense the most general possible.
Although the proof is contained in \cite{oliker}, we give a short argument.

\begin{prop}\label{pr:cod-dec}
Let $S$ be a closed surface. Given 
  $b\in\cC(S,h)$ self-adjoint tensor satisfying Codazzi equation with respect to the hyperbolic metric $h$, a holomorphic quadratic differential $q$ and a smooth function $u\in\cun(S)$ are uniquely determined so that $b=b_q+\Hess u-u\id$.
\end{prop}
\begin{proof}
 The subspaces $\left\{b\in\cC(S,h): b \textnormal{ is traceless}\right\}$ and $\left\{b\in\cC(S,h): b=\Hess u-u\id \right\}$ have trivial intersection by  Example \ref{Codazzi traceless real part hqd}.
% First let us show that this decomposition is unique. Indeed, suppose $b$ is traceless and $\tr b=\tr(u\id-\Hess u)=2u-\Delta u=0$. Since the %operator $2\mathrm{id}-\Delta$ is invertible on the compact surface $S$, then $u=0$. 
%Compare also Lemma \ref{lemma decomposizione laplaciano singularities}.
Now, let $b\in\cC(S, h)$ and $f=\tr (b)\in\cun(S)$. Again, since $\Delta-2\mathrm{id}$ is invertible, there exists some function $u$ such that $f=\Delta u-2u=\tr(\Hess u-u\id)$. Therefore $b-(\Hess u-u\id)$ is traceless. This concludes the proof of the direct sum decomposition.
\end{proof}

From the exact sequence (\ref{eq:exsq}) we have a long exact sequence in cohomology
\begin{equation} \label{long exact sequence}
   0\rightarrow H^0(S,\cF)\rightarrow H^0(S, \cun)\rightarrow H^0(S, \cC)\rightarrow H^1(S, \cF)\rightarrow H^1(S, \cun)~.
\end{equation}
Since the holonomy representation $\hol$ is irreducible, then $H^0(S,\cF)$ is trivial. Moreover, $H^1(S, \cun)$
vanishes since $\cun $ is a fine sheaf.
So we have a short exact sequence
\begin{equation} \label{coboundary operator}
\begin{CD}      0@>>> \cun(S)@>H>> \cC(S,h)@>\delta>> H^1(S, \cF)@>>>0~.
\end{CD}
\end{equation}

Since $\tilde S$ is contractible,
the cohomology group $H^1(S, \cF)$ is naturally identified with the group $H^1_{\tiny\hol}(\pi_1(S), \R^{2,1})$.
The identification goes as follows. Take a good cover $\cU$ of $S$ and let $\tilde \cU$ be its lifting on $\tilde S$.
By Leray Theorem  $H^1(S,\cF)=H^1(\check{C}^\bullet(\cU, \cF))$.

The pull-back $\pi^*\cF$ of the sheaf $\cF$ on the universal cover  is isomorphic to the sheaf $\underline{\R}^{2,1}$ of $\R^{2,1}$-valued locally constant function.
Moreover there is a natural left action of $\pi_1(S)$ on $\check C^k(\tilde \cU,  \underline{\R}^{2,1})$ given by
\[
      (\alpha\star s)(i_0,\ldots, i_k)=\hol(\alpha) s (\alpha^{-1}i_0,\ldots, \alpha^{-1}i_k)\,,
\]
where we are using the fact that $\pi_1(S)$ permutes the open subsets in $\tilde \cU$.

Now, the complex $\check{C}^\bullet(\cU,\cF)$ is identified by pull-back with the sub-complex of $\check{C}^\bullet(\tilde\cU, \R^{2,1})$,
say $\check{C}^\bullet(\tilde\cU, \R^{2,1})^{\pi_1(S)}$, made of elements invariant by the action of $\pi_1(S)$.

Since $H^1(\tilde S, \R^{2,1})=0$, 
given a $\pi_1(S)$-invariant cocycle $s\in \check{Z}^\bullet(\tilde\cU, \R^{2,1})^{\pi_1(S)}$, there is
a $0$-cochain $r\in \check{C}^0(\tilde \cU, \R^{2,1})$ such that $\check{d}(r)=s$, that is
$s(i_0, i_1)=r(i_1)-r(i_0)$ for any pair of open sets in $\tilde \cU$ which have nonempty intersection. 

Although in general $r$ is not $\pi_1(S)$-invariant, 
for any $\alpha\in\pi_1(S)$ we have that $\check d (\alpha\star r-r)=0$, so there is an element $t_\alpha\in\R^{2,1}$ such that
$\hol(\alpha)r(\alpha^{-1}i_0)-r(i_0)=t_\alpha$ for every $i_0$.

It turns out that the collection $(t_\alpha)$ verifies the cocycle condition, so it determines an element of
$H^1_{\tiny\hol}(\pi_1(S), \R^{2,1})$. This construction provides the required isomorphism
\[
    H^1(\check{C}^\bullet(\tilde\cU, \R^{2,1})^{\pi_1(S)})\to H^1_{\tiny \hol}(\pi_1(S), \R^{2,1})~.
\]

%The de Rham complex  of $E$-valued $k$-forms (where $D$ is the exterior differential associated to the flat connection on $E$)
%\[
%\begin{CD}
%  0@>>>\cF(E)@>D>>\Omega^0(E)@>D>>\Omega^1(E)>D>>\Omega^2(E)@>D>>\ldots
%\end{CD}
%\]
%is a fine resolution of the sheaf $\cF$. By the de Rham Abstract Theorem
%the cohomology $H^1(S, \cF)$ coincides with the de Rham cohomology group $H^1(S, E)$.
%On the other hand given any $E$-valued $D$-closed $1$-form $\omega$, 
%its lifting to the universal covering $\tilde\omega$ is exact, so there is a function $\xi:\Hh^2\rightarrow\R^{2,1}$ such that
%\[
%    \tilde\omega=d\xi
%\]
%(notice that we are using that on $\tilde S$ the fiber bundle is trivial as flat bundle).
%Notice that $\tilde\omega$ is equivariant by the action of the fundamental group in the sense that
%\[
 %     \gamma^*(\tilde\omega)= f(\gamma^{-1})\tilde\omega
%\]
%for any $\gamma\in\pi_1(S)$.

%Fix $\gamma\in\pi_1(S)$ and consider the $\R^{2,1}$ function on $H^1(S,E)$:
%\[
%     F_\gamma(x)=\xi(\gamma x)- f(\gamma)\xi(x)~.
%\]
%Differentiating it and using (\ref{}) and (\ref{}) we deduce that $dF=0$ that is $F$ is constant $\tau(\gamma)$.
%Since by a direct computation we have that $F_(\alpha\beta)(x)=F_\alpha(\beta x)+ f(\alpha) F_\beta(x)$,
%the correspondence $\gamma\mapsto \tau(\gamma)$ is a cocycle.
%
%If  the primitive $\xi$ is changed  by a constant vector $\xi_0$ the cocyle change by a coboundary.
%So we have a natural 
Using this natural identification we will explicitly describe the connecting homomorphism 
$\delta:\cC(S,h)\rightarrow H^1_{\tiny\hol}(\pi_1(S),\R^{2,1})$. 
Let $b\in\cC(S,h)$ and let $\tilde b:T\Hyp^2\rar T\Hyp^2$ be the lifting of $b$ to the universal cover. 
By Proposition \ref{existence of function}, there exists $\hat u\in\cun(\tilde S)$ such that $\tilde b=\Hess\hat u-\hat u\id$. 
By the equivariance of $b$, 
for every $\alpha\in\pi_1(S)$, $\hat u\circ\alpha^{-1}$ is again such that $\tilde b=\Hess(\hat u\circ\alpha^{-1})-(\hat u\circ\alpha^{-1})\id$. 
%Recall that $\tilde b$ is $\pi_1(S)$-equivariant, from which it follows
%\begin{align*}
%\tilde b(x)(w)&=d\gamma^{-1}(\tilde b(\gamma x)(d\gamma(w))=d\gamma^{-1}(\hat u(\gamma x)d\gamma(w)-\nabla_{d\gamma(w)}(\grad \hat u_{\gamma x}))= \\
%&=\hat u(\gamma x)w-\nabla_{w}d\gamma^{-1}(\grad\hat u_{\gamma x})=(\hat u\circ\gamma)(x)w-\nabla_{w}\grad(\hat u\circ\gamma)_{x}\,.
%\end{align*}
%We have used that $\nabla_{d\gamma(w)}d\gamma(v)=d\gamma\nabla_w v$ for an isometry $\gamma$ and that $\grad(\hat u\circ\gamma)_x=d\gamma^{-1}(\grad\hat u_{\gamma x})$. Indeed:
%\begin{align*}
%\langle \grad(\hat u\circ\gamma),w\rangle _x&=d(\hat u\circ\gamma)_x(w)=d\hat u_{\gamma x}\circ d\gamma_x(w)= \\
%&=\langle \grad \hat u_{\gamma x},d\gamma_x(w)\rangle _{\gamma x}=\langle d\gamma^{-1}\grad \hat u_{\gamma x},w\rangle _x\,.
%\end{align*}
%Then $\hat u\circ\gamma$ is such that $\tilde b=(\hat u\circ\gamma)\id-\Hess(\hat u\circ\gamma)$. 
By the exactness of (\ref{eq:exsq}) there is a vector $t_\alpha\in \R^{2,1}$ such that 
\[
(\hat u-\hat u\circ\alpha^{-1})(x)=\langle t_\alpha, \dev(x)\rangle
\]
where $\dev:\tilde S\to\Hyp^2$ is a developing map for the hyperbolic structure on $\tilde S$. The map $\alpha\to t_\alpha$ 
gives a cocyle. Since the definition depends on the choice of $\hat u$, it is easy to check that $t_\bullet:\pi_1(S)\rar\R^{2,1}$ is well-defined up to a coboundary. The cohomology class of $t_\bullet$ in $H^1_{\tiny\hol}(\pi_1(S), \R^{2,1})$ coincides
with  $\delta(b)$.

\subsection{Geometric interpretation} \label{geo_interpretation}
\begin{defi}
Let $S$ be a $\mathrm{C}^2$ Cauchy surface in a flat space-time 
$M$ and denote by $s$ its shape operator computed with respect to the future-pointing
normal vector.
We say that $S$ is strictly future convex if $s$ is positive.
%We say that $S$ is strictly past convex if $B$ is negative.
\end{defi}

%\begin{remark}
%Suppose $b:TS\rar TS$ is  an operator on a closed surface $S$ such that $\det b\neq 0$ at every point of $S$. Then necessarily $\det b> 0$ everywhere. Indeed, $\det b$ cannot change sign and if $b$ had $\det b<0$ at every point, then the unit eigenvectors relative to the positive eigenvalues at every point would define a nonvanishing vector field on $S$, which is not possible for a surface of genus $g\geq 2$. Therefore, $b$ is either positive definite or negative definite at every point. We will write respectively these cases with $b>0$ and $b<0$.
%\end{remark}
%For instance, any such operator can arise as the shape operator of a strictly convex embedded surface in the quotient of a domain in $\R^{2,1}$. More precisely, consider a MGH flat space-time $M$ with a Cauchy space-like strictly convex embedding $\sigma:S\rar M$. We can associate to $\sigma$ the embedding data $(I,s)$, where $I(v,w)=\langle \sigma(v),\sigma(w)\rangle$ is the first fundamental form of $S$ and $s(v)=\nabla_v N$ is the shape operator, for $N$ a unit normal vector to $\sigma(S)$. It is known that $I$ and $s$ satisfy the Gauss-Codazzi equations \cite{Petersen}, so we obtain an element of the set $D$ defined as:

In this section we fix a topological surface $S$ of genus $g\geq 2$. We consider pairs $(M, \sigma)$ where
$M$ is a maximal globally hyperbolic flat space-time (we will use the acronym MGHF hereafter) and $\sigma:S\rightarrow M$ is an embedding onto
a strictly future-convex  Cauchy surface. Recall $M$ is homeomorphic to $S\times\R$; we will always implicitly consider embeddings $\sigma:S\rar M$ which are isotopic to the standard embedding $S\hookrightarrow S\times\{0\}$.

By a classical result of \cite{choquet}, those pairs are parameterized by the embedding data of $\sigma$ which are the Riemannian metric $I$\ 
induced by $\sigma$ on $S$ and the shape operator $s$ of the immersion (computed with respect to the future-pointing normal vector).
Pairs $(I,s)$ are precisely the solutions of the so called Gauss-Codazzi equations:

\begin{equation} \label{GC} \tag{GC}
\begin{cases}
\det s=-K_I \\
\dn_I s=0
\end{cases}.
\end{equation}

We consider the set of embedding data of strictly convex Cauchy surfaces in a MGHF space-time, namely:
\begin{equation*}
\D= \left\{(I,s):\begin{aligned}
        &I\text{  Riemannian metric on }S\\ & s:TS\rar TS\text{ positive, self-adjoint for }I \\
        & (I,s) \text{ solves equations \eqref{GC}}
         \end{aligned}
 \right\}\,.
\end{equation*}
Observe that the Riemannian metric $I$ in a pair $(I,s)\in\D$ is necessarily of negative curvature. First of all we want to show that the space $\D$ can be naturally identified with the space 

\begin{equation*}
\E= \left\{(h,b):\begin{aligned}
        &h\text{ hyperbolic metric on }S\\
&b:TS\rar TS\text{ self-adjoint for }h,\dn_h b=0,b> 0
         \end{aligned}
 \right\}\,.
\end{equation*}

\begin{prop} \label{identificationDE}
Let $(I,s)$ be an element of $\D$. Then the metric $h(v, w)=I(sv, sw)$ is hyperbolic and the operator $b=s^{-1}$ satisfies
the Codazzi equation for $h$.
Conversely if $(h,b)\in \E$ then $I=h(bv, bw)$ and $s=b^{-1}$ are solutions of \eqref{GC}.
\end{prop} 
\begin{proof}
Using the formula $\nabla^{h}=s^{-1}\nabla^I s$ which relates the Levi-Civita connection of $h$ and $I$ (see \cite{labourieCP} or \cite{Schlenker-Krasnov}), it is easy to check that $b$ is an $h$-Codazzi tensor, and the same implication with the roles of $h$ and $I$ switched. We also have $K_h={K_I}/{\det s}$ when $h=I(s,s)$, where $K_h$ and $K_I$ are the curvatures of $h$ and $I$. Therefore $K_h=-1$. Viceversa, starting from $(h,b)$, one obtains the Gauss equation $K_I=-1/\det b=-\det s$.  
\end{proof}

\begin{remark}
The group $\Diffeo(S)$ naturally acts both on $\E$ and $\D$.
It is important to remark here that the identification given by Proposition \ref{identificationDE} commutes with those actions. 
\end{remark}

The following theorem shows the relation between the embedding data $(I, s)$  of a convex  Cauchy embedding $S$ into $M$
and the holonomy of $M$.

\begin{thmx}\label{theorem holonomy}
Let $M$ be a globally hyperbolic space-time and $S$ be a uniformly convex Cauchy surface with embedding data $(I, s)$.
Let $h$ be the third fundamental form of $S$ and $b=s^{-1}$.
Then
\begin{itemize}
\item the linear holonomy of $M$ coincides with the holonomy of $h$;
\item the translation part of the holonomy of $M$ coincides with $\delta b$.
\end{itemize}
\end{thmx}

%\begin{theorem} \label{theorem holonomy}
%Let $\sigma:S\rightarrow M$ be a strictly future convex Cauchy embedding of $S$ into some MGHF space-time $M$.
%Denote by $(I,s)$ the embedding data of $\sigma$ and by $(h, b)$ the corresponding data in $\E$.
%Then
%\begin{itemize}
%\item the linear holonomy of $M$ coincides with the holonomy of $h$;
%\item the translation part coincides with the cocycle $\delta(b)\in H^1_{\tiny\hol}(\pi_1(S), \R^{2,1})$,
%where $\delta$ is the coboundary operator constructed in (\ref{coboundary operator}).
%\end{itemize}
%\end{theorem}

The rest of this section is devoted to the proof of this theorem.
We first give a more geometric meaning to the correspondence between $\D$ and $\E$.
The key ingredient is the Gauss map.
Given a Cauchy immersion $\sigma:S\to M$ we can consider the equivariant immersion of the universal cover of $S$,
$\tilde\sigma:\tilde S\to\R^{2,1}$ obtained by composing the inclusion of $\tilde S$ into $\tilde M$ with the developing map
of $M$. Notice that the holonomy of $\tilde\sigma$ coincides with the holonomy of $M$.

The Gauss map of the immersion is then the map $G:\tilde S\to\Hyp^{2}$ sending a point $x$ to the future normal of the immersion
$\tilde\sigma$. 
Notice that $d\sigma(T_x\tilde S)=\langle G(x)\rangle^\perp= T_{G(x)}\Hyp^{2}$. 
Denote by $\tilde s$ and $\tilde I$  the lifting of $s$ and $I$ to the universal cover.
Using that $\tilde s$ is the covariant derivative
of the future  normal field by the flat $\R^{2,1}$-connection, it is immediate to see that 
\begin{equation}\label{eq:dg}
    dG_x(v)=d\sigma(\tilde s(v)).
\end{equation}
This identity shows that the pull-back of the hyperbolic metric through $G$ is the metric $h(v,w)=I(sv, sw)$ so $G$ is a local
isometry between $(\tilde S, \tilde h)$ and $\mathbb H^2$. This implies that  $G$ is the developing map of $h$.

\begin{prop}\label{linear part holonomy}
%Let $M$ be a MGH flat spacetime homeomorphic to $S\times\R$. 
Let $(I,s)$ be the embedding data of a strictly convex spacelike Cauchy surface in some MGHF space-time $M$, and denote by
$(h,b)$ the pair in $\E$ corresponding to $(I,s)$.
If $\tilde\sigma:\tilde S\rightarrow\R^{2,1}$ is the space-like immersion corresponding to the data $(I,s)$, then the corresponding
Gauss map  $G:\tilde S\rightarrow\Hyp^2$ is a developing map for $h$.
\end{prop}

We now want to compute the translation part of the holonomy of $M$, once the embedding data $(I,s)$ are known. 
In particular we want to show that the translation part of the holonomy equals $\delta b$ where $(h,b)$ is the pair corresponding to
$(I, s)$.

To this aim we need to construct a function $\hat u:\tilde S\rightarrow\R$ such that $\tilde b=\hess_{\tilde h} \hat u -\hat u \id$.

\begin{prop} \label{caratterizzazione embedding inversi gauss}
Let $(I, s)\in \D$ and  $(h,b)\in \E$ be the corresponding pair. 
Let $\tilde\sigma:\tilde S\rar\R^{2,1}$ be an embedding whose first fundamental form 
is $\tilde I$ and whose shape operator is $\tilde s$, and denote by $G:\tilde S\rightarrow\Hyp^2$ its Gauss map.
%The embedding $\tilde\sigma:\Hyp^2\rar\R^{2,1}$ has linear holonomy $f(\gamma)=\gamma$ if and only if $\tilde\sigma(x)=\hat u(x)x-\grad\hat u(x)$ where $\hat u\in\cun(\Hyp^2)$ is such that $\tilde s^{-1}=\tilde b=\hat u\id-\Hess\hat u$ . 
Let us define $\hat u:\tilde S\to\R$ as  $\hat u(x)=\langle \tilde\sigma(x),G(x)\rangle$.
Then $\tilde s^{-1}=\tilde b=\Hess_{\tilde h}\hat u-\hat u\id$.
\end{prop}
\begin{proof}
Notice that the statement is local so we may suppose that $G$ is an isometry between $\tilde S$ and an open subset $U$ of
$\Hyp^2$. So we can  identify  $(\tilde S, \tilde h)$ to $U$ through the map $G$.
Under this identification we have $\hat u(x)=\langle \tilde\sigma(x), x\rangle$ and (\ref{eq:dg}) becomes
$d\tilde\sigma_x(v)=b_x(v)$.
%Suppose $\tilde G\circ\tilde\sigma=\textrm{id}$, which is equivalent to the hypothesis of $f(\gamma)=\gamma$ being the linear part. Define $\hat u(x)=-\langle \tilde\sigma(x),x\rangle$. 
%\noindent We decompose in $\R^{2,1}$ $$\tilde\sigma(x)=-\langle \tilde\sigma(x),x\rangle x+\tilde\sigma(x)^{T_x}=-\hat u(x)x+\tilde\sigma(x)^{T_x}\,.$$
By a computation we have
$$d\hat u_x(v)=\langle d\tilde\sigma_x(v),x\rangle+\langle\tilde\sigma(x),v\rangle=\langle\tilde\sigma(x)^{T_x},v\rangle\,,$$
where we are using that $d\tilde\sigma_x(v)=b_x(v)\in T_x\Hyp^2=x^\perp$. This shows that $\grad\hat u(x)=\tilde\sigma(x)^{T_x}$ and concludes that $\tilde\sigma(x)=\grad\hat u(x)-\hat u(x)x$. Furthermore,
\begin{align*}
\tilde b_x(v)=&d\tilde\sigma_x(v)=\nabla_v\grad\hat u+\langle\grad\hat u,v\rangle x-d\hat u_x(v)x-\hat u(x)v\\ =&( \Hess\hat u(x))(v)-\hat u(x)v\,.
\end{align*}
%Viceversa, if $\tilde\sigma(x)=\hat u(x)x-\grad\hat u(x)$ and $b=\hat u\id-\Hess\hat u$, by the same computation $d\tilde\sigma_x(v)=\tilde b_x(v)$ and it follows from Lemma \ref{lemma due condizioni per inversa gauss} that $\tilde G\circ\tilde\sigma=\textrm{id}$.
\end{proof}

The argument of the proof shows that in general the map $\tilde\sigma$ can be reconstructed using $G$ and $\hat u$ by the formula
\begin{equation}\label{eq:normal}
   \tilde\sigma(x)= dG(\grad\hat u(x))-\hat u(x) G(x)~.
\end{equation}

\begin{remark}
By a result  of Mess, if the metric $\tilde I$  on $\tilde S$ is complete (that is the case if it comes from a metric $I$ on the closed surface
$S$), then the immersion $\tilde\sigma$ is in fact an embedding on a space-like surface of $\R^{2,1}$.

It turns out that $S$ is future strictly convex iff the future $F$ of $\tilde\sigma(\tilde S)$ in $\R^{2,1}$ is convex.
The function $\hat u\circ G^{-1}:\Hyp^2\to \R$ coincides with the support function of $F$ (see \cite{Fillastre} for details on the support function
of convex subsets in Minkowski space).
 
\end{remark}

We can now conclude the proof of Theorem \ref{theorem holonomy}.

\begin{lemma}
Let $(I,s)\in \D$ and $(b,h)$ the corresponding pair in $\E$.
Denote by $\tilde\sigma:\tilde S\to \R^{2,1}$ the space-like immersion corresponding to
$(I,s)$ and by $G:\tilde S\to\Hyp^2$ the corresponding Gauss map.
Consider the function $\hat u:\tilde S\rightarrow\R$ defined by
$\hat u(x)=\langle  \tilde\sigma(x),G(x)\rangle.$
Then $$(\hat u-\hat u\circ \alpha^{-1})(x)=\langle t_\alpha, G(x)\rangle\,,$$ where $t_\alpha$ is the translation
part of the holonomy of the immersion $\tilde \sigma$.
\end{lemma}

\begin{proof}
We have
\[
   \tilde\sigma(\alpha^{-1} x)=\hol(\alpha^{-1})\tilde\sigma(x)-\hol(\alpha)^{-1}t_\alpha\,,
\]
and
\[
  G(\alpha^{-1} x)=\hol(\alpha^{-1})(G(x))~.
\]
So
\[
\hat u(\alpha^{-1} x)=\langle G(\alpha^{-1} x), \tilde\sigma(\alpha^{-1}x)\rangle=
\hat u(x)-\langle t_\alpha,G(x)\rangle~.
\]
and this concludes the proof.
\end{proof}

\begin{remark}\label{rm:closedhol}
If $S$ is a closed surface of genus $g\geq 2$,
the holonomy distinguishes MGHF structures on $S\times\R$ containing a convex Cauchy surface.
So Theorem \ref{theorem holonomy} implies that 
that two elements of $\E$, say $(h,b)$, and $(h',b')$, correspond to isotopic Cauchy immersions into the same space-time if and only if
\begin{itemize}
\item there is an isometry $F:(S, h)\to(S, h')$ isotopic to the identity.
\item $\delta(b)=\delta(b')$ (this makes sense as the holonomies of $h$ and $h'$ coincide for the previous point).
\end{itemize}

Let us denote by $\cS_+(S)$ the set of  isotopy classes of MGHF structures on $S\times\mathbb R$ containing a future convex surface.
Then, $\cS_+(S)$ can be realized as the quotient of $\E$ up to the 
identify $(h,b)$ and $(h', b')$ if the previous conditions are satisfied.
\end{remark}

\subsection{A de Rham approach to Codazzi tensors}

We want to give a different description of the connection homomorphism
\[
  \delta:\cC(S,h)\to H^1(S, \cF)~.
\]

Let us denote by $D$ the flat connection over $F\to S$, and let $\Omega^k(S, F)$ be the space of $F$-valued $k$-forms
over $S$, namely the space of smooth sections of the bundle $\Lambda^kT^*S\otimes F$.
The exterior differential is the operator
\[
    d^D:\Omega^k(S, F)\to\Omega^{k+1}(S, F)
\]
defined on simple elements by $d^D(\omega\otimes t)=d(\omega)\otimes t+(-1)^{k}\omega\wedge Dt$.

As $F$ is flat, $d^D\circ d^D=0$ and the de Rham cohomology of the bundle $F$ is defined as the cohomology of
the complex $(\Omega^\bullet(S, F), d^D)$.
As
\[
0\to\cF\to\Omega^0(-, F)\to\Omega^1(-, F)\to\ldots
\]
 is a fine resolution of $\cF$, by de Rham theorem $H^1(\cF)$ is naturally identified with $H^1_{\mathrm{dR}}(S, F)$.
 
The first aim of this section is to give a characterization of Codazzi operators over $S$ in terms of de Rham complex.
Recall that we have a developing section $\bfs{\iota}:S\to F$ obtained as the projection of the developing map.
The covariant derivative of $\bfs{\iota}$ provides a natural monomorphism $\bfs{\iota}_*:TS\to F$, where $\bfs{\iota}_*(v)=D_v\bfs{\iota}$.
If $S=\Hyp^2$, then $\bfs{\iota}$ corresponds to the natural inclusion $\Hyp^2\to \R^{2,1}$, and $\bfs{\iota}_*$ corresponds 
to the inclusion of tangent spaces of $\Hyp^2$ in $\R^{2,1}$.

Now, given any operator $b:TS\to TS$ we can consider the composition  $\bfs{\iota}_* b$ as an $F$-valued  $1$-form on $S$.
The following simple computation gives a characterization of self-adjoint Codazzi tensors:

\begin{prop}
Let $b: TS\to TS$ be any operator. 
Then $b$ is self-adjoint and Codazzi if and only if $\bfs{\iota}_* b$ is closed. 
\end{prop}
\begin{proof}
The usual splitting of the flat connection of $\R^{2,1}$ into the Levi-Civita connection of $\Hyp^2$
and second fundamental form gives in this setting the following formula
\[
    D_v(\bfs{\iota}_*X)=\bfs{\iota}_*(\nabla_v X)+ h(X, v)\bfs{\iota}(x)
\]
for any vector field $X$ over $S$ and any tangent vector $v$ at $x$.
Given two vector fields $X,Y$ on $S$ we get
\begin{align*}
   d^D(\bfs{\iota}_*b)(X, Y)=& D_X(\bfs{\iota_*}(bY))-D_Y(\bfs{\iota_*}(bX))-\bfs{\iota}_*b[X,Y] \\
   =&\bfs{\iota_*}(\nabla_X(bY)-\nabla_Y(bX)-b[X,Y])+ (h(X, bY)-h(Y,bX))\bfs{\iota}(x) \\
   =&\bfs{\iota_*}(d^\nabla b(X,Y))+(h(X, bY)-h(Y,bX))\bfs{\iota}(x)
\end{align*}
As the image of $\bfs{\iota_*}$ at $T_xS$ is the orthogonal complement of
$\bfs{\iota}(x)$ in $F_x$, the previous computation proves the statement.
\end{proof}

We can now give a description of the connection homomorphism $\delta$ under the usual identification between $H^1(S,\cF)$ and $H^1_{\mathrm{dR}}(S,F)$.

\begin{prop}\label{pr:216}
The connecting homomorphism $\delta:\cC(S,h)\rar H^1_{\mathrm{dR}}(S,F)$ of the short exact sequence \eqref{coboundary operator} is expressed by the formula
$$\delta(b)=[\bfs{\iota}_*b]\,.$$
\end{prop}

Before proving the Proposition, we give a preliminary Lemma.

\begin{lemma}\label{lm:potential}
Given a function $u\in\cun(S)$ and $b\in\cC(S,h)$, we have $b=\Hess u-u\id$ if and only if $\bfs{\iota}_*b=d^{D}(\bfs{\iota}_*\grad u-u\bfs{\iota})$.
\end{lemma}
\begin{proof}
By an explicit computation, 
\begin{align*}
d^D(\bfs{\iota}_*\grad u-u\bfs{\iota})=&
D(\bfs{\iota}_*\grad u)-(du)\bfs{\iota}-uD\bfs{\iota}\\=&
\bfs{\iota}_*(\hess u)+ h(\grad u, \bullet)\bfs{\iota}
-(du)\bfs{\iota}-u\bfs{\iota}_*\\ =&
\bfs{\iota}_*(\hess u -u\id)\,.
\end{align*}
\end{proof}

\begin{remark}
If $b$ is positive and  
$(I, s)\in\D$ are the embedding data associated with  $(h,b)$, the corresponding  map 
$\tilde{\sigma}:\tilde S\to \R^{2,1}$, considered as a section of the trivial flat $\R^{2,1}$-bundle 
on $\tilde S$, solves the equation
\[
    d^D\tilde{\sigma}=\bfs{\iota}_*b\,,
\]
 so Lemma \ref{lm:potential} is a generalization of Formula \eqref{eq:normal}.
\end{remark}

\begin{proof}[Proof of Proposition \ref{pr:216}]
The construction of the operator $\delta$ works as follows.
Take a good cover $\{U_i\}$ of $S$. On each $U_i$ there is a  function
$u_i$ such that $b|_{U_i}=\hess u_i-u_i\Id$.
Now $u_{i_1}-u_{i_0}=V(t_{i_0i_1})$ for some  flat sections $t_{i_0i_1}$ of $F$ on $U_{i_0}\cap U_{i_1}$.
The family $\{t_{i_0i_1}\}$ forms an $F$-valued $1$-cocyle. Since $\Omega^0(-,F)$ is fine, there are smooth (but in general non flat)
sections $\eta_i$ over $U_i$ such that 
\begin{equation}\label{eq:coc}
t_{i_0i_1}=\eta_{i_1}-\eta_{i_0}\,.
\end{equation}
The differentials of $\eta_i$ glue to a global $F$-valued closed form which represents $\delta b$ in the de~Rham cohomology. % (an in particular its cohomology class is independent of the choice of $\eta_i$).
We claim that  $\eta_i=\bfs{\iota}_*(\grad u_i)-u_i\bfs{\iota}$ satisfy the condition \eqref{eq:coc}.
From the claim and Lemma \ref{lm:potential} we easily get that $\delta b=[\bfs{\iota}_*b]$.

To prove the claim it is sufficient to check the following formula
\[
   u_{i_1}(x)-u_{i_0}(x)=\langle (\eta_{i_1}-\eta_{i_0})(x), \bfs{\iota}(x)\rangle\,,
\]
which is immediate once one recalls that the image of $\bfs{\iota}_*$ is orthogonal to $\bfs{\iota}$.
\end{proof}

%----------------------------------------------sezione 3---------------------------------

%\input bs-sec3-nuovo.tex

\section{Relations with Teichm\"uller theory} \label{closed surfaces}
In this section we will consider symmetric Codazzi tensors as infinitesimal deformations of a metric, inducing
in this way an infinitesimal deformation of the conformal structure.

We will see that in the closed case the first order variation of the holonomy map $$\Hol:\cT(S)\to\cR(\pi_1(S), \SO(2,1))/\!\!/ \SO(2,1)$$ which associates to a conformal structure the holonomy of the hyperbolic structure in its conformal class, can be 
explicitly computed in terms of the coboundary operators $\delta$ we already considered.

As a by-product we will see that the corresponding map $\mathbb E\to T\cT(S)$, induces to the quotient a bijective map
between the space $\cS_+(S)$ of MGHF structures on $S\times \mathbb R$ and $T\cT(S)$. 

\subsection{Killing vector fields and Minkowski space} \label{subsec Killing}
The Lie algebra  $\so(2,1)$ is naturally identified to the set of
Killing vector fields on $\mathbb H^2$, so it is realized as a subalgebra of the space of
smooth vector fields $\mathfrak{X}(\Hyp^2)$ on $\mathbb H^2$.
There is a natural action of $\SO(2,1)$ on $\mathfrak{X}(\Hyp^2)$ that is simply defined by
\[
  (A_* X)(x)=dA (X(A^{-1}x))\,. 
\] 
This action restricts to the adjoint action on $\so(2,1)$.

The Minkowski cross product is defined by $v\boxtimes w=*(v\wedge w)$, where $*:\Lambda^2(\mathbb R^{2,1})\to\mathbb R^{2,1}$
is the Hodge operator associated to the Minkowski product. As in the Euclidean $3D$ case, it leads to 
a natural identification between $\mathbb R^{2,1}$ and $\so(2,1)$ which commutes with the action
of $\SO(2,1)$. Basically  any vector $t\in\mathbb R^{2,1}$ is associated  
with the Killing vector field on $\mathbb H^2$ defined by $X_t(x)=t\boxtimes x$. 
This identification will be denoted  by $\Lambda:\mathbb R^{2,1}\to \so(2,1)$.

Notice that  any hyperbolic surface $S$ is equipped with a $\SO(2,1)$-flat bundle $F_{\so(2,1)}$ 
whose flat sections  correspond to Killing vector fields on $S$.
Elements of $F_{\so(2,1)}$ are germs of Killing vector fields on $S$, so an evaluation map
$\ev:F_{\so(2,1)}\to TS$ is defined.

On the other hand, the isomorphism $\Lambda$ provides a flat isomorphism, that with a standard abuse we still denote by $\Lambda$,
between the $\mathbb R^{2,1}$-flat bundle $F$ and $F_{\so(2,1)}$.
Under this identification the developing section $\bfs{\iota}:S\to F$ corresponds to the section sending each point to the infinitesimal
generator   of the rotation about the point.

We recall that the Zarinski tangent space at a point $\rho:\pi_1(S)\to\SO(2,1)$ to the character variety $\cR(\pi_1(S), \SO(2,1))/\!\!/ \SO(2,1)$ 
is canonically identified with the cohomology group $H^1_{\tiny\Ad \rho}(\pi_1(S), \so(2,1))$ (\cite{Goldman}). The identification goes in the following way.
Given a differentiable  path of representations $\rho_t$ such that $\rho_0=\rho$, we put
\[
   \dot \rho(\alpha)(x)= \left. \frac{d}{dt}\right|_{t=0}\rho_t(\alpha)\circ \rho_0(\alpha)^{-1}(x). %f_0(\alpha)^{-1}_*(X_\alpha)\,,
 \]
%where $X_\alpha$ is the field defined by $X_\alpha(f_0(\alpha)x)=\frac{df_t(x)}{dt}|_{t=0}$.

\subsection{Deformations of hyperbolic metrics}
We fix a hyperbolic surface $(S, h)$ -- that in this subsection we will not necessarily assume complete -- with holonomy
$\hol:\pi_1(S)\to\SO(2,1)$. We still make the assumption that $\hol$ is not elementary.

It is well known that if $h_t$ is a family of hyperbolic metrics on $S$ which smoothly depend on $t$ and such that $h_0=h$, then
the $h$-self-adjoint operator $b=h^{-1}\dot h$ satisfies 
 the so called  Lichnerowicz equation (see \cite{tromba})
\[
 L(b)=-(\Delta-1/2)\tr b+\delta_h\delta_h b=0\,,
\]
where $\delta_h b$ is the 1-form obtained by contracting $\nabla b$ namely,  $\delta_h(b)(v)=\tr(\nabla_{\bullet} b) (v)$,
whereas the second $\delta_h$ is the divergence on $1$-forms, $\delta(\omega)=\tr h^{-1}\nabla\omega$.

If $X$ is  a vector field on $S$ with compact support, 
and $f_t:S\to S$ is the flow generated by $X$, then putting $h_t=f_t^*(h)$, we
have that $h^{-1}\dot h=2\mathbf S\nabla X$, where $\mathbf S\nabla X$ denotes the symmetric part of the operator $\nabla X$.
It follows that  $\mathbf S\nabla X$ is a solution of Lichnerowicz equation.
Being $L$ a local operator,  we deduce that $\mathbf S\nabla X$ is a solution of Lichnerowicz equation  for any vector field $X$. 

Conversely, as any deformation of a hyperbolic metric is locally trivial, it can be readily shown that
any solution of  Lichnerowicz equation can be locally written  as $\mathbf S\nabla X$ for some vector field $X$.

Denoting by $\mathfrak{X}$ the sheaf of  smooth vector fields  and by $\cL$ the sheaf of solutions of Lichnerowicz equation,
the sheaf morphism $\mathbf S\nabla:\mathfrak{X}\to\cL$,  defined by $X\mapsto \mathbf S\nabla X$, is surjective.
On the other hand the kernel of this morphism is the subsheaf of $\mathfrak{X}$ of Killing vector fields.
This is simply the image of the sheaf $\cF_{\so(2,1)}$ of flat sections of $F_{\so(2,1)}$ through the evaluation map
$\ev:\cF_{\so(2,1)}\to\mathfrak{X}$.

So we have a short exact sequence of sheaves
\[
  0\to\cF_{\so(2,1)}\to\mathfrak{X}\to\cL\to 0
\]
that in cohomology gives a sequence
\begin{equation}\label{eq:deltaconn}
\begin{CD}
 0@>>>\mathfrak{X}(S)@>>>\cL(S)@>\mathfrak{d}>>H^1_{\tiny\hol}(\pi_1(S), \so(2,1))@>>> 0\,,
\end{CD}
\end{equation}
where  $\hol$ is the holonomy of $h$ and 
again we are using the canonical identification $$H^1(S,\cF_{\so(2,1)})=H^1_{\tiny\hol}(\pi_1(S), \so(2,1)).$$
We claim that if $b$ is the first order deformation of a family of hyperbolic metrics,  namely $b=h^{-1}\dot h$, then  
 $\mathfrak{d}$ coincides with the derivative  of the holonomy map (which we still denote by $\Hol$)
$$\Hol:\cM_{-1}\to\cR(\pi(S), \SO(2,1))/\!\!/ \SO(2,1)$$ along the family of metrics.

More precisely we will prove the following result.

\begin{prop}\label{pr:deltastar}
Let $h_t$ be a family of hyperbolic metrics on a surface $S$, not necessarily complete, and suppose that $h$ depends smoothly on $t$.
Then we can find a family of representations $\hol_t$ so that 
\begin{itemize}
\item $\hol_t$ is a representative of the holonomy of $h_t$ in its conjugacy class.
\item $\hol_t$ smoothly depends on $t$.
\end{itemize}
Moreover we have
\begin{equation}
   2\dot\hol=\mathfrak{d}(h^{-1}\dot h)~.
\end{equation}
\end{prop}

\begin{remark} \label{superpippo}
A simple way to understand the operator $\mathfrak{d}$ is the following. Given $b\in\cL(S)$, on the universal covering there is a field $T$ such that $\tilde b=\bfs{S}\nabla T$. As $\tilde b$ is $\pi_1(S)$-invariant, it turns out that $T-\alpha_* T$ is a Killing vector field on $\tilde S$ for any $\alpha\in\pi_1(S)$. Thus there is an element $\tau_\alpha\in\so(2,1)$ such that $d\dev(T-\alpha_*T)(x)=\tau_\alpha(\dev(x))$ and $\mathfrak{d}(b)$ coincides with the cocycle $\tau_\bullet$.
\end{remark}

\begin{proof}
Notice that  the exponential maps $\exp_t$ of $h_t$  define  a differentiable map  on an open subset $\Omega$ of 
$(-\epsilon,\epsilon)\times T\tilde S$ 
\[
\mathbf{\exp}:\Omega\to \tilde{S}
\]
such that $\mathbf{\exp}(t,x,v)=\exp_t(x,v)$.

Now, let $\dev_t$ be the developing map of $h_t$ normalized 
so that $\dev_t(p_0)=\dev_0(p_0)$ and $d(\dev_t)(p_0)=d(\dev_0)(p_0)$, 
and consider the map
\[
\mathbf{\dev}:(-\epsilon,\epsilon)\times\tilde S\to\hyp^2
\]
defined by $\mathbf{\dev}(t,p)=\dev_t(p)$.
As $\dev_t$   commutes with the exponential map
\[
      \dev_t(\exp_t(x,v))=\exp_{\mathbb H^2}((\dev_t)_*(x,v))\,,
\]
one readily sees that $\mathbf{\dev}$ is differentiable.

As a representative for the holonomy representation $\hol_t$ of $h_t$
can be chosen so that
\[
    \dev_t\circ\alpha=\hol_t(\alpha)\circ\dev_t\,,
\]
it turns out that the map $\hol:(-\epsilon, \epsilon)\to\cR(\pi_1(S),
\SO(2,1)$ is differentiable as well.

Now differentiating the identity
\[
   h_t(v, w)= \langle d\dev_t(v), d\dev_t(w)\rangle
\]
one sees that on the universal covering
\[
     h^{-1}\dot h=2\mathbf S\nabla \dot\dev\,,
\]
where we have put $\dot\dev=d(\dev_0)^{-1}\frac{d\dev_t}{dt}|_{t=0}$.
On the other hand, for any $\alpha\in\pi_1(S)$,  differentiating the identity
\[
    \dev_t(\alpha x)=\hol_t(\alpha)\dev_t(x)
\]
one gets that
\[
    (d\dev_0)\dot\dev(\alpha x)=(d\dev_0)(d\alpha)\dot\dev(x)+\dot\hol(\alpha)\dev_0(\alpha x)~.
\] 
It can be checked (compare Remark \ref{superpippo}) that $\mathfrak{d}(h^{-1}\dot h)_\alpha=2d\dev_0(\dot\dev-\alpha_*\dot \dev)$, hence the conclusion follows. 
 \end{proof}

\subsection{Codazzi tensors as deformations of conformal structures on a closed surface}
In this section we restrict to the case $S$ closed. We fix a hyperbolic metric $h$ with holonomy $\hol$ and denote by $J:TS\to TS$ the almost-complex structure induced by $h$.

We remark that in general a Codazzi tensor $b$ for $h$ is not an infinitesimal deformation of a family of hyperbolic metrics, unless $b$ is traceless.
Indeed the following computation holds.
\begin{lemma}
If $b$ is a self-adjoint Codazzi tensor of a hyperbolic surface $(S,h)$ then 
$L (b)=\tr (b)/2$~.
\end{lemma}
\begin{proof}
Notice that $(\delta_h b)(v) =\tr(\nabla_\bullet b)(v)=\tr\nabla_v b$, where the last equality holds
as $b$ is Codazzi. As the trace commutes with $\nabla$,  $(\delta_h b)(v)=(d\tr b)(v)$ so $\delta_h b=d\tr b$,
and  $\delta_h\delta_h b=\Delta(\tr b)$.
The conclusion follows immediately.
\end{proof}

 We may consider $b$ as an infinitesimal deformation of the conformal structure.
Indeed  for small $t$ the bilinear form
\[
    \hat h_t(v,w)=h((\Id+tb)v, (\Id+tb)w)
\]
defines a path of Riemannian metrics which  smoothly depends on $t$.

So for each $t$, there is a uniformization function $\psi_t:S\to\mathbb R$ such that
$h_t= e^{2\psi_t}\hat h_t$ is the unique hyperbolic metric conformal to $h_t$.
It is well known that the conformal factor $\psi_t$ smoothly depends on $t$ (compare  \cite{tromba}), so
the path of holonomies $[\hol_t]$ defines a smooth path in the character variety.

The following proposition computes the first order  variation of $\hol_t$ (that is an element of
$H^1_{\tiny\Ad\circ\tiny\hol}(\pi_1(S), \so(2,1))$ in terms of $b$.  

\begin{prop}\label{pr:holclos}
Let $h$ be a hyperbolic metric on $S$, $X_h\in\Teich(S)$ be its complex structure and let $b\in\cC(S, h)$. Let $b=b_q+\hess u-u\Id$ the decomposition of $b$ given in Proposition \ref{pr:cod-dec}.
Then
\[
   d\Hol_{X_h}([b_0])=-\Lambda(\delta(Jb_q))
\] 
where $\Lambda: H^1_{\tiny\hol}(\pi_1(S), \mathbb R^{2,1})\to H^1_{{\tiny\Ad\circ\tiny\hol}}(\pi_1(S), \so(2,1))$ is the
isomorphism induced by the $\SO(2,1)$-equivariant isomorphism $\Lambda:\R^{2,1}\rar \so(2,1)$.
\end{prop}
\begin{remark}
As $b_q$ is traceless and self-adjoint, $J b_q$ is traceless and self-adjoint as well, and it turns out that $J b_q=b_{iq}$.
\end{remark}
\begin{proof}
Let $\dev_t$ be  a family of  developing maps of $h_t$ depending smoothly on $t$ and denote by
$\hol_t$ the holonomy representative  for which $\dev_t$ is equivariant.

By Proposition \ref{pr:deltastar},
 $2\dot\hol=\mathfrak{d}(h^{-1}\dot h),$ where $\mathfrak{d}:\cL(S)\to H^1_{\tiny\Ad\circ\tiny\hol}(\pi_1(S), \so(2,1))$ is the connecting
homomorphism defined in \eqref{eq:deltaconn}.

By differentiating the identity $h_t(v,w)=e^{2\psi_t}h((\Id+tb)v, (\Id+tb)w)$ one gets
\[
   \dot h(v,w)=2 h((\dot \psi\Id+b)v, w)
\]
that is
\[
   \frac{1}{2} h^{-1}\dot h =(\dot \psi-u)\Id+\Hess u+b_q~.
\]
Now $h^{-1}\dot h$, $b_q$ and $\Hess u=\mathbf S\nabla(\grad u)$ are solutions of Lichnerowicz equation, so by linearity 
$L((\dot\psi- u)\Id)=0$. An explicit computation shows that this precisely means that the function $\phi=\dot \psi-u$
must satisfy the equation $\Delta\phi-\phi=0$. As $\phi$ is a regular function on $S$ and we are assuming $S$ is closed, we deduce that
$\phi\equiv 0$.

Thus $h^{-1}\dot h=2\Hess u+2b_q$. Notice  that $\mathfrak{d}(\Hess u)=0$ as $\Hess u=\mathbf S\nabla(\grad u)$ on $S$.
So we get
\[
   2\dot\hol=\mathfrak{d}(h^{-1}\dot h)=2\mathfrak{d}(b_q)~.
\]
To compute $\mathfrak{d}(b_q)$ we have to find a vector field $T$ on the universal covering, such that
$b_q=\mathbf S\nabla T$, and $\mathfrak{d}(b_q)$ is determined by  the equivariance of $T$.

Now as $Jb_q$ is still a symmetric Codazzi tensor,  on the universal cover we can find
a function $v$ such that $Jb_q=\hess v-v\Id$. This implies that $b_q=-J\hess v+J v=-\nabla(J\grad v)+J v$.
That is, the field $T=-J\grad v$ satisfies the property we need.
It follows that
\begin{equation} \label{paperoga}
  (d\dev_0)^{-1}(\dot\hol(\alpha))=(-J\grad v)-\alpha_*(-J\grad v) ~,
\end{equation}
where $\dev:\tilde S\to\mathbb H^2$ is the developing map for $h$.
 
The conclusion then follows by comparing Equation \eqref{paperoga} with the formula proved in the following Lemma, that we prove separately as it does not depend on the fact that $S$ is closed.\end{proof}

\begin{lemma}
Let $b$ a Codazzi tensor on any hyperbolic surface $S$ and let $v$ be a function on the universal cover such that
$b=\Hess f-f\Id$.
Then for any $\alpha\in\pi_1(S)$  and $x\in\tilde S$ we have that
\[
   J\grad f-\alpha_*(J\grad f)=-(d\dev)^{-1}\Lambda(\delta b)_{\alpha}\,,
\]
%\[
%d\dev_x(J\grad f(x) -\alpha_*(J\grad f))(x))=-(\delta b)_\alpha\boxtimes \dev(x)~.
%\]
where $\dev:\tilde S\to\mathbb H^2$ is the developing map.
\end{lemma}
\begin{proof}
Let us  put $t_\alpha=(\delta b)_\alpha$.
As   $\Lambda(t)(\bullet)=t\boxtimes \bullet$,
we have to prove that for any $x$ in $\tilde S$
\[
d\dev_x(J\grad f(x) -\alpha_*(J\grad f))(x)=-t_\alpha\boxtimes \dev(x)\,.
\]
The point is that $(f-f\circ\alpha^{-1})(x)=\langle t_{\alpha}, \dev(x)\rangle$ for any $x\in\tilde S$. 
Thus
$$d\dev_x(\grad f-\alpha_*\grad f)=t_{\alpha}+\langle t_\alpha, \dev(x)\rangle \dev(x)~.$$
As for any $v\in T_x \tilde S$ we have that $d\dev_x(Jv)=J_{\mathbb H^2}d\dev_x(v)=\dev(x)\boxtimes d\dev_x(v)$ 
we get
$$
d\dev_x(J(\grad f -\alpha_*\grad f))= 
\dev (x)\boxtimes d\dev_x(\grad f -\alpha_*\grad f)=
\dev(x)\boxtimes t_{\alpha}\,,
$$
and the conclusion follows.
%The thesis follows by recalling that $\Lambda(t)(\bullet)=\bullet\boxtimes t$.
\end{proof}

\begin{remark}
We want to emphasize the reasons we restricted Proposition \ref{pr:holclos} to the case
of closed surfaces.

There are some technical issues.
For instance the metric $\hat h_t$ is not well defined if the eigenvalues of $b$ are not bounded and one
should use $\Id+\chi_tb$, where $\chi_t$ is a function going sufficiently fast to $0$ at infinity for each $t$
and such that $\partial_t \chi_t(0)=1$.
Moreover  the  uniformization factor $\psi_t$ on $t$ is well defined only if we restrict on some classes of hyperbolic metrics
(e.g. complete metrics, metrics with cone singularities) and its smooth dependence on the factor is more complicated.

More substantial problems are related to the splitting of $b$ as traceless part and trivial part.
The splitting is related to the solvability of the equation $\Delta u-2u=\tr (b)$. Now in the non closed case to get
existence and uniqueness of the solution, some asymptotic behavior of $\tr (b)$ must be required.

A related  problem is that on an open surface there are smooth non trivial
 solutions of the equation $\Delta \phi-\phi=0$.
So in order to prove that $\phi\equiv 0$, $\phi$ is needed  to have some good behavior at infinity (which can be obtained only  requiring  some extra hypothesis on the behavior of $b$ in the ends).

We will discuss the case of hyperbolic metrics with cone singularity in Section \ref{punctured surfaces}, where these problems will become evident.
 \end{remark}
 
\subsection{A global parameterization of MGHF space-times with closed Cauchy surfaces} \label{glopar}

We consider the space $\E$ introduced in Section \ref{sec:cauco}. We know that this space parameterizes
embedding data of uniformly convex surfaces in some MGHF space-time, and we have already remarked
that the space $\cS_+(S)$ of MGHF structures on $S\times\mathbb R$ containing a closed convex Cauchy surface is
the quotient of $\E$ by identifying $(h,b)$ and $(h',b')$ if $h$ and $h'$ are isotopic and $\delta b=\delta b'$ (see Remark
\ref{rm:closedhol}) 

We want to use results of the previous section to construct a natural bijection between $\cS_+(S)$
and the tangent bundle of the Teichm\"uller space of $S$.
Let us briefly recall some basic facts of Teichm\"uller theory that we will use. See for instance \cite{gardiner} for more details.

Elements of $\cT(S)$ are   complex structures on $S$, say $X=(S, \mathcal A)$,
 up to isotopy. In the classical Ahlfors-Bers theory, 
 the tangent space of $\cT(S)$ at a point $[X]\in\cT(S)$ is identified with a quotient of 
the space of Beltrami differentials $\cB(X)$. A Beltrami differential is a $L^\infty$ section of the bundle $K^{-1}\otimes \bar K$,
where $K$ is the canonical bundle of $X$, that simply means that a Beltrami differential is a $(0,1)$-form with value in the holomorphic
tangent bundle of $X$.
There is a natural pairing between quadratic differentials and Beltrami differentials, 
given by the integration of the $(1,1)$ form obtained  by contraction
\[
    \langle q, \mu\rangle=\int_S q\bullet\mu\,,
\]
where in complex chart $q\bullet\mu:=q(z)\mu(z)dz\wedge d\bar z$, if $\mu=\mu(z)d\bar z/dz$ and $q=q(z)dz^2$.

We say that a Beltrami differential $\mu$ is trivial if $\langle q,  \mu\rangle =0$ for any holomorphic quadratic differential.
We will denote by $\cB(X)^\perp$ the subspace of trivial Beltrami differentials.

The tangent space  $T_{[X]}\cT(S)$  is naturally identified with $\cB(X)/\cB(X)^\perp$ as a complex vector space.
The identification goes as follows: suppose to have a $C^1$-path $\gamma:[0,1]\to\cT(S)$ such that $\gamma(0)=[X]$.
Then it is possible to choose a family of representatives $\gamma(t)=[X_t]$ such that the Beltrami differential $\mu_t$
of the identity map $\Id: X\to X_t$ is a $C^1$ map of $[0,1]$ in $\cB(X)$. It is a classical fact that $\dot\mu(0)$ 
does not depend on the choice of the representatives $X_t$ up to a trivial differential, and thus one can identify 
$\dot\gamma(0)$ with the class of $[\dot\mu]$ in $\cB(X)/\cB(X)^\perp$.
Smooth trivial Beltrami operators can be expressed as $\bar\partial\sigma$, where $\sigma$ is a section of $K^{-1}$.

In order to link this theory with our construction it seems convenient
to identify the holomorphic tangent bundle $K^{-1}$ of a Riemann
surface $X=(S,\mathcal A)$ with its real tangent bundle $TS$. Basically if $z=x+iy$
is a complex coordinate one identifies the tangent vector
$a\frac{\partial\,}{\partial x}+b\frac{\partial\,}{\partial y}$ with
the holomorphic tangent vector $(a+ib)\frac{\partial\,}{\partial
  z}$. Notice that under this identification the multiplication by $i$
on $K^{-1}$ corresponds to the multiplication by the almost-complex
structure $J$ associated with  $X$.  Moreover Beltrami differentials correspond to
operators $m$ on $TS$ which are anti-linear for $J$: $mJ=-Jm$.  By
some simple linear algebra this is equivalent to $\tr (m)=0$ and $\tr
(Jm)=0$, or analogously Beltrami differentials correspond to traceless
operators which are symmetric for some conformal metric on $(TS, J)$.

More explicitly, if in local complex coordinate $\mu=\mu(z)d\bar z/dz$, the corresponding operator in the real coordinates is
\begin{equation} \label{formulabeltramicodazzi}
    m= \left(\begin{array}{ll} \Re(\mu) & \Im(\mu)\\ \Im(\mu) & -\Re(\mu)\end{array}\right)\,.
\end{equation}
We recall that $\Teich(S)$ is a complex manifold. Its almost-complex structure $\cJ$ corresponds in the complex notation to the multiplication by $i$ of the Beltrami differential. In the real notation, this is the same as $\cJ([m])=[Jm]$.

Now we denote by $X=X_h$ the complex structure determined by a hyperbolic metric $h$.
Given a Codazzi operator for the metric $h$, we have considered a smooth path of metrics
$\hat h_t=h(\Id+tb, \Id+tb)$, which determines a smooth path in the Teichmuller space $[X_t]$, where
$X_t$ is the complex structure on $S$ determined by $\hat h_t$.

It turns out that the tangent vector of this path is simply the class of the Beltrami differential $b_0$, where
$b_0=b-(\tr b/2)\Id$ is the traceless part of $b$. 

The main theorem we prove in this section is the following:

\begin{thmx}\label{theorem parametrization closed surfaces}
Let $h$ be a hyperbolic metric on a closed surface $S$, $X$ denote the complex structure underlying $h$, and $\cC(S, h)$ be the space of
self-adjoint $h$-Codazzi tensors. Then the following diagram is commutative 
\begin{equation}\label{eq:diag-comm}
\begin{CD}
\cC(S, h) @>\Lambda\circ \delta>>H^1_{\tiny\Ad\circ\tiny\hol}(\pi_1(S),\so(2,1))\\
@V\Psi VV                                       @A d\Hol AA\\
T_{[X]}\cT(S) @>>\mathcal J > T_{[X]}\cT(S)
\end{CD}
\end{equation}
where $\Lambda:H^1_{\tiny\hol}(\pi_1(S), \R^{2,1})\to H^1_{\tiny\Ad\circ\tiny\hol}(\pi_1(S),\so(2,1))$ is the natural isomorphism, and
$\cJ$ is the complex structure on $\cT(S)$.
\end{thmx}

%\begin{theorem} \label{theorem parametrization closed surfaces}
%The map
%\[
%\Psi:\E\to T\cT(S)
%\]
%sending $(h,b)$ to $([X_h], [b_0])$ is surjective. Moreover, 
%$\Psi(h,b)=\Psi(h',b')$ if and only if $h$ and $h'$ are isotopic metrics and $\delta(b)=\delta(b')$.
%
%
%\noindent In particular the map induces to the quotient a bijective map
%\[
%\bar\Psi: \cS_+(S)\to T\cT(S).
%\] 
%\end{theorem}

\begin{remark} \label{remarkremark}
In order to prove this Theorem, we need to link the Levi-Civita connection on $S$ with the complex structure $X$. %This remark will be used also in Section \ref{punctured surfaces}.

Let $X=(S,\mathcal A)$ be a Riemann surface with underlying space $S$, and let $h$ be a Riemannian metric on $S$ which is
conformal in the complex charts of $\mathcal A$. Then through the canonical identification between $TS$ and $K^{-1}$, $h$ corresponds
to a Hermitian product over $K^{-1}$.

Being $K^{-1}$ a holomorphic bundle over $S$, there is a Chern connection $D$ on $K^{-1}$ associated with $h$.
As in complex dimension $1$ any Hermitian form is K\"ahler, the connection $D$ corresponds to
the Levi Civita connection of $h$ (regarded as a real Riemannian structure on $S$), through the identification $K^{-1}\cong TS$
(see Proposition 4.A.7 of \cite{huybrechts2005complex}).

Now if in a conformal coordinate $z$ the metric is of the form $h=e^{2\eta}|dz|^2$, then the connection form
of $D$ is simply 
\begin{equation}\label{eq:chern}
\omega=2\partial\eta~, 
\end{equation}
where $\partial\eta=\frac{\partial\eta}{\partial z}dz$.
Thus a simple computation shows that the connection form of $\nabla$ with respect to the real conformal frame $\partial_x, \partial_y$ is
\begin{equation}\label{eq:levicivita}
      A= d\eta\otimes\Id-(d\eta\circ J)\otimes J\,.
\end{equation}
%where $J$ is the almost complex structure associated with $h$.
Finally, as the $\bar\partial$-operator on $K^{-1}$ corresponds to the $(0,1)$-part of $D$,
holomorphic sections of $K^{-1}$ correspond to vector fields $Y$ such that $(\nabla Y)$ commutes with $J$.
This means that $\mathbf S\nabla Y$ must be a multiple of the identity, i.e.
$\nabla Y=\lambda\Id+\mu J$ for some functions $\lambda$ and $\mu$.
\end{remark}

\begin{proof}
The proof is based on the computation in Proposition \ref{pr:cod-dec}.
%We will prove that the following diagram is commutative
%\begin{equation}\
%\begin{CD}
%\cC(S,h) @>\Lambda\circ \delta>>H^1_{\tiny\Ad\circ\tiny\hol}(\pi_1(S),\so(2,1))\\
%@V\Psi VV                                       @A d\Hol AA\\
%T_{[X]}\cT(S) @>>-\mathcal J > T_X\cT(S)
%\end{CD}
%\end{equation}
%where $\cJ$ is the complex structure on $\cT(S)$.
%
Using the decomposition  $b=b_q+\Hess u-u\Id$, we see that $Jb_0=Jb_q+\mathbf{S}\nabla(J\Hess u)=
Jb_q+\bar\partial (J\grad u)$,
where we are using that the $\bar\partial$-operator on $K^{-1}$ coincides with the anti linear part of $\nabla X$ 
(under the identification $K^{-1}=TS$). In particular $[Jb_0]=[Jb_q]$ as elements of $T_X\cT(S)$.
As  $d\Hol([Jb_0])=\Lambda \delta(b_q)=\Lambda \delta(b)$ by Proposition \ref{pr:holclos} we conclude that the diagram is commutative.
\end{proof}

\begin{corx} \label{corollary same spacetimes}
 Two embedding data $(I,s)$ and $(I', s')$ correspond to Cauchy surfaces contained in the same flat globally hyperbolic space-time
if and only if 
\begin{itemize}
\item the third fundamental forms $h$ and $h'$ are isotopic;
\item the infinitesimal variation of $h$ induced by $b$ is Teichm\"uller equivalent to the infinitesimal variation of $h'$ induced by $b'$.
\end{itemize}
In particular the map induces to the quotient a bijective map
\[
\bar\Psi: \cS_+(S)\to T\cT(S)\,.
\] 
\end{corx}
\begin{proof}
As it is known (\cite{Goldman}) that $$d\Hol: T_X\cT(S)\to T_{[\tiny\hol]}\left(\cR(\pi_1(S), \SO(2,1))/\!\!/ \SO(2,1)\right)$$ is an isomorphism, 
and that the holonomy distinguishes maximal globally hyperbolic flat space-times with compact surface (\cite{Mess}), the result follows by the commutative
of \eqref{eq:diag-comm}.
The only  point to check
is that the restriction of the map $\Lambda\circ\delta:\cC(S, h)\to H^1_{\tiny\hol}(\pi_1(S), \so(2,1))$  on the subset of positive Codazzi tensors $\cC^+(S, h)=\left\{b\in\cC(S, h):b>0\right\}$ is surjective.
This follows from the fact that $\delta:\cC(S, h)\to H^1_{\tiny\hol}(\pi_1(S), \mathbb R^{2,1})$ is surjective and that
for any smooth Codazzi tensor we can find a constant $M$ such that $b+M\Id$ is positive.
\end{proof}

\subsection{Symplectic forms} \label{Symplectic forms}
We fix a hyperbolic metric $h$ on a closed surface $S$ and  use the same notation as in the previous section.
We consider the Goldman symplectic form $\omega^B$ on $H^1_{\tiny\Ad\circ\tiny\hol}(\pi_1(S), \so(2,1))$, which depends on the choice
of a non-degenerate $\Ad$-invariant symmetric form $B$ on $\so(2,1)$, and the Weil-Petersson symplectic form $\omega_{\WP}$ on $T\cT(S)$. In \cite{Goldman}, Goldman proved that $d\Hol:(T_{[X_h]}\cT(S), \omega_{W\!P})\to (H^1_{\tiny\hol}(\pi_1(S), \so(2,1)), \omega^B)$ is symplectic
up to a multiplicative factor (which depends on the choice of $B$).

In this section we give a different proof of this fact. We will compute in a simple way the pull-back of the forms $\omega^B$ and 
$\omega_{\WP}$ respectively through the maps $\Lambda\circ\delta:\cC(S,h)\to H^1_{\tiny\hol}(\pi_1(S), \so(2,1))$ and 
$\Psi:\cC(S,h)\to T_{X_h}\cT(S)$ introduced in the previous section and show that  they  coincide up to a factor.
The thesis will directly follow by the commutativity of  \eqref{eq:diag-comm}.

%Let $\kappa$ be the Killing form on $\so(2,1)$. A preliminary simple computation is the following
%\begin{lemma}
%$\kappa(\Lambda t, \Lambda t')=2 \langle t, t'\rangle$
%\end{lemma}
%\begin{proof}
%As
%\[
%[\Lambda t, \Lambda s]=\Lambda_{t\boxtimes s}
%\]
%the adjoint representation of $\so(2,1)$ is isomorphic to the natural representation, so
%\[
%   \kappa(\Lambda t, \Lambda_{t'})=\tr(\Lambda_t\circ \Lambda_{t'})~.
%\]
%An explicit computation shows that
%\[
%\Lambda_t\circ\lambda_{t'}(s)=\langle t, t'\rangle s-\langle t, s\rangle t'\,,
%\]
%and taking the trace we get
%\[
 %\]
 In the definition of the Goldman form $\omega^B$ given in \cite{Goldman}, the model $\psl(2,\R)$ of the algebra $\so(2,1)$  is considered, and in that model the following $\Ad$-invariant form is taken:
\[
    B(X, Y)=\tr (XY)\,.
\]
With this choice we can compute $B$ in terms of the Minkowski product on the bundle $F$.
\begin{lemma}\label{lm:BF}
Let $B$ be the form on $\so(2,1)$ obtained by identifying $\so(2,1)$ with $\psl(2,\R)$. Then
$B(\Lambda(t), \Lambda(s))=(1/2)\langle t,s\rangle$.
\end{lemma}
\begin{proof}
As the space of $\Ad$-invariant symmetric forms on $\so(2,1)$ is $1$-dimensional, there exists
$\lambda_0$ such that $B(\Lambda(t), \Lambda(s))=\lambda_0\langle t,s\rangle$.

In order to compute $\lambda_0$, let us consider an isometry $\Gamma$ between $\Hyp^2$ and the upper half-space $H^+$, sending
$t_0=(1,0,0)$ to $i$. As $\Lambda(t_0)$ is a generator of the elliptic group around $t_0$, we have that $\Gamma\Lambda(t_0)\Gamma^{-1}$
is a multiple of the matrix
\[
X_0=\begin{pmatrix}
0 & -1\\
1 & 0
\end{pmatrix}\,.
\]
Using that $\exp(t\Lambda(t_0))$ is $2\pi$ periodic, whereas $\exp(tX)$ is $\pi$-periodic in $PSL(2,\R)$ we deduce that
$\Gamma\Lambda(t_0)\Gamma^{-1}=\pm (1/2)X_0$, so 
$B(\Lambda(t_0), \Lambda(t_0))=-1/2=1/2\langle t_0, t_0\rangle$ and $\lambda_0=1/2$.
\end{proof}

\noindent To define the symplectic form $\omega^B$, $H^1_{\tiny\Ad\circ\tiny\hol}(\pi_1(S), \so(2,1))$ is identified with $H^1_{\mathrm{dR}}(S, F_{\so(2,1)})$. Then we set
\[
    \omega^B(\sigma, \sigma')=\int_S B(\sigma\wedge\sigma')\,,
\]
where $B(\sigma\wedge\sigma')$ is obtained by alternating the real $2$-form $B(\sigma(\bullet), \sigma'(\bullet))$.

\noindent Analogously  a symplectic form $\omega^F$ is defined on $H^1_{\mathrm{dR}}(S, F)$ by setting
\[
\omega^F( s, s')=\int_S\langle s\wedge s'\rangle\,,
\]
where $s,s'$ are $F$-valued closed $1$-forms.

By Lemma \ref{lm:BF} one gets that $\omega^B(\Lambda(s), \Lambda(s'))=(1/2)\omega^F(s, s')$.

\begin{prop}\label{pr:omegaB}
Let $\delta:\cC(S,h)\to H^1_{dR}(S, F)$ be the connecting homomorphism.
Then
 \begin{equation}\label{eq:omegaF}
 \omega^F(\delta(b), \delta(b'))=\frac{1}{2}\int_S \tr (Jbb')\da_h\,,
 \end{equation}
or analogously
 \begin{equation}\label{eq:omegaB}
 \omega^B(\Lambda(\delta(b)), \Lambda(\delta(b')))=\frac{1}{4}\int_S \tr (Jbb')\da_h\,.
 \end{equation}
\end{prop}
\begin{proof}
By Proposition \ref{pr:216}, $\delta(b)=[\bfs{\iota}_*b]$ and $\delta(b')=[\bfs{\iota}_*b']$, where
$\bfs{\iota}_*:TS\to F$ is the inclusion induced by the developing section.
In particular if $\{e_1, e_2\}$ is an orthonormal frame on $S$ we have
\begin{align*}
\langle( \delta b) \wedge(\delta b')\rangle= &\frac{1}{2}\left(\langle \bfs{\iota}_*b(e_1), \bfs{\iota}_*b'(e_2)\rangle-
\langle \bfs{\iota}_*b(e_2), \bfs{\iota}_*b'(e_1)\rangle\right)\\ = &
\frac{1}{2}\left(h(b e_1, b'e_2)-h(b e_2, b' e_1)\right)\\ = &
\frac{1}{2}\left(h(be_1, b'J e_1)+h(be_2, b'Je_2)\right)\\ = &
\frac{1}{2}\tr(Jbb')\,.
\end{align*}
Formula \eqref{eq:omegaF} immediately follows.
\end{proof}

\begin{remark}\label{rk:wd}
A consequence of the previous proposition is that if $b$ and $b'$ are Codazzi operators, then 
\[
    \int_S \tr(Jbb')\da_h=0
\]
whenever one of the two factors is of the form $\Hess u-u\Id$.
This could also be deduced by a direct computation.
\end{remark}

We now consider the computation of the Weil-Petersson symplectic form $\omega_{\WP}$.
In conformal coordinates, if $q(z)=f(z)dz^2$, $q'(z)=g(z)dz^2$ and 
$h(z)=e^{2\eta}|dz|^2$, then the $2$-form
\[
       \frac{f\bar g}{e^{2\eta}}dx\wedge dy
\]
is independent of the coordinates. Recall that the Codazzi tensor $b_q$ is defined as the $h^{-1}\Re(q)$.  
A simple computation shows that in the conformal basis $\{\partial_x, \partial_y\}$ the operator $b_q$ is represented by the matrix
\begin{equation} \label{formulahqdcodazzi}
 e^{-2\eta}\begin{pmatrix} \Re (f) & -\Im (f)\\ -\Im (f) & -\Re (f)\end{pmatrix}\,.
 \end{equation}
 The Weil-Petersson product is defined as 
\[
g_{\WP}(q,q')=\int_S  \frac{f\bar g}{e^{2\eta}}dx\wedge dy~.
\]
A local computation, using expression \eqref{formulahqdcodazzi} of the matrices $b_q$ and $b_{q'}$ shows that
\begin{equation}\label{eq:WP}
   g_{\WP}(q,q')=\frac{1}{2}\int_S\tr(b_{q}b_{q'})\da_h\, +\, \frac{i}{2}\int_{S}\tr(Jb_{q}b_{q'})\da_h\,.
\end{equation}
This expression is at the heart of the following computation.

\begin{prop}\label{pr:omegaWP}
Given $b,b'\in\cC(S, h)$, the following formula holds:
\begin{equation}\label{eq:omegaWP}
   \omega_{\WP}(\Psi(b), \Psi(b'))=2 \int_S\tr (Jbb')\da_h\,.
\end{equation}
\end{prop}
\begin{proof}
Let  $q$ be a holomoprhic quadratic differential.
By a local computation, using Equations \eqref{formulabeltramicodazzi} and \eqref{formulahqdcodazzi} which relate the expression in complex charts of the Beltrami differential $\Psi(b)=[b_0]$ and the holomorphic quadratic $q$ differential to their expression as operators on the real tangent space $TS$, the contraction form of $\Psi(b)$ and $q$ equals
$$q\bullet \Psi(b)=-(\tr(Jb_0b_q)+i\tr(b_0b_q))\da_h\,.$$
%The computation is done recalling that if in a conformal chart $b_0=\begin{pmatrix} \alpha & \beta \\
%\beta &-\alpha\end{pmatrix}$ and $q=f(z)dz^2$, 
%the corresponding Beltrami differential is $\Psi(b)=(\alpha+i\beta)\frac{d\bar z}{dz}$
%whereas $b_q$ is given by \eqref{}.
It follows that
\begin{equation}\label{eq:cont}
    \langle q,\Psi(b)\rangle=-\int_S(\tr(Jb_0b_q)+i\tr(b_0b_q))\da_h\,.
\end{equation}

Comparing this equation with \eqref{eq:WP} we see that the antilinear
map $\mathcal K^2(S)\to T_{X_h}\cT(S)$ defined  by the Weil-Petersson product is
\[
    q\to \Psi\left(\frac{Jb_q}{2}\right)\,.
\]
So we have dually  that
\[
  g_{\WP}(\Psi(b_q), \Psi(b_q'))=4 g_{\WP}(\Psi(Jb_q/2), \Psi(Jb_q'/2))=4g_{\WP}(q,q')
\]
and using \eqref{eq:WP} we see that
\[
   \omega_{\WP}(\Psi(b_q), \Psi(b_q'))=2 \int_S\tr (Jb_qb'_q)\da_h\,.
\]
To get the formula in general, notice that, if $b=b_q+\hess u-u\Id$ and $b'=b_{q'}+\hess u'-u'\Id$, 

\begin{align*}
\omega_{\WP}(\Psi(b), \Psi(b'))=&\omega_{\WP}(\Psi(b_q), \Psi(b_{q'}))\\= &
2\int_S\tr(Jb_qb_{q'})\da_h\\=&
2\int_S\tr(Jbb')\da_h\,.
\end{align*}
where the last equality holds by Remark \ref{rk:wd}.
\end{proof}

\begin{cor} \label{corollarytheoremgoldman}
The Weil-Petersson symplectic form $\omega_{\WP}$ and the Goldman symplectic form $\omega^{B}$ are related by:
$$\Hol^*(\omega^B)=\frac{1}{8}\omega_{\WP}\,.$$
\end{cor}
The proof, which is a new proof of Goldman's Theorem presented in \cite{Goldman}, follows directly by the commutativity of diagram \eqref{eq:diag-comm} and formulae \eqref{eq:omegaB} and \eqref{eq:omegaWP}.

%----------------------------------------------sezione 4---------------------------------

%\input bs-sec4-nuovo.tex

%\input bs-sec4old.tex   vecchia versione, non definitiva

\section{The case of punctured and singular surfaces} \label{punctured surfaces}

\subsection{Hyperbolic metrics with cone singularities}\label{ss:conprel}

We now consider more deeply the case of surfaces with cone singularities.
Let us fix a closed surface $S$ of genus $g$ and a finite set of points $\puct=\{p_1,\ldots, p_k\}$ on $S$.
%We will always make the assumption that $2-2g-k<0$.
Finally fix $\theta_1,\ldots\theta_k\in (0,2\pi)$.
Recall by \cite{Troyanov} that a singular metric on $S$ with cone angles $\theta_i$ at $p_i$
is a smooth metric $h$ on $S\setminus \puct$ such that for any $i=1\ldots,k$.
there is a conformal coordinate $z$ in a neighborhood $U_i$ of $p_i$ such that
$h|_{U_i\setminus\{p_i\}}=|z|^{2\beta_i}e^{2\xi_i(z)}|dz|^2$, where
$\beta_i=\frac{\theta_i}{2\pi}-1\in(-1,0)$ and $\xi_i$ is a continuous function on $U_i$.
We will denote by $\bfs{\beta}=\sum \beta_ip_i$ the divisor associated with the metric $h$.
We will always assume $\chi(S, \bfs\beta):=\chi(S)+\sum\beta_i<0$. 

%We will moreover suppose that that the curvature $K_g$ is bounded and extends to a continuous function
%on $S$. An important consequence of this fact is that the conformal factor $f$ 
%extends to a $C^{1,\alpha}$ function at the singular point (see \cite{troyanov}).

%It is convenient to recall that around each puncture there is another convenient system of coordinates.
%In fact in \cite{} it is shown that there are polar coordinates  in a neighborhood of a puncture
%so that the metric in this coordinates is written in the simple form
%\begin{equation}\label{forma metrica singolare}
%h= dr^2+r^2 h(r,\phi)^2 d\phi^2
%\end{equation}
%where $(r,\phi)\in(0,\epsilon)\times [0,2\pi]$, 
%and $h$ extends to $h(0,\phi)=\frac{\theta_i}{2\pi}$ at the origin, for each $\phi\in[0,2\pi]$.

\begin{example} \label{examples cone singular metrics}
The local model of a hyperbolic metric with cone singularity of angle $\theta_0$
is obtained by taking a wedge in $\Hyp^2$ of angle $\theta_0$ and glueing its edges by a
rotation. See Figure \ref{modelconepoint} below, in the Poincar\'e disc model.

\begin{figure}[htbp]
\centering
\includegraphics[height=4cm]{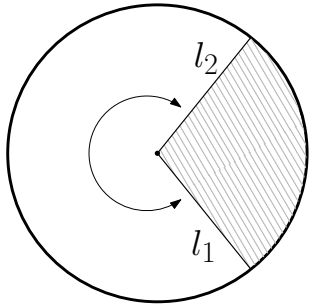}
\caption{The model of a hyperbolic surface with a cone point. The wedge in $\Hyp^2$ is the intersection of the half-planes bounded by two geodesic $l_1,l_2$. The edges are glued by a rotation fixing $l_1\cap l_2$.} \label{modelconepoint}
\end{figure}

More formally one can consider the universal cover $H$ of $\Hyp^2\setminus\{p\}$.
Its isometry group is the universal cover of the stabilizer of $p$ in $\isom(\Hyp^2)$.
Indeed we can consider on $H\cong (0,+\infty)\times\R$ the coordinates $(r,\theta)$ obtained by pulling back
the polar coordinates of $\Hyp^2$ centered at $p$.
If we take $p=(0,0,1)$ the projection map is simply
\[
    d(r,\theta)=(\sh r\cos\theta, \sh r\sin\theta,\ch r)~.
\]
We have $d^*(h_{\Hyp^2})=dr^2+\sh^2(r)d\theta^2$. The isometry group of $H$
coincides with the group of horizontal translations $\tau_{\theta_0}(r,\theta)=(r, \theta_0+\theta)$.
For a fixed $\theta_0$, the completion $\Hyp_{\theta_0}$ of the quotient of $H$ by the group generated by $\tau_{\theta_0}$
is the model of a hyperbolic surface with cone singularity $\theta_0$.

According to the definition given in \cite{mazzrubinKE}, 
\emph{polar coordinates} on $\Hyp_{\theta_0}$ are obtained by taking $r$ and $\phi=(2\pi/\theta_0)\theta\in[0,2\pi]$ 
and the metric takes the form
 $$dr^2+\left(\frac{\theta_0}{2\pi}\right)^2(\sinh r)^2 d\phi^2.$$

To construct a conformal coordinate on $\Hyp_{\theta_0}$ it is convenient 
to consider the holomorphic covering of $\pi:H\to \Hyp^2\setminus\{p\}$.
Taking the Poincar\'e model of the hyperbolic plane centered at $p$, we can realize $H$ as the upper half plane with projection 
$\pi(w)=\exp(iw)$. In this model the pull-back metric is simply
\[
     \pi^*\left(\frac{4|dz|^2}{(1-|z|^2)}\right)=\frac{4 e^{-2y}}{(1-e^{-2y})^2}|dw|^2\,,
\]
and isometries are horizontal translations.
It can be readily shown that the dependence of coordinates $(r, \theta)$ on the conformal coordinate $w=x+iy$ is of the form  
$r=\log\frac{1+e^{-y}}{1-e^{-y}},\ \theta=x$.
In particular the map $\tau_{\theta_0}$ in this model is still of the form $\tau_{\theta_0}(w)=w+\theta_0$.

Now we have a natural holomorphic projection 
$\pi_{\theta_0}:H\to \D\setminus\{0\}$ given by $\pi_{\theta_0}(w)=\exp(ik w)$ where $k=2\pi/\theta_0$. 
The automorphism group  of $\pi_{\theta_0}$ is generated by the translation  $\tau_{\theta_0}$.
So a conformal metric $h_{\theta_0}$ is induced on $\D\setminus\{0\}$ that makes $\D$ a model of $\Hyp_{\theta_0}$.
Putting $z=\exp(ikw)$ and $w=x+iy$ it turns out that $|z|=e^{-ky}$ so $e^{-y}=|z|^{1/k}$ and as $dz=ik\exp{(ikw)}dw$ one gets that
\begin{equation}\label{eq:hypconf}
   h_{\theta_0}=\frac{4}{k^2}\frac{|z|^{2(1/k-1)}}{(1-|z|^{2/k})^2}|dz|^2=4(1+\beta)^2\frac{|z|^{2\beta}}{(1-|z|^{2(1+\beta)})^2}|dz|^2\,,
\end{equation}
where we have put $\beta=\theta_0/2\pi-1$.

%Finally notice that in this case putting $h_{\theta_0}=|w|^{2\beta}e^{2\xi}|dw|^2$, it turns out that 
%\begin{equation}\label{eq:hypconf}
%\xi=-\frac{1}{2}\log(\beta+1)-log(1+|z|^{2(1+\beta))
%\end{equation}

Analogously, one can construct the models of Euclidean and spherical
singular points.
In polar coordinates the
flat cone metric takes the form
$$d\rho^2+  \left( \frac{\theta_0}{2\pi} \right) ^2 \rho^2  d\phi^2~.$$
In the conformal coordinate the flat metric is simply
$|z|^{2\beta}|dz|^2$. 
\end{example}

We mainly consider the case where $h$ is a hyperbolic metric with cone singularities.
In particular we denote by $\cH(S,\bfs{\beta})$ the set of singular hyperbolic metrics on $S$ with
divisor $\bfs{\beta}$.
We will endow a neighborhood of a cone point of a hyperbolic surface with a Euclidean metric with the same cone singularity,
that we call the \emph{Klein Euclidean metric} associated with $h$.
The construction goes as follows.

Take a point $p\in\Hyp^2$ and consider the radial projection of $\Hyp^2$ to the affine plane $P$ tangent to $\Hyp^2$ at $p$, 
\[
   \pi:\Hyp^2\to P\,.
\]
The map $\pi$ is frequently used to construct  the  Klein model of $\Hyp^2$. 
The pull-back of the Euclidean metric $g_P$ of $P$ to $\Hyp^2$ is invariant by the whole stabilizer of $p$.
Hence the pull-back of this metric on the universal cover $H$ of $\Hyp^2\setminus\{p\}$ is a Euclidean metric
invariant by the isometry group of $H$. 
The latter metric thus projects to a Euclidean metric $g_K$ on $\Hyp_{\theta_0}$, still having a cone singularity of angle $\theta_0$ at the cone point.
We observe that $h$ and $g_K$ are bi-Lipschitz metrics in a neighborhood of the singular point.

\begin{remark}
On a surface with cone singularity, 
the metric $g_K$ is defined only in a regular  neighborhood of a cone point.
\end{remark}

One of the reason we are interested in this metric is that there is a useful relation between the Hessian computed with respect to the metric
$h$ and the Hessian computed with respect to the metric $g_K$.

\begin{lemma}[Lemma 2.8 of \cite{bonfill}]\label{lm:kleincorr}
Let $S$ be a hyperbolic surface with a cone point at $p$. 
Let $u$ be a function defined in a neighborhood of $p$ and $r$ the the hyperbolic distance from $p$.
Consider the function $\bar u=(\ch r)^{-1}u$, then
\begin{equation}\label{eq:hasseucl}
  D^2\bar u(\bullet, \bullet)=(\ch r)^{-1}h((\Hess u-u\id)\bullet, \bullet)\,,
\end{equation} 
 where $D^2\bar u$ is the Euclidean Hessian of $\bar u$ for the metric $g_K$, considered as a bilinear form.
\end{lemma}

The computation in \cite{bonfill} was done only locally in $\Hyp^2$, but as the result is of local nature it works
also in the neighborhood of a cone point.

\subsection{Codazzi tensors on a hyperbolic surface with cone singularities}

A Codazzi operator on the singular hyperbolic surface $(S, h)$ is a smooth  self-adjoint operator $b$ on $S\setminus\puct$ which
solves the Codazzi equation $d^{\nabla}_h b=0$. 
It is often convenient to require some regularity of $b$ around the singularity.

We basically consider two classes of regularity.
We say $b$ is bounded if the eigenvalues of $b$ are uniformly bounded on $S$, and denote by
$\cC_\infty(S,h)$ the space of bounded Codazzi operators.
We say that $b$ is  of class $L^2$ (with respect to $h$)
 if $$\int_S \tr (b^2)\da_h<+\infty\,,$$ where $\da_h$ is the area form of the singular metric.
 We denote by $\cC_2(S,h)$ the space of $L^2$-Codazzi tensors. Since $\tr (b^2)=||b||^2$, $b$ is in $\cC_2(S,h)$ if and only if $||b||\in L^2(S,h)$.
 Notice that if $u\in \cun(S\setminus\puct)$, then the operator $b=\hess u-u\id$ is a Codazzi operator
 of the singular surface. We have that $b\in \cC_2(S,h)$ if $u$ and $||\hess u||$ are in $L^2(S,h)$.
 
 As in the general case, given a holomorphic  quadratic differential $q$ on $S\setminus \puct$, 
 the self-adjoint operator $b_q$ defined
 by  $\Re q(u,v)=h(b_q(v), w)$ is a traceless Codazzi operator of the singular surface $(S, h)$.
 The regularity of $b_q$ close to the singular points can be easily understood in terms of the singularity of $q$.
 In particular we have
 \begin{prop} \label{integrability harmonic codazzi}
 The Codazzi operator   $b_q$ is in $\cC_2(S,h)$ if and only if $q$ has at worst simple poles at the singularities.
 Moreover if $\theta_i\leq \pi$ then  $b_q$ is bounded around $p_i$, and if $\theta_i<\pi$ then $b_q$ continuously extends at $p_i$.
 \end{prop}
 
For the sake of completeness we prove an elementary Lemma, that will be used in the proof of Proposition \ref{integrability harmonic codazzi}.
\begin{lemma} \label{lemma integrability}
Let $D$ be a disc in $\mathbb C$ centered at $0$ and $f$ be a holomorphic function
on $D\setminus\{0\}$. If $|z|^a|f|^p$ is integrable for some $a\in(0,2)$  and $p\in[1,2]$, 
then $f$ has at worst a pole of order $3$ at $0$.
If moreover $2p-a\geq 2$, then the pole at $0$ is at most simple.
\end{lemma}
 \begin{proof}
 First we notice that with our assumption $\hat f(z)=z^2f(z)\in L^1(D, |dz|^2)$. 
 As the function $|\hat f(z)|$ is subharmonic, the value at a point $z$ of $|\hat f|$ is estimated from above
 by the mean value of $|\hat f|$ on the ball centered at $z$ with radius $|z|$
 \[
   |\hat f(z)|\leq\frac{1}{2\pi i|z|^2}\int_{B(z, |z|)}|\hat f(\zeta)|d\zeta\wedge d\bar\zeta~.
 \]
 As the integral is estimated by the norm $L^1$ of $\hat f$ we get that
 $|z|^2 |\hat f(z)|\leq C$. Thus $\hat f$ has a pole of order at worst $2$. As $1/z\in L^1(D,|dz|^2)$, whereas
 $1/z^2\notin L^1(D,|dz|^2)$, the pole in $0$ cannot be of order 2.
 
 This implies that $f$ has at most a pole of order $3$.
 Suppose $f$ has a pole of order $2$ or $3$, then
close to $0$ we have $|f(z)|>C|z|^{-2}$ where $C$ is some positive constant.
Thus $|z|^a|f(z)|^{p}\geq C^p |z|^{a-2p}$, that implies $2p-a<2$. 
  \end{proof}

 \begin{proof}[Proof of Proposition \ref{integrability harmonic codazzi}]
 The problem is local around the punctures. Let us fix a conformal coordinate $z$  in a neighborhood
 of a puncture $p_i$, so that   the metric takes the form
 $h=e^{2\xi}|z|^{2\beta}|dz|^2$ where $\xi$ is a bounded function, 
 whereas $q=f(z)dz^2$ where $f$ is a holomorphic function on the punctured disc $\{z\,|\,0<|z|<\epsilon\}$. 
Using the expression \eqref{formulahqdcodazzi} for the operator $b_q$ in real coordinates, we get
 \[
 ||b_q||^2=\tr (b_q^2)=2e^{-4\xi}|z|^{-4\beta}|f|^2\,.
 \]
 On the other hand the area form is $\da_h=e^{2\xi}|z|^{2\beta}dx\wedge dy$,
 so the problem is reduced to the integrability of the function
 $|z|^{-2\beta}|f(z)|^2$. As $-2\beta\in(0,2)$, Lemma \ref{lemma integrability} implies that this happens if and only if
  $f$ has at worst a simple pole at $0$.
  
 The same computation shows that 
 $||b_q||^2(z)<C|z|^{-4\beta-2}$.
 In particular if $\beta\in(-1,-1/2]$ (that is if the cone angle is $\theta\leq\pi$)
 the operator $b_q$ is bounded around $p_i$, whereas if $\beta<-1/2$ the operators 
 continuously extends at $p_i$ with $b_q(p_i)=0$.
   \end{proof}
 
 We want to prove now that $\delta b$ is a continuous function of $b$ with respect to the $L^2$-distance.
 \begin{prop}\label{pr:contdelta}
The map $\delta:\cC_2(S, h)\to H^1_{\tiny\hol}(\pi_1(S), \R^{2,1})$ is continuous for the $L^2$-distance on $\cC_2(S,h)$.
\end{prop}
We antepone to the proof an elementary lemma.

\begin{lemma}\label{sublemma}
Let $f_n$ be a sequence of smooth functions defined on a planar disc $U$ of radius $R$ 
such that 
\begin{itemize}
\item $f_n(0)\to 0$ and $df_n(0)\to 0$ as $n\to+\infty$;
\item $||D^2f_n||_{L^2(U)}\to 0$.
\end{itemize}
then $df_n\to 0$ and $f_n\to 0$ in $L^2(U)$.
\end{lemma}
\begin{proof}
By Taylor expansion  along the path $\gamma(t)=ty$, we write
\[
  df_n(y)=df_n(0)+\int_0^1 D^2 (f_n)_{ty}(y,\bullet)dt\,,
\]
so by applying Cauchy-Schwartz inequality we deduce that
\[
   ||df_n(y)-df_n(0)||^2<C\int_0^1||D^2 (f_n)_{ty}||^2dt\,,
\]
where $C$ is a constant depending on $R$.
By integrating, using Fubini-Tonelli Theorem we get
\[
   \int_U||df_n(y)- df_n(0)||^2<C||D^2 f_n||_{L^2(U)}^2~.
\]
Using that $df_n(0)\to 0$ we conclude that $df_n\to 0$ in $L^2(U)$.
A similar computation shows that $f_n\to 0$ in $L^2(U)$.
\end{proof}

\begin{proof}[Proof of Proposition \ref{pr:contdelta}]
Let $b_n$ be a sequence of Codazzi operators on $(S, h)$ such that
$||b_n||_{L^2(h)}\to 0$. We have to prove that $\delta b_n\to 0$.

 Lifting $b_n$ on the universal covering we have a family of functions $u_n:\tilde S\to \R$ such that
\[
  \tilde b_n=\hess u_n-u_n\id\,,
\]
and $u_n(x)-u_n(\alpha^{-1}x)=\langle \dev(x), \delta(b_n)_\alpha\rangle$ for every $\alpha\in\pi_1(S)$.
We claim that we can choose $u_n$ such that $u_n(x)\to 0$ almost everywhere on $\tilde S$.
The claim immediately implies that $\delta(b_n)_\alpha\to 0$ for every $\alpha\in\pi_1(S)$.

In order to prove the claim we fix a point $x_0\in\tilde S$ and normalize $u_n$ so that $u_n(x_0)=0$ and
$du_n(x_0)=0$ (this is always possible by adding some linear function to $u_n$).

Consider now the Klein Euclidean metric $g_K$ on $\tilde S$ obtained by composing the developing map with the projection of 
 $\Hyp^2$ to the tangent space $T_{\dev(x_0)}\Hyp^2$. (Notice that covering transformations are not isometries for $g_K$.)
Take the function $\bar u_n=(\ch r)^{-1} u_n$ where $r(x)$ is the hyperbolic  distance between $\dev(x)$ and $\dev (x_0)$.
Let $U$ be any bounded subset of $\tilde S$.
Combining the hypothesis on $b_n$  and Lemma \ref{lm:kleincorr}, we have that $||D^2\bar u_n||_{h}\to 0$ in $L^2(U, h)$.
As $h$ and $g_K$ are bi-Lipschitz over $U$ we have that $||D^2 \bar u_n||_{L^2(U,g_K)}\to 0$. % on every bounded subset $U$ of $\tilde S$.

By Lemma \ref{sublemma} there is a neighborhood $U_0$ of $x_0$ where $\bar u_n\to 0$ and $d\bar u_n\to 0$.
So the set
\[
    \Omega=\{x\in\tilde S : \exists\, U \textrm{ neighborhood of }x \textrm{ such that } ||\bar u_n||_{L^2(U, g_K)}\to 0 \textrm{, and }  ||d\bar u_n||_{L^2(U, g_K)}\to 0\}
\]
 is open and non-empty.
 Again Lemma \ref{sublemma} implies that this set is closed so $\Omega=\tilde S$.
 The claim easily follows.
\end{proof}

 We are now ready to prove  that the decomposition of Codazzi tensors for closed surfaces described in Proposition \ref{pr:cod-dec}
 holds for singular surfaces, when the correct behavior around the singularity is considered.
 For a singular metric $h$, we need to introduce the Sobolev spaces $W^{k,p}(h)$ of functions
 in $L^p(h)$ whose distributional derivatives  up to order $k$ (computed with the Levi-Civita connection of $h$)
 lie in $L^p(h)$.

\begin{prop}\label{pr:conedecomp}
Given $b\in\cC_2(S,h)$, a holomorphic quadratic differential $q$  and a  function $u$ are uniquely determined so that
\begin{itemize}
\item   $b=b_q+\Hess u-u\id$;
\item  $q$ has at  worst simple poles at $\puct$;
\item   $u\in C^\infty(S\setminus\puct)\cap W^{2,2}(h)$.%\cap C^{0,\alpha}(S)\cap  W^{2,2}(h)$.
\end{itemize}
Such decomposition is orthogonal, in the sense that
\[
 ||b||^2_{L^2}=||b_q||^2_{L^2}+||\Hess u-u\id||^2_{L^2}~.
\]
Moreover, if $\theta_i<\pi$ and  $b$ is bounded then $u\in W^{2,\infty}(h)$.
\end{prop}

Again we antepone an elementary Lemma of Euclidean geometry.
 
  \begin{lemma}\label{lm:tech2}
Let $(U, g)$ be a closed disc equipped with a Euclidean metric with a cone angle at $p_0$.
Let $f$ be a smooth function on $U\setminus\{p_0\}$ such that $||D^2 f||$ is in $L^2(U, g)$ (resp. $L^\infty(U, g)$).
Then $f$ and $||\grad f||$ lie in $L^2(U, g)$ (resp. $L^\infty(U,g)$). 
\end{lemma}
\begin{proof}
Let us consider coordinates $r,\theta$ on $U$. We may suppose that $U$ coincides with the disc of radius $1$.
The Taylor expansion of $f$ along the  geodesics $c(t)=(1-(1-r)t, \theta)$  is 
\[
   f(r,\theta)=f(1,\theta)-(1-r)\langle \grad f(1,\theta), \partial_r\rangle +(1-r)^2\int_{0}^1 (1-t)D^2f_{(1-(1-r)t, \theta)}(\partial_r, \partial_r)dt~.
\]
Now the function $\hat f(r, \theta)=f(1,\theta)-(1-r)\langle \grad f(1,\theta), \partial_r\rangle$ is bounded and we can estimate
\[
(f(r, \theta)-\hat f(r, \theta))^2<\int_0^1 ||D^2f(tr, \theta)||^2dt\,.
\]
If $D^2f$ is bounded this formula shows that $f\in L^\infty(U, g)$.
By integrating the inequality above on $U$ we also see that if $||D^2f||\in L^2(U, g)$ then $f$ is in $L^2(U, g)$ as well.

Cutting $U$ along a radial geodesic we get a planar domain.
We can find on this domain two parallel orthogonal unitary fields $e_1, e_2$ and basically one has to prove
that $f_i=\langle \grad f, e_i\rangle\in L^2(U,g)$ (resp. $L^\infty(U, g)$).
This can be shown by a simple Taylor expansion as above noticing that
$\grad f_i=(D^2 f) e_i$, and  so $||\grad f_i||\in L^2(U, g)$ (resp. $L^\infty(U, g)$).
\end{proof}

\begin{proof}[Proof of Proposition \ref{pr:conedecomp}]
The proof is split in three steps:
\begin{itemize}
\item[Step $1$] We prove that if $\hess u-u\Id$ is in $\cC_2(S, h)$ (resp. $\cC_\infty(S, h)$)  then $u\in W^{2,2}(S, h)$ (resp. $W^{2,\infty}(S,h)$).
\item[Step $2$] Denote by $\cC_{\tiny{\mbox{tr}}}(S, h)$ the space of trivial Codazzi tensors in $\cC_2(S, h)$.
We prove that $\cC_{\tiny{\mbox{tr}}}(S, h)^\perp$ coincides with the space of traceless Codazzi tensors.
\item[Step $3$] We prove that the orthogonal splitting holds 
\[
    \cC_2(S, h)=\cC_{\tiny{\mbox{tr}}}(S, h)\oplus\cC_{\tiny{\mbox{tr}}}(S, h)^\perp~.
\]
(This is not completely obvious as $\cC_2(S, h)$ is not complete.)
\end{itemize}

To prove Step $1$ first suppose $\hess u-u\Id\in\cC_2(S, h)$. We check that $u$ and $||du||_h$ lie in $L^2(S, h)$. 
Notice that we only need to prove integrability of $u^2$ and $||du||^2_h$ in a neighborhood $U$ of a
puncture $p_i$. Consider the function $\bar u=(\ch r)^{-1} u$, where $r$ is the distance from the puncture. 
By Lemma \ref{lm:kleincorr}, the Hessian of $\bar u$ computed with respect to the Klein Euclidean metric $g_K$ is simply
\[
    D^2\bar u(\bullet, \bullet)=(\ch r)^{-1} h((\hess u-u\id)\bullet, \bullet)\,,
\]
so we see that $||D^2\bar u||_{g_K}$ is in $L^2(U, g_K)$. By Lemma \ref{lm:tech2} we conclude that $u$ and $||du||_{g_K}$ lie in $L^2(U, g_K)$. 
As $g_K$ and $h$ are bi-Lipschitz we conclude that $u$ and $||du||_h$ are in $L^2(U, h)$.
The proof of Step $1$ easily follows by noticing that $\hess u=(\hess u-u\id)+u\Id$.

The case where $\hess u-u\Id$ is in $\cC_\infty(S,h)$ can be proved adapting the argument above in a completely obvious way.

Let us  prove Step 2. Imposing the orthogonality with the trivial Codazzi tensor corresponding to  $u=1$ we see that
 $\cC_{\tiny{\mbox{tr}}}(S, h)^\perp$ is contained in the space of traceless Codazzi tensors.
 
 To prove the reverse inclusion  it suffices to show that 
 \[
   \int_S\tr(b_q(\hess u-u\id))\da_h=0\,.
\]
for any holomorphic quadratic differential $q$ with at worst simple poles at the punctures and for any $u\in W^{2,2}(h)$.

As $b_q$ is Codazzi, $b_q\hess u=\nabla(b_q\grad u)-\nabla_{\grad u}b_q$.
Using that $\tr (b_q)=0$, we get
\[
 \int_S\tr(b_q(\hess u-u\id))\da_h=\int_S\tr(b_q\hess u)\da_h=\int_S \Div(b_q\grad u)\da_h\,.  
\]
So if $B_r$ is a neighborhood of the singular locus formed by discs of radius $r$ we have
\begin{equation} \label{trace divergence limit}
\int_S\tr(b_q(\hess u-u\id))\da_h=\lim_{r\to 0} \int_{\partial B_r} h(b_q\grad u, \nu) d\ell_r\,,
\end{equation}
where $d\ell_r$ is the length measure element of the boundary whereas $\nu$ is the normal to the boundary pointing 
inside $B_r$.
Now we claim that there exists a sequence $r_n\to 0$ such that 
\[
 \int_{\partial B_{r_n}} ||b_q\grad u||_h d\ell_{r_n}\to 0
\]
as $n\to+\infty$.  Using this sequence in Equation \eqref{trace divergence limit} we get the result.

In order to prove the claim, notice that as $u\in W^{2,2}(h)$, $\grad u\in W^{1,2}(h)$. By Sobolev embedding,  $||\grad u||_h\in L^p(h)$ for any 
$p<+\infty$. On the other hand  as in the proof of Proposition \ref{integrability harmonic codazzi}, $||b_q||\sim |z|^{-2\beta_i}|f|$, 
where $z$ is a conformal coordinate on $U$ with $z(p_i)=0$ and we are putting $q(z)=f(z)dz^2$. 
As the area element of $h$ is $\sim |z|^{2\beta_i}dxdy$, using that $|f|<C|z|^{-1}$ we see that
$||b_q||_h$ lies in $L^{2+\epsilon}(U, h)$ for  $\epsilon<2|\beta_i|$ for $i=1,\ldots k$.

\noindent So by a standard use of H\"older estimates we get  $f=||b_q\grad u||_h$ lies in $L^2(S, h)$.
\noindent Now suppose that there exists a singular point $p$ and some number $a$ such that 
\[
    \int_{\partial B_r} f d\ell_r>a
\]
for any $r\in(0, r_0)$. By Schwarz inequality we get
\[
   a<\ell(\partial B_r)^{1/2}(\int_{\partial B_r} f^2)^{1/2}\,.
\]
As $\ell(\partial B_r)<Cr$ for some constant $C$, we have
\[
          \int_{\partial B_r} f^2>\frac{a^2}{Cr}\,.
 \]
Integrating this inequality and using Tonelli Theorem we obtain
\[
    ||f||^2_{L^2}>\int_0^{r_0} dr \int_{\partial B_r} f^2>\frac{a^2}{C}\int_0^{r_0}\frac{dr}{r}\,.
\]
which gives a contradiction.

Finally we prove Step $3$.
Considering the completion  of $\cC_2(S, h)$  we have an orthogonal decomposition
\[
    \overline{\cC_2(S, h)}=\overline{ \cC_{\tiny{\mbox{tr}}}(S, h)}\oplus \overline{\cC_{\tiny{\mbox{tr}}}(S, h)^\perp}\,,
\]
where we used  a bar to denote the completed space.
By Step $2$ and Proposition \ref{integrability harmonic codazzi} the subspace $\cC_{\tiny{\mbox{tr}}}(S, h)^\perp$ is finite dimensional, so it coincides with its completion.
This implies that the splitting above induces a splitting
\[
\cC_2(S, h)=(\overline{ \cC_{\tiny{\mbox{tr}}}(S, h)}\cap \cC_2(S, h))\oplus \cC_{\tiny{\mbox{tr}}}(S, h)^\perp\,.
\]
By Proposition  \ref{pr:contdelta} we notice that   $\cC_{\tiny{\mbox{tr}}}(S, h)=\delta^{-1}(0)$ is closed in  $\cC_2(S, h)$, so the first addend is actually
$\cC_{\tiny{\mbox{tr}}}(S, h)$, and the proof is complete.
\end{proof}

Now we want to show  that if $b\in\cC_2(S,h)$ then the cocycle $\delta b$ is trivial around all
the punctures, that means that for every peripheral loop $\alpha$ there exists 
$t_0\in\mathbb R^{2,1}$ such that  $(\delta b)_{\alpha}=\hol(\alpha) t_0-t_0$. 

\begin{remark}
As $\hol(\alpha)$ is an elliptic transformation, proving that $\delta b$ is trivial around $\alpha$ is
equivalent to proving that the vector $(\delta b)_\alpha$ is orthogonal to the axis of $\hol(\alpha)$.
The importance of this condition will be made clear in next subsection. Basically, it corresponds to being the translation part of the holonomy of a MGHF manifold with particles. Moreover, it can be checked that a cocycle trivial around the punctures, if regarded as an element in $H^1_{\tiny\Ad\circ\tiny\hol}(\pi_1(S),\so(2,1))$ by means of the isomorphism $\Lambda$, corresponds to a first-order deformation of hyperbolic metrics which preserves the cone singularity. This clarifies how Theorem \ref{theorem parametrization closed surfaces} can be extended to the case of manifolds with particles, which is the aim of Section \ref{sectionteichsing}.
\end{remark}

By Proposition \ref{pr:conedecomp}, in order to prove that $\delta b$ is trivial around the punctures, it is sufficient to consider the case $b=b_q$ where $q$ is a holomorphic quadratic differential 
with at most a simple pole at singular points.

\begin{lemma}\label{lm:triv}
Let $U$ be a neighborhood of a cone point $p$ in a hyperbolic surface of angle $\theta_0\in(0,+\infty)$ 
and let $b$ a Codazzi operator on $U$
such that $||b(x)||<C_0 r(x)^{\alpha}$ where $r(x)$ is the distance from the cone point and $\alpha$ is some fixed number
bigger than $-2$. Then $\delta b$ is trivial.

Moreover, there exists a function $u\in W^{1,2}(h)\cap C^{0,\alpha}(U)$, smooth over $U\setminus\{p\}$,
such that $b=\Hess u-u\id$.
If $\alpha>-1$, then $u$ is Lipschitz continuous around $p$.
\end{lemma}
\begin{proof}
As in Example \ref{examples cone singular metrics}, let $H$ be the universal cover of $\mathbb H^2\setminus\{x_0\}$, and $(r, \theta)$ be global coordinates
on $H$ obtained by pulling back the polar coordinates on $\Hyp^2$ centered at $x_0$. We can assume $x_0=(0,0,1)$.
The cover $d:H\to\Hyp^2$ is then of the form
\[
     d(r,\theta)=(\sh r\cos\theta, \sh r\sin\theta,\ch r)\,. 
\]
Up to shrinking $U$, we may suppose that $U$ is the quotient of the region
$\tilde U=\{(r,\theta)\,|\, r<r_0\}$ by the isometry $\tau_{\theta_0}(r,\theta)=(r, \theta+\theta_0)$.

If $\tilde b$ is the lifting of $U$, there is some function $u$ on $\tilde U$ such that
$\tilde b=\hess u-u\id$. Moreover, $t=(\delta b)_\alpha$ is such that
\[
  (u-u\circ\tau_{\theta_0}^{-1})(r,\theta)=u(r,\theta)-u(r,\theta-\theta_0)=\langle d(r,\theta), t\rangle\,.
\]
Integrating $du$ on the path $c_r(\theta)=(r, \theta)$ with $\theta\in[0,\theta_0]$ we get
\[
   \int_{c_r}du= \ch(r)\langle x_0, t\rangle +O(r)\,.
\]
So, in order to conclude it is sufficient to prove that 
\begin{equation}\label{eq:lmdu}
\int_{c_r}du\rightarrow 0\quad \textrm{as }r\rightarrow 0~.
\end{equation}

To this aim we consider the Klein Euclidean metric $g_K$ on $U$ introduced in Subsection ~\ref{ss:conprel}.
Notice that $g_K$ is equivalent to the hyperbolic metric $h$  in $U$.
In particular if $\rho$ is the Euclidean distance from $p$, we have that $\rho\sim r$.
Let $\bar u=(\ch r)^{-1} u$. By Lemma \ref{lm:kleincorr} we have $D^2\bar u(\bullet,\bullet)=(\ch r)^{-1}h((\hess u-u \id)\bullet,\bullet)$ as bilinear forms, so
we get   $||D^2 \bar u||_{g_K}\sim\rho^{\alpha}$ on $\tilde U$.
A simple integration on vertical lines
shows that $||d\bar u||_{g_K}(r,\theta)\leq  C_0+C_1\rho^{\alpha+1}<C_0+C_2 r^{\alpha+1}$
for any $\theta\in[0,\theta_0]$ and $r\in[0,r_0]$, where $C_0$, $C_1$ and $C_2$ are constant depending
on $\alpha$, $r_0$, $\sup_{0\leq \theta\leq \theta_0 }||du||(r_0, \theta)$.

In particular $|\int_{c_r}d\bar u|\leq C_3(1+r^{\alpha+1})\ell_{g_K}(c_r)\leq C_4(1+r^{\alpha+1})r$.
So if $\alpha>-2$ this integral goes to $0$ as $r\rightarrow 0$.
As $\int_{c_r}du=\ch r\int_{c_r}d\bar u$, \eqref{eq:lmdu} follows.

We conclude that $\langle x_0,t\rangle=0$, or equivalently $t$ can be decomposed as 
$t=\hol(\alpha)t_0-t_0$ for some vector $t_0\in \R^{2,1}$,
and up to adding the linear function $f(r,\theta)=\langle d(r,\theta), t_0\rangle$,
we may suppose that $u$ is $\tau_{\theta_0}$-periodic. 
Hence $u$ projects to a function on $U$, that with some abuse
we still denote $u$, such that $b=\hess u-u\id$.

By the estimate on $d\bar u$ we also deduce that $\bar u$ is uniformly continuous around
the singular point. More precisely the following estimate holds:
\begin{align*}
    |\bar u(r_1, \theta_1)-\bar u(r_2, \theta_2)|&\leq |\bar u (r_1, \theta_1)-\bar u(r_1, \theta_2)|+|\bar u(r_1, \theta_2)-\bar u(r_2, \theta_2)|
    \\ &\leq C_3(1+r^{\alpha+1})r|\theta_1-\theta_2|+C_0|r_1-r_2|+\frac{C_2}{\alpha+2}|r_1^{\alpha+2}-r_2^{\alpha+2}|~.
\end{align*}
So $\bar u$ extends to a continuous function on $U$ and the same holds for $u$.

Writing $u=\ch(r)\bar u$ we get $du=\sh(r)\bar udr+\ch(r)d\bar u$.
As $\bar u$ is bounded it results that  $||du||<C_5+C_6r^{\alpha+1}$ and this estimate shows that $||du||\in L^2(U,h)$
and $u$ is Lipschitz if $\alpha>-1$.
\end{proof}

\begin{prop}\label{pr:trivial}
Let $(U,h)$ be a disc with a hyperbolic metric with a cone singularity of angle $\theta_0\in(0,2\pi)$ at $p$.
Let $q$ be a holomorphic quadratic differential with at most a simple pole in $p$.
Then $\delta b_q$ is trivial. 
Moreover there exists a Lipschitz function $u$ over $U$ that is  smooth over
$U\setminus\{p\}$ such that
$b=\Hess u-u\id$.
\end{prop}
\begin{proof}
If $z$ is a conformal coordinate on $U$ with $z(p)=0$ 
and $r$ is the distance from the singular point, 
we know that $r\sim|z|^{\beta+1}$.
On the other hand
if $q=f(z)dz^2$ and $h=e^{2\xi}|z|^{2\beta}|dz|^2$ we have
 $||b_q||^2=e^{-4\xi}z^{-4\beta}|f(z)|^2$, so by the assumption
$||b_q||^2<C |z|^{2(-1-2\beta)}$. In particular
$||b_q||<Cr^\alpha$ with $\alpha=\frac{-1-2\beta}{1+\beta}$.
As $\alpha>-1$ for any $\beta\in (-1,0)$ we can apply Lemma \ref{lm:triv}
and conclude.
\end{proof}
 
 \begin{remark}
 If cone angles are in bigger than $2\pi$ (but different from integer multiples of $2\pi$) the same argument shows that $b_q$ is trivial around the puncture as well.
 The main difference is that the exponent $\alpha$ lies in $(-2,-1]$ so Lemma \ref{lm:triv} ensures that the function $u$
 is H\"older continuous at the puncture and $du$ is only $L^2$-integrable over $U\setminus\{p\}$.
 \end{remark}

A simple corollary of Proposition \ref{pr:trivial} is that if $q$ is a quadratic differential on $(S, h)$ with at most simple poles
at the punctures then the cohomology class of $\delta b_q$ can be expressed as $\delta(b)$
for some operator $b\in\cC_\infty(S, h)$

\begin{cor}
Let $q$ be a quadratic differential on $S$ with at most simple poles at punctures.
There exists a Lipschitz function $u$ on $S$, smooth on $S\setminus\puct$ such that
$b=b_q-(\Hess u-u\id)$ is bounded. Moreover $u$ can be chosen so that $b$ is uniformly positive definite,
i.e. $b\geq a\id$ for some $a>0$.
\end{cor}
\begin{proof}
By Proposition \ref{pr:trivial} around each puncture $p_i$ there exists a function $u_i$ such that
$b_q=\hess u_i-u_i\id$.
By a partition of the unity it is possible to construct a smooth  function $u$ such that $u$ coincides with $u_i$ 
is some smaller neighborhood of $p_i$. In particular the support of  $b'=b_q-\hess u+ u\id$  is compact in $S\setminus\puct$,
so $b'$ is bounded.

In order to get $b$ uniformly positive it is sufficient to consider the constant function $v=||b'||_{\infty}+a$.
Then $b=b'-\Hess v+v\id=b_q-\Hess (u+v)+(u+v)\id$ is uniformly positive since $b'-\Hess v+v\id=b'+v\id>a\id$.
\end{proof}

%%%%%%%%%%%%%%%%%%%%%%%%%%%%%%%

\subsection{Flat Lorentzian space-times with particles} \label{cone singularities}

We now consider maximal globally hyperbolic flat manifolds with cone singularities along timelike lines. 

\begin{defi} \label{defi spacetime singular}
We say that a Lorentzian space-time $M$ has cone singularities along timelike lines (which we call also particles) if there is a collection of lines $s_1,\ldots,s_k$ such that $M^*=M\setminus (s_1\cup\ldots\cup s_k)$ is endowed with a flat Lorentzian metric and each $s_i$ has a neighborhood isometric to a slice in $\R^{2,1}$ of angle $\theta_i<2\pi$ around a timelike geodesic, whose edges are glued by a rotation around this timelike geodesic.
\end{defi}

\begin{figure}[htbp]
\centering
\includegraphics[height=7cm]{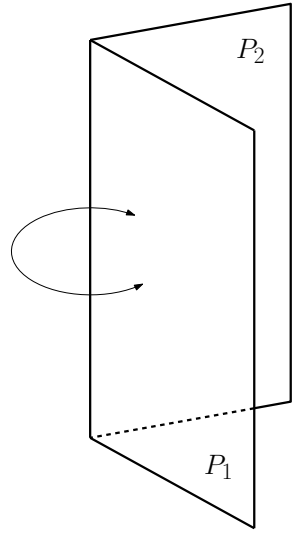}
\caption{The model of a flat particle. The slice in $\R^{2,1}$ is the intersection of the half-spaces bounded by two timelike planes $P_1,P_2$. The edges of the slice are glued by a rotation fixing $P_1\cap P_2$.} \label{modelparticle}
\end{figure}

By definition it is immediate that in a neighborhood of a point $p_i\in s_i$ there are coordinates $(z,t)\in D\times\mathbb R$, for $D$ a disc,
such that $s_i$ corresponds to the locus $\{z=0\}$ and the metric  $g$ takes the form 
\begin{equation} \label{singular metric coordinates}
g=|z|^{2\beta_i}|dz|^2-dt^2,
\end{equation} 
\noindent where
$\beta_i=\frac{\theta_i}{2\pi}-1$. %Analogously there are 
%polar coordinates $(\rho,\phi, t)$ such that 
%\begin{equation} \label{singular metric polar coordinates}
%g=-dt^2+d\rho^2+\left(\frac{\theta_i}{2\pi}\right)^2\rho^2d\phi^2.
%\end{equation}
Notice that the restriction of the metric $g$ to the slices $\{t=const\}$ are isometric Euclidean metrics with cone singularities $\theta_i$.

By definition the holonomy of a loop surrounding a  particle of angle $\theta$ is conjugated with a pure rotation of angle $\theta$, so
the translation part of the holonomy is orthogonal to the axis of rotation.

A closed embedded  surface $S\subset M$ is space-like if $S\subset M^*$ is space-like and for any point $p_0\in S\cap s_i$, in 
the coordinates $(z,t)$ defined in a neighborhood of $p_0$, $S$ is the graph of a function $f(z)$.
We will say that the surface is orthogonal to the singular locus at $p_0$ if $|f(z)-f(z_0)|=O(r^2)$ where $r$ is the intrinsic distance
on $D$ from the puncture. As $r=\frac{1}{\beta+1}|z|^{\beta+1}$ this condition is in general not equivalent to requiring that the differential of $f$ at $z_0$
(computed with respect to the coordinates $x,y$) vanishes. 

If $S$ is a space-like surface in a flat space-time with cone singularity, then the intersection of $S$ with the singular locus
is a discrete set and   a Riemannian metric $I$ is defined on $S^*=S\setminus (s_1\cup\ldots\cup s_n)$.
Clearly the first fundamental form and the shape operator on $S^*$, say $(I, s)$, satisfy the Gauss-Codazzi equation.

A notion of globally hyperbolic extension of $S$ makes sense even in this singular case and the existence and uniqueness of the maximal
extension can be proved by adapting \textit{verbatim} the argument given for the Anti de-Sitter case in \cite{bbsads}.

Like the previous section, we fix a topological surface $S$ and $k$ points $\puct=\{p_1\ldots, p_k\}$ and
 study pairs $(I,s)$  on $S\setminus\puct$ corresponding to embedding data of a Cauchy surface $S$ in a 
 globally hyperbolic space-time with particles of cone angles $\theta_1,\ldots,\theta_k$, so that the point $p_i$ lies
on the particle with cone angle $\theta_i$.  The divisor of $(I,s)$ is by definition $\bfs{\beta}=\sum\beta_i p_i$, where
we have put $\beta_i=\frac{\theta_i}{2\pi}-1$. Again, we consider embeddings which are isotopic to $S\hookrightarrow S\times\{0\}\subset M\cong S\times\R$.

As in the closed case we consider uniformly convex surfaces, but we also consider an upper bound for
the principal curvatures. More precisely  
we will assume that $s$ is a bounded and uniformly positive operator: that means that there exists a number $M$ such that $\frac{1}{M}\id<s<M\id$; in other words we require 
the eigenvalues of $s$ at every point to be  bounded between $1/M$ and $M$.

\noindent We now want to show that the set 
\begin{equation*}
\D_{\bfs{\beta}}= \left\{(I,s):\begin{aligned}
        &I\text{ and }s\text{ are embedding data of a closed uniformly convex Cauchy surface}\\
&\text{in a flat Lorentzian manifold with particles of angles }\theta_1,\ldots,\theta_k\\
 &\text{orthogonal to the singular locus with }s\text{ bounded and uniformly positive}
         \end{aligned}
 \right\}
\end{equation*}
is in bijection with
\begin{equation*}
\E_{\bfs{\beta}}= \left\{(h,b):\begin{aligned}
        &h\text{ hyperbolic metric on }S \text{ with cone singularities of angles }\theta_1,\ldots,\theta_k\\
&b:TS^*\rar TS^*\text{ self-adjoint},\text{ bounded and uniformly positive, } d^{\nabla}_{h} b=0
         \end{aligned} \right\}.
\end{equation*}

\noindent More precisely, we show the following relation.

%\begin{theorem} \label{corrispondenza coppie caso singolare}
%Let $(I,s)$ be the embedding data of a uniformly convex surface in a flat space-time with particles of angles $\theta_1,\ldots,\theta_k$, with $s$ 
%bounded and uniformly positive. Then the metric $h(v, w)=I(s(v), s(w))$ is a hyperbolic metric with cone singularities of angles $\theta_1,\ldots,
%\theta_k$ and the operator $b=s^{-1}$ satisfies the Codazzi equation for $h$.
%Conversely if $(h,b)\in \E_{\bfs{\beta}}$ 
%then $(I=h(b\bullet, b\bullet),s=b^{-1})$ are the embedding data of a uniformly convex surface in a flat spacetime with 
%particles, orthogonal to the singular locus.
%\end{theorem}

\begin{thmx}\label{corrispondenza coppie caso singolare}
%Let us fix a \emph{divisor} $\bfs{\beta}=\sum\beta_ip_i$ on a surface with $\beta_i\in(-1,0)$ and consider the following sets:
%\begin{itemize}
%\item $\E_{\bfs{\beta}}$ is the set of embedding data $(I,s)$ of bounded and uniformly convex 
%Cauchy surfaces on flat space-times with particles so
%that for each $i=1,\ldots , k$ a particle of angle $2\pi(1+\beta_i)$ passes through $p_i$.
%\item $\D_{\bfs{\beta}}$ is the set of pairs $(h,b)$, where $h$ is a hyperbolic metric on $S$ with a cone singularity
%of angle $2\pi(1+\beta_i)$ at each $p_i$ and $b$ is a self-adjoint solution of Codazzi equation for $h$, bounded and  uniformly positive.
%\end{itemize}
The correspondence $(I,s)\to (h,b)$, where $h=I(s,s)$ and $b=s^{-1}$, induces a bijection between $\E_{\bfs{\beta}}$ and $\D_{\bfs{\beta}}$.
\end{thmx}

This is a more precise version of Proposition \ref{identificationDE}, which additionally deals with the condition that $I$ and $h$ have cone singularities.
The fact that the hyperbolic metric associated to $(I,s)\in \D_{\bfs{\beta}}$ has a cone singularity at the singular points is a simple consequence of the fact that
its completion is obtained by  adding a point. 
In the opposite direction things are less clear and we will give a detailed proof of the fact that given $(h,b)$ as in the hypothesis of the theorem,
the surface $S$ can be embedded in a singular flat space-time with embedding data $(I,s)$.

We first prove two Lemmas which will be used in the proof of Theorem \ref{corrispondenza coppie caso singolare}.
%We prove Proposition \ref{lemma metrica euclidea} and another lemma which will be used in the proof.

\begin{lemma} \label{lemma metrica euclidea}
Let $g$ be a Euclidean metric on a disc $D$ with a cone point of angle $\theta_0\in(0,2\pi)$ at the point $0$. 
Suppose $u$ to be a $C^2$ function
on the punctured disc $D^*=D\setminus\{0\}$
 such that the Euclidean Hessian $\Hess_g u$ is bounded %and uniformly positive 
 with respect to $g$, 
 namely there exists a  constant $a_1>0$ such that $\Hess_g u<a_1\id$. Then $||\grad_g u(x)||\leq a_1 d_g(0,x)$
 and $u$ extends at $0$.
 
If moreover, $\Hess_g u$ is uniformly positive, i.e. there is a constant $a_2$ such that $a_2\Id<\Hess_g u$, then 
 the metric
 $g'(v,w)=g(\Hess_g u(v),\Hess_g u(w))$ is a Euclidean metric on $D$ with a cone point of angle $\theta_0$ at $0$.
\end{lemma}

\begin{proof}
Let $\tilde D^*$ be the universal cover of the punctured disc $D^*$, and let $\tilde u$ be the lifting of $u$ on $\tilde D^*$. Let $
\dev$ be a developing map for the Euclidean metric on $D$; we can assume $0$ is the fixed point of the holonomy of a path winding around $0$ in 
$D$. Then we define the map $\varphi:\tilde D^*\rar \R^2$ given by $$\varphi(x)=(\dev)_*(\grad_{ g}\tilde u(x))\,,$$
where with a standard abuse we denote by $g$ also the lifting of the Euclidean metric to $\tilde D^*$. 

By the hypothesis
\[
||d\varphi(v)||_{\R^2}<a_1||v||_{g}
\]
for any tangent vector $v$. 
In particular this estimate  implies that $\varphi$ is $a_1$-Lipschitz, so it extends to the metric
completion of   $\tilde D^*$, which is composed by one point $\tilde 0$  fixed
by the covering transformations. 
%Hence $\varphi$ extends to $p_\infty$.
As $\varphi$ conjugates the generator of the covering transformations $\tilde D\to D$ 
with the rotation of angle $\theta_0$ in $\R^2$, 
we must have $\varphi(\tilde 0)=0$, as $\varphi(\tilde 0)$ must be fixed by the action of the 
holonomy.  
This implies that 
\begin{equation} \label{pimpa}
||\grad_g \tilde u(x)||= ||\varphi(x)||=d_{\R^2}(\varphi(x), \varphi(0))\leq a_1 d_g(\tilde 0 ,x)\,.
\end{equation} 
From the boundedness of $\grad u$ we also obtain that $u$ is Lipschitz
hence $u$ extends with continuity to the metric completion $D$.
This concludes the first part.

Suppose now that $\Hess_g u$ is uniformly positive.
By construction, the pull-back $\varphi^* g_{\R^2}=
g(\Hess_{g} \tilde u(\bullet),\Hess_{g} \tilde u(\bullet))$ is a Euclidean metric for which $\varphi$ is a developing map, and has the 
same holonomy as $\dev$. 
We claim that $\varphi$ (suitably restricted if necessary to the lift of a smaller neighborhood of $0$, which we still 
denote by $D$), is a covering on $U\setminus\{0\}$, where $U$ is a neighborhood of $0$ in $\R^2$. This will show that $\varphi$ lifts to a 
homeomorphism $\tilde\varphi$ from $\tilde D^*$ and the universal cover of $U\setminus\{0\}$ conjugating the generator of $\pi_1(D^*)$ to
an element of the isometry group $\widetilde{SO}(2)$ of this covering.
Therefore $\tilde\varphi$ descends to an isometric homeomorphism  between $D$  equipped with the metric $g'$ and a model of a Euclidean disc with a cone singularity. 

To show the claim, observe that $\varphi$ is a local homeomorphism. Moreover, since $\Hess_g u$ is bounded by $a_2\id$ and $a_1\id$, we have 
\begin{equation} \label{bilipischitz local condition}
a_2||v||_{g}<||d\varphi(v)||_{\R^2}<a_1|v||_{g}
\end{equation} for any tangent vector $v$. 

We prove now that $\varphi(x)\neq 0$ for any $x\neq \tilde 0$.
Consider the geodesic path $\gamma:[0,t_0]\rightarrow \tilde D^*$ 
joining $x$ to $\tilde 0$ parametrized by arc length: 
it is simply the lifting of a radial geodesic in $D^*$. 
Then we have that 
\[
    g(\grad_{g} u (x),\, \dot\gamma (t_0))
  - g(\grad_{ g} u (\gamma(\epsilon)),\, \dot\gamma (\epsilon))
  =\int_\epsilon^{t_0} g(\Hess_g(\dot \gamma(t)),\dot\gamma(t))dt~.  
\]
Notice that $ |g(\grad_{g} u (\gamma(\epsilon)), \dot\gamma (\epsilon))|\leq ||\grad_{g} u (\gamma(\epsilon))||
=||\varphi (\gamma(\epsilon))||$. So, since $\varphi(\gamma(\epsilon))\rightarrow 0$ as $\epsilon\rightarrow 0$, we get
\[
 g(\grad_{g} u (x),\, \dot\gamma (t_0))=\int_0^{t_0} g(\Hess_g(\dot \gamma(t)),\dot\gamma(t))dt\,,
\]
and the integrand in the RHS is bigger than $a_2 t_0$. We deduce that
\begin{equation}\label{eq:unifdist}
  ||\varphi(x)||=||\grad_{g}u(x)||\geq a_2 t_0=a_2 d(x, \tilde 0)~.
\end{equation}

To conclude that $\varphi$ is a covering of $U\setminus \{0\}$ for some neighborhood of $0$ in $\mathbb R^2$,
it suffices to show that $\varphi$ has the path lifting property. Given any path $\gamma:[0,1]\to U\setminus\{0\}$, take a small $\rho_0$ so that $\gamma$ is contained in $U\setminus B(0,\rho_0)$. Therefore by \eqref{pimpa} a local lifting of $\gamma$
 is contained in a region of $\tilde D^*$ uniformly away from $\tilde 0$, hence metrically complete for the metric $g$. 
Equation \eqref{bilipischitz local condition} ensures that any local lifting of 
$\gamma$ has finite lenght and thus by standard arguments the lifting
can be defined on the whole interval $[0,1]$.

To complete the second point, we need to show that $g'$ has the same cone angle as $g$. By construction $g$ and $g'$ have the same holonomy, hence the cone angle can only differ by a multiple of $2\pi$. Consider the one-parameter family of functions $u_s:\tilde D^*\rar\R$, $$u_s(x)=s u(x)+\frac{1}{2}(1-s) d_g(0,x)^2.$$
The metrics $g_s(v,w)=g(\Hess_g u_s(v),\Hess_g u_s(w))$, constructed as above, form a one-parameter family of Euclidean metrics with cone singularities, depending smoothly on $s$. Moreover $g_0=g$ and $g_1=g'$. All the metrics $g_s$ have the same holonomy on a path around $0$, for the same construction. Therefore, by discreteness of the possible cone angles $\{\theta_0+2k\pi\}$, all metrics $g_s$ must have the same cone angle $\theta_0$.
\end{proof}

\begin{lemma} \label{lemma metrica euclidea2}
Let $S$ be a surface embedded in a flat space-time with particles. Suppose
that in a cylindrical neighborhood $C=D\times (a,b)$ of a particle where the metric is of the form $g=|z|^{2\beta_i}|dz|^2-dt^2$ as in \eqref{singular metric coordinates}, $S$ is the graph of a function $f:D\to (a,b)$.
If the shape operator of $S$  satisfies $ a_2\id<s<a_1\id$, then 
\begin{itemize}
\item There are constants $A_1$ and $A_2$ such that $A_2\id< \Hess_g f<A_1\id$, where $g$ is the Euclidean singular metric
on $D$.
\item $|f(x)-f(x_0)|=O(\rho^2)$ where $\rho$ is the Euclidean distance from the singular point $x_0$.
\end{itemize}
\end{lemma}

\begin{remark}
This lemma implies that any surface with shape operator bounded and uniformly positive is automatically orthogonal
to the singular locus.
\end{remark}

\begin{proof}
Let $S_C$ be the disc $S\cap C=\mbox{graph}(f)$.
Let $\dev$ be the developing map on the universal cover of $C$.
We may assume that $p_0=(0,0,1)$ is fixed by the holonomy of $C$, so
$\dev(x,t)=\dev_0(x)+(0,0,t)$, where $\dev_0$ is the developing map of the singular disc $(D,g)$.

Let $G:\tilde S_C\to\mathbb H^2$ be the Gauss map of the immersion $\dev|_{\tilde S_C}$.
First we notice that $G$ is locally bi-Lipschitz by our hypothesis, and since 
the diameter of $S_C$ is bounded, then the image through $G$ of a fundamental domain of the covering
$\tilde S_C\to S_C$ is contained in a hyperbolic ball $B(p_0, r_0)$. As the holonomy of $G$ is an elliptic group fixing $p_0$
we get that $G(\tilde S_C)$ is entirely contained in $B(p_0, r_0)$.

Now let $\pi:\mathbb H^2\to\mathbb R^2$ be the radial projection $\pi(x,y,z)=(x/z, y/z)$. Notice that
the restriction of $\pi$ over the disc $B(p_0, r_0)$ is a bi-Lipschitz diffeomorphism onto a Euclidean disc $B(p_0, \rho_0)$ where $\rho_0=\tanh r_0$.
Let $A=A(r_0)$ be the bi-Lipschitz constant of $\pi|_{B(p_0, r_0)}$.

Observe that, at a point $(x,f(x))$, the vector $\grad_g f(x)+\partial_t$ is a multiple of the normal vector of $S_C$. Hence the normal vector of the immersion $\dev:\tilde S_C\to\R^{2,1}$ at $(\tilde x,\tilde f(\tilde x))$ is parallel to the vector
$(\dev_0)_*(\grad_{ g} \tilde f(\tilde x))+(0,0,1)$.
Then it is easy to check that $\pi(G(\tilde x,\tilde f(\tilde x)))=(\dev_0)_*(\grad_{ g}\tilde f(\tilde x))$. So we get
\[
g(\hess_g \tilde f(v), \hess_g \tilde f(v))=\langle (\pi\circ G)_*(v'), (\pi\circ G)_*(v')\rangle\,,
\]
where $v'=v+df(v)\partial_t$.
Using that $\pi$ is $A$-bi-Lipschitz and that $\langle G_*(v'), G_*(v')\rangle = I(sv', sv')$ we get on $D$
\begin{equation}\label{eq:lll}
                                 \frac{a_1^2}{A}I(v',v') <g(\hess_g(f)v, \hess_g(f)v)<A a_2^2 I(v',v')\,.
\end{equation}
Now $I(v', v')=g(v,v)-(df(v))^2$ so it turns out that $g(\hess_g(f)\bullet, \hess_g(f)\bullet)$ is uniformly bounded.
By the first part of Lemma \ref{lemma metrica euclidea} there is $a_0$ such that $||\grad_g f(x)||<a_0 \rho$, where $\rho$
is the distance from  the singular point $x_0$.
In particular, in a small neighborhood of the singular point we have $(1-\epsilon)g(v,v)\leq I(v',v')$, for a fixed small $\epsilon$.
By \eqref{eq:lll} we conclude the proof of the first part of the lemma.

Finally,  as $||\grad_g f(x)||<a_0 \rho$,  a simple integration along the geodesic connecting $x$ to $x_0$ shows that
$|f(x)-f(x_0)|<(a_0/2) \rho^2$.
\end{proof}

\begin{proof} [Proof of Theorem \ref{corrispondenza coppie caso singolare}]
We start by showing that $I=h(b\bullet,b\bullet)$ and $s=b^{-1}$ define an embedding in a flat manifold with particles. 
The key point is to prove that a neighborhood  $U$ of a singular point $p_0\in \puct$ can be immersed as a graph in the model of the cone
singularity with embedding data $(I,s)$. Once this has been proved, by a standard application of the uniqueness of the extension one sees
that $S$ can be globally immersed in a space-time with cone singularity as in Definition \ref{defi spacetime singular}.

%Let $d:\tilde S^*\rar\Hyp^2$ be a developing map for $h$. From \eqref{eq:normal}, we know that $\tilde\sigma:\tilde S^*\rar\R^{2,1}$ given by $\tilde \sigma(x)=d_*(\grad_h u(x))-u(x)d(x)$ gives a local embedding of a surface with the required embedding data. We need to consider what happens at a singular point $p_0$. 
Let $D$ be a small disc centered at $p_0$ and let $\tilde D^*$ be the universal cover of the punctured disc $D^*=D\setminus\{p_0\}$. Suppose the cone angle at $p_0$ is $\theta_0$. Now we denote by $d:\tilde D^*\rar\Hyp^2$ the restriction of the developing map of $h$ to $\tilde D^*$; we can assume $(0,0,1)$ is the fixed point for the holonomy of a path winding around $p_0$. Consider the radial projection $\pi:\Hyp^2\rar \R^2$ from hyperbolic plane to the horizontal plane  $\{(x,y,z)\,|\,z=1\}$ in $\R^{2,1}$, namely $\pi(x,y,z)=(x/z,y/z)$. %Observe that $\pi$ gives the usual identification between the hyperboloid model and the Klein model of $\Hyp^2$. 
Let $\dev=\pi\circ d$, which is a developing map for the Klein Euclidean metric $g_K$ on $D^*$ introduced in Subsection \ref{ss:conprel}, which will be denoted simply by $g$ in the following.

Now let $u$ be a function on $D^*$ such that $\tilde b=\Hess_h u-  u\id$, that exists  by Proposition \ref{pr:trivial}. 
We consider the function $\bar u$ on $D^*$ as $\bar u(x)= u(x)/\ch r$ where we have put $r=d_h(p_0,x)$. 
We know that $\tilde\sigma:\tilde D^*\rar\R^{2,1}$ given by $\tilde \sigma(x)=d_*(\grad_h u(x))-u(x)d(x)$ gives an immersion of $\tilde{D^*}$  with the required embedding data. Here with some abuse
we denote by $u$ also the lifting of $u$ to the universal cover.
By \cite[Lemma 2.8]{bonfill} it turns out that the orthogonal projection of $\tilde\sigma(x)$ onto the horizontal plane $\mathbb R^2\subset\mathbb R^{2,1}$ (where again we are supposing that the vertical direction is fixed by the holonomy of a loop around $p_0$)  is exactly 
\[
\varphi(x)=(\dev)_*(\grad_g \bar u(x))\,,
\]
Thus $\tilde \sigma(x)=\varphi(x)+f(x)(0,0,1)$, for some function $f:\tilde D^*\to\R$. 
Notice that the holonomy of $\sigma$ is simply the elliptic rotation around the vertical line of angle $\theta_0$, so it turns out that
$f$ is invariant by the action of covering transformation of $\tilde D^*$ so it projects to a function still denoted by $f$ on $D^*$.

Now we claim that $\varphi$ is the developing map for a Euclidean structure with cone singularity $\theta_0$ over $D^*$.
In fact by Lemma \ref{lm:kleincorr} we know that  for $v,w\in T_x D^*$
\begin{equation} \label{relation Hessians}
g(\Hess_g \bar u(v),w)=\frac{1}{\ch r}h(b(v),w)\,.
\end{equation}
%This relation can be found in (citare Bonsante Fillastre Lemma 2.8) for functions defined  on the hyperboloid or Klein model of $\Hyp^2$, and is easily adapted to our setting since the computation is local. 
Therefore, $\Hess_g \bar u$ is bounded and uniformly positive. Indeed $b$ and the factor $1/\cosh r$ appearing in Equation (\ref{relation Hessians}), as well as the metric $h$ compared to $g$, are bounded on $D^*$. Applying Lemma \ref{lemma metrica euclidea}, $g'=g(\Hess_g \hat u(\cdot),\Hess_g \hat u(\cdot))$ is a Euclidean metric with cone angle $\theta_0$ 
and $\varphi$ coincides with its developing map.

As a consequence, if we consider on $M=D^*\times\mathbb R$ the flat Lorentzian metric with particle $g'-dt^2$, its developing map is $\mathrm{Dev}(x,t)=\varphi(x)+t(0,0,1)$. In particular we have that $\mathrm{Dev}(x,f(x))=\tilde\sigma(x)$.
So we have shown that the map $\sigma(x)=(x, f(x))$ is an immersion of $D^*$ into the space-time with particles  $M$ with embedding data
$(I, s)$.
The fact that the immersion is orthogonal to the singular locus follows from Lemma \ref{lemma metrica euclidea2}.

We prove now the opposite implication showing that if $(I,s)$ are the embedding data of a Cauchy surfaces in a space-time $M$ with singularities
of angle $\theta_i$, then $h=I(s\bullet, s\bullet)$ is a hyperbolic metric with cone points of angle $\theta_i$.

We know $h=I(s\bullet,s\bullet)$ is a hyperbolic metric, with holonomy a rotation of angle $\theta_i$ around each cone point $p_i$ of $I$. Moreover $s$ is bounded and uniformly positive, hence $h$ admits a one-point completion as $I$ does, and this is sufficient to show that $h$ has cone singularity. As $I=h(b,b)$, by applying  the first part of the proof to the pair $(h,b)$ 
one sees  that the cone angles of $h$ coincide with the cone angles of $M$.
\end{proof}

\subsection{A Corollary about metrics with cone singularities}

We write here a consequence of the previous discussion, which might be of interest independently of Lorentzian geometry.

\begin{thmx} \label{hbb cone metric}
Let $h$ be a hyperbolic metric with cone singularities and let $b$ be a Codazzi, self-adjoint operator for $h$, bounded and uniformly positive. Then $I=h(b\bullet,b\bullet)$ defines a singular metric with the same cone angles as $h$. Moreover if $I=e^{2\xi}|w|^{2\beta}|dw|^2$
in a conformal coordinate $w$ around a singular point $p$, the factor $\xi$ extends to a H\"older continuous function at $p$.
\end{thmx}

To prove Theorem \ref{hbb cone metric}, we consider the uniformly convex surface constructed in Theorem \ref{corrispondenza coppie caso singolare}, with first fundamental form $I$. We must show that this metric has cone singularities. For every puncture of $S$, consider the singular Euclidean metric on the disc $D$, provided by Lemma \ref{lemma metrica euclidea}. We now call this metric $g$ instead of $g'$. Suppose the embedding in the chart $D^*\times \R$ with the metric $g-dt^2$ is the graph of a function $f:D^*\to\R$. Hence 
if $z$ is a conformal coordinate for the metric $g$,  the metric $I$ can be written in this coordinate as
$$I=|z|^{2\beta}|dz|^2-df\otimes df=|z|^{2\beta}\left(|dz|^2-\frac{df\otimes df}{|z|^{2\beta}}\right)\,,$$
where $g=|z|^{2\beta}|dz|^2$.
\begin{lemma} \label{lemma holder coefficients}
The coefficients of the metric 
$$I'=|dz|^2-\frac{df\otimes df}{|z|^{2\beta}}=|z|^{-2\beta}I$$
extend to H\"older functions defined on $D$.
\end{lemma}

\begin{proof}
%First observe that the factor $du\otimes du/|z|^{2\beta}$ is bounded. Indeed, since $I$ is a spacelike metric, for every vector $v$ we have $|du(v)|<||v||_{g'}=|z|^{\beta}||v||_{|dz|^2}$.
It suffices to show that the functions $|z|^{-\beta}\partial_x f$ and $|z|^{-\beta}\partial_y f$ are H\"older functions. We will give the proof for the first function, which we denote $$F(z)=|z|^{-\beta}\frac{\partial f}{\partial x}.$$
We split the proof in three steps. Recall from Lemmas \ref{lemma metrica euclidea2} and \ref{lemma metrica euclidea} that we have $||\Hess_g f||\leq C$ and $||\grad_g f||_g\leq C \rho$ for some constant $C$, where $\rho=\frac{1}{1+\beta}|z|^{1+\beta}$ 
is the intrinsic Euclidean distance.

First, consider the path $\gamma_1(t)=\varrho e^{it}$ for $t\in[\phi_0,\phi_1]$, so that $|\gamma_1(t)|=\varrho$ is constant. We claim that $$|F(\gamma_1(\phi_1))-F(\gamma_1(\phi_0))|\leq C_1\varrho^{\beta+1}|\phi_1-\phi_0|$$
for some constant $C_1$. Consider 
\begin{align} \label{stima1holder}
\begin{split}
F(\gamma_1(\phi_1))-F(\gamma_1(\phi_0))&=\int_{\phi_0}^{\phi_1}\frac{d(F\circ\gamma_1(t))}{dt}dt \\
&=\int_{\phi_0}^{\phi_1}\varrho^{-\beta}\left(g\left(\Hess_g f(\dot\gamma_1),\frac{\partial}{\partial x}\right)+df\left(\nabla_{\dot\gamma_1}\frac{\partial}{\partial x} \right)\right).
\end{split}
\end{align}
Now we have
\begin{equation} \label{termine1stima1holder}
|g\left(\Hess_g f(\dot\gamma_1),\partial_x \right)|\leq C||\dot\gamma_1||_g||\partial_x||_g=C\varrho^{2\beta+1}.
\end{equation}
On the other hand, a computation (using Equation \eqref{eq:levicivita} in Remark \ref{remarkremark}) shows $\nabla_{\dot\gamma_1}\partial_x=\beta\partial_y$, hence
\begin{equation} \label{termine2stima1holder}
|df\left(\nabla_{\dot\gamma_1}\partial_x\right)|_g\leq\beta||\grad_g f||_g||\partial_x||_g\leq (|\beta|/\beta+1) C\varrho^{2\beta+1}.
\end{equation}
Using \eqref{termine1stima1holder} and \eqref{termine2stima1holder} in \eqref{stima1holder}, we get
$$|F(\gamma_1(\phi_1))-F(\gamma_1(\phi_0))|\leq C_1\varrho^{\beta+1}|\phi_1-\phi_0|%\leq C''\rho^{\beta+1}|\theta_1-\theta_0|^{\beta+1}
$$
%since $|\theta_1-\theta_0|\in[0,2\pi]$.

As a second step, we consider the path $\gamma_2(t)=t z_0$ with $|z_0|=1$, for $t\in[t_0,t_1]$. We claim that
$$|F(\gamma_2(t_1))-F(\gamma_2(t_0))|\leq C_2|t_1-t_0|^{\beta+1}.$$
Since $|\gamma_2(t)|=t$, we have 
\begin{equation} \label{stima2holder}
\frac{d(F\circ\gamma_2(t))}{dt}=-\beta t^{-\beta-1}\frac{\partial f}{\partial x}+t^{-\beta}\left(g\left(\Hess_g f(\dot\gamma_2),\frac{\partial}{\partial x}\right)+df\left(\nabla_{\dot\gamma_2}\frac{\partial}{\partial x} \right)\right).
\end{equation}
From $||\grad_g f||\leq C\rho$, we get $|\partial f/\partial x|\leq (C/(\beta+1))t^{2\beta+1}$, hence the first term in \eqref{stima2holder} is bounded  by $(|\beta|/\beta+1) Ct^{\beta}$. For the second term we have as above
\begin{equation} \label{termine2stima2holder}
|g\left(\Hess_g f(\dot\gamma_2),\partial_x \right)|\leq C||\dot\gamma_2||_g||\partial_x||_g=Ct^{2\beta}\,,
\end{equation}
\noindent whereas in this case $\nabla_{\dot\gamma_2}\partial_x=(\beta/t)\partial_x$, from which we get
\begin{equation} \label{termine3stima2holder}
|df\left(\nabla_{\dot\gamma_2}\partial_x\right)|_g\leq|\beta/t|||\grad_g f||_g||\partial_x||_g\leq (|\beta|/\beta+1) Ct^{2\beta}\,.
\end{equation}
By integrating we get
$$|F(\gamma_2(t_1))-F(\gamma_2(t_0))|\leq\int_{t_0}^{t_1}\left| \frac{ d(F\circ\gamma_2(t))}{dt} \right| dt\leq C_2|t_1^{\beta+1}-t_0^{\beta+1}|\leq C_2|t_1-t_0|^{\beta+1}.$$

We can now conclude the proof. Given two points $z_0=\varrho_0e^{i\phi_0}, z_1=\varrho_1 e^{i\phi_1}\in D$, assuming $\varrho_1\geq \varrho_0$, consider the point $z_2=\varrho_0 e^{i\phi_1}$. We have
\begin{align*}
|F(z_1)-F(z_0)|&\leq |F(z_1)-F(z_2)|+|F(z_2)-F(z_0)|\leq C_2|\varrho_1-\varrho_0|^{\beta+1}+C_1\varrho_0^{\beta+1}|\phi_1-\phi_0| \\
&\leq C_3(|z_1-z_0|^{\beta+1}+\varrho_0^{\beta+1}|\phi_1-\phi_0|^{\beta+1}) \\
&\leq C_4|z_1-z_0|^{\beta+1}.
\end{align*}
In the second line we have used that $|\varrho_1-\varrho_0|\leq |z_1-z_0|$ and the constant $C_3$ involves $C_1$, $C_2$ and a factor which bounds $|\phi_1-\phi_0|^{\beta+1}$ in terms of $|\phi_1-\phi_0|$, since $|\phi_1-\phi_0|\in[0,2\pi]$. In the last line, we have used $\varrho_0|\phi_1-\phi_0|\leq (\pi/2)|z_1-z_0|$.
\end{proof}

\begin{proof}[Proof of Theorem \ref{hbb cone metric}]
By a classical theorem of Korn and Lichtstein (see \cite{chern} for a proof), Lemma \ref{lemma holder coefficients} implies that
there exists a $\mathrm C^{1,\alpha}$ conformal coordinate $w=w(z)$ for the metric $I'$ in a neighborhood of the point $p$.
Now $I=|z|^{2\beta}I'=e^{2\xi}|w|^{2\beta}|dw|^2$ where $\xi$ is some continuous function on a neighborhood of the puncture. This concludes the proof that $I$ has a cone point of the same angle as $h$.
\end{proof}

\subsection{Cauchy surfaces in MGHF space-times with particles}

The purpose of this subsection is to prove a step towards the parametrization of MGHF structures with particles on $S\times\R$ by means of the tangent bundle of Teichm\"uller space of the punctured surface $S$. The parametrization will be completely achieved in Corollary \ref{cor:main2}.

Given two uniformly convex Cauchy surfaces in a MGHF space-time $M$, we already know from Theorem \ref{theorem holonomy} that the holonomy of the third fundamental form of a Cauchy surface coincides with the linear holonomy of $M$. However, differently from the closed case, this is not sufficient to guarantee that the two hyperbolic metrics obtained as third fundamental form correspond to the same point in Teichm\"uller space. We prove this separately in Proposition \ref{compito ingegneria}, by using techniques similar to those developed in \cite{mehdi}.

Next, we prove a converse statement in Proposition \ref{existence and uniqueness with singularities}. Namely, if two pairs of embedding data $(I,s)$ and $(I',s')$ are such that the third fundamental forms $h=I(s\bullet,s\bullet)$ are isometric via an isometry isotopic to the identity, and the translation parts of the holonomy are in the same cohomology class, then $(I,s)$ and $(I',s')$ are embedding data of uniformly convex Cauchy surfaces in the same space-time.

\begin{prop} \label{compito ingegneria}
Let $(I,s),(I',s')\in \D_{\bfs{\beta}}$ be embedding data of uniformly convex Cauchy surfaces in the same space-time with particles $M$. Then $h=I(s\bullet,s\bullet)$ and $h'=I'(s'\bullet,s'\bullet)$ are isotopic.
\end{prop}

We first give an observation which will be used several times in the proof.

\begin{remark} \label{remark scoiattolo}
Given a uniformly convex surface $S$ with embedding data $(I,s)$, let $(h,b)\in \E_{\bfs{\beta}}$ the corresponding pair. Let $S(t)$ be the surface obtained by the future normal flow of $S$ at time $t$. It is known that $S(t)$ corresponds to the pair $(h,b+t\id)$, namely, the third fundamental form is constant along the normal flow. Moreover, as the first fundamental form of $S(t)$ is $I_t=h((b+t\id)\bullet,(b+t\id)\bullet)$, $h$ can also be recovered as $h=\lim_{t\rar\infty}\frac{1}{t^2}I_t$.
\end{remark}

We will also use the following lemma concerning the properties of flat globally hyperbolic space-times.

\begin{lemma} \label{lemma cocci}
Let $S_1$ and $S_2$ be uniformly convex surfaces in a MGHF space-time with particles. If $S_2$ is contained in the future of $S_1$ and in the past of $S_1(a)$, then $S_2(t)$ is contained in the future of $S_1(t)$ and in the past of $S_1(a+t)$ for every $t>0$.
\end{lemma}

\begin{proof}
We claim that a point $x$ is in the future of $S_1(t)$ if there is a timelike path with future endpoint in $x$, of length at least $t$, entirely contained in the future of $S_1$. From this claim, the thesis follows directly.

To prove the claim, assume $x$ is in the past of $S_1(t)$. The pull-back of the Lorentzian metric of $M$ using the normal flow takes the form $-ds^2+g_s\,,$
where $g_s$ are Riemannian metrics. Hence every causal path contained in the future of $S_1$ and with endpoint $x$ has length at most $t$.
\end{proof}

\begin{proof}[Proof of Proposition \ref{compito ingegneria}]
Let $\sigma_1:S\rar M$ and $\sigma_2:S\rar M$ be embeddings of uniformly convex Cauchy surfaces $S_1$ and $S_2$ with embedding data $(I,s)$ and $(I',s')$ respectively. Assume first that $S_2$ is contained in the future of $S_1$ and in the past of $S_1(a)$. Applying Lemma \ref{lemma cocci} and \cite[Proposition 4.2]{mehdi}, one sees that the projection from $S_2(t)$ to $S_1(a+t)$ obtained by following the normal flow of $S_1$ is distance-increasing, and is a diffeomorphism by the property of Cauchy surfaces.
Hence we obtain a one-parameter family of 1-Lipschitz diffeomorphisms (which clearly extend to the punctures) $f_t:S_1(a+t)\rar S_2(t)$. 

Recall that by Remark \ref{remark scoiattolo}, for the first fundamental form $I_t$ of $S_2(t)$, $I_t/t^2$ converges to the third fundamental form $h_1$, and analogously for $h_2$. 
%Thus $f_t:(S_1(a+t),h_1)\rar(S_2(t),h_2)$ is $(1+\epsilon)$-Lipschitz for $t$ big.
Hence by Ascoli-Arzel\`a Theorem we obtain a 1-Lipschitz map $f_\infty:(S_1,h_1)\rar (S_2,h_2)$ homotopic to $f_0$. Since the areas of $(S_1,h_1)$ and $(S_2,h_2)$ coincide by Gauss-Bonnet formula, $f_\infty$ is necessarily an isometry. %From the construction, it is clear that $f_\infty$ is isotopic to the identity.
Moreover, by construction it is clear that $(\sigma_2)^{-1}\circ f_\infty\circ\sigma_1$ is isotopic to the identity.

In the general case, given $S_1$ and $S_2$ uniformly convex, it suffices to replace $S_2$ by $S_2(k)$ for $k$ to reduce to the previous case. Indeed, we have already observed that the third fundamental forms coincide for $S_2$ and $S_2(k)$.
\end{proof}

We now move to the proof of a Proposition \ref{existence and uniqueness with singularities}, which is a converse statement.

\begin{prop} \label{existence and uniqueness with singularities}
Let $h$ be a hyperbolic metric on $S$ with cone singularities and let $b,b'\in\cC_\infty(S, h)$ be bounded and uniformly positive Codazzi operators. If $\delta(b)=\delta(b')$, then the pairs $(I,s)$ and $(I',s')$ corresponding to $(h,b)$ and $(h,b')$ are embedding data of uniformly convex Cauchy surfaces in the same MGHF space-time with particles.
\end{prop}

%Let $M$ be a MGH spacetime containing a strictly convex Cauchy surface. Let $h$ be the hyperbolic metric with singularities corresponding to the linear holonomy $f$. We consider the Codazzi self-adjoint $b\in E^*_{\theta_1,\ldots,\theta_n}$ which are positive definite (that is, correspond to the embedding data of a strictly convex surface) with $\Phi(b)=t$, where $t$ is the translation part of the holonomy $\rho$. We want to show that any two choices of $b$ correspond to embeddings of surfaces in the same MGH spacetime.
%For this purpose, we first show that any small deformation of $b$ keeping the holonomy fixed provide embeddings in the same spacetime.
Recall from Theorem \ref{theorem holonomy} that, under the hypothesis, $\delta(b)$ is the translation part of the holonomy of the space-time $M$ provided by the embedding data $(I,s)$, and the linear part is the holonomy of $h$. The idea of the proof is to show that any small deformation of $b$ which leaves the holonomy invariant gives an embedding into the same space-time $M$. Then, we use connectedness of the space 
$$\left\{b'\in \cC_\infty(S, h): b'\text{ is uniformly positive and }\delta(b')=\delta(b)\right\}.$$
So we prove the first assertion, by a standard argument.

\begin{lemma} \label{lemma nearby structures}
Let $b\in \cC_\infty(S, h)$ be uniformly positive, and $\dot b\in \cC_\infty(S, h)$ be such that $\delta(\dot b)=0$. Let $M$ be the MGH space-time obtained from the embedding data $(h,b)$. Then there exists $\epsilon>0$ such that for $s\in (-\epsilon,\epsilon)$ every pair $(h,b+s\dot b)$ gives an embedding of a uniformly convex spacelike surface into the space-time $M$.
\end{lemma}
\begin{proof}
Let $\tilde\sigma:\tilde S\rar\R^{2,1}$ be the embedding constructed as in Proposition \ref{caratterizzazione embedding inversi gauss}.
We can choose a smooth path of developing maps $\dev_s:\tilde S\rar\R^{2,1}$ having embedding data $(h,b+s\dot b)$, for $s\in(-\epsilon,\epsilon)$. Since by linearity $\delta(b+s\dot b)=\delta(b)$ for all $s$, $\dev_s$ all have the same holonomy. We can also assume $\dev_0$ extends to a developing map $\mathrm{Dev}_0$ for $M$ defined on the lifting of a tubular neighborhood $\tilde T\cong \tilde S\times(-a,a)$ of $S$ in $M$ and there is a covering $\{U_\alpha\}$ of $S$ such that, for every $\alpha$, either $\mathrm{Dev}_0$ is an isometry on its image when restricted to $U_\alpha\times(-a,a)$ (if $U_\alpha$ does not contain any singular point) or there is a chart for $U_\alpha\times(-a,a)$ to a manifold $(D\times \R,g-dt^2)$ where $g$ is a Euclidean metric on the disc $D$ with a cone point at $0$. Restricting to a smaller $\epsilon$ if necessary, and using the fact that all $\dev_s$ have the same holonomy, we see that $\dev_s$ provides an embedding of $\tilde S$ into the space-time obtained by gluing the charts $\{U_\alpha\times(-a,a)\}$, for $s\in(-\epsilon',\epsilon')$. This concludes the proof.
\end{proof}

\begin{proof}[Proof of Proposition \ref{existence and uniqueness with singularities}] Given $b$ and $b'$ as in the hypothesis, %let $\mu_1\geq \mu_2>0$ and $\mu'_1\geq \mu'_2>0$ the eigenvalues of $b$ and $b'$ respectively. 
let %$v=\max\{||b||_\infty,||b'||_\infty\}$
$v=||b'||_\infty$ be a constant function and let $ b_s=b+sv\id$ for $s\in[0,1]$. Since $b$ is modified by adding $sv\id=-\Hess(sv)+(sv)\id$, $\delta(b_s)=\delta(b)$ for every $s$. Clearly, $b_s$ is bounded and uniformly positive for every $s$. By Lemma \ref{lemma nearby structures}, $(h,b)$ and $(h,b_1)$ correspond to embeddings of uniformly convex surfaces in the same space-time $M$. Now the same argument can be applied to $b'_s=b'+s(b_1-b')$, which is again bounded and uniformly positive for every $s$, since by construction the eigenvalues of $b_1$ at every point are larger than the largest eigenvalue of $b'$, and $\delta(b'_s)=\delta(b')$ by linearity of $\delta$. Therefore, $b$ and $b'$ correspond to embeddings in the same space-time $M$.
\end{proof}

%----------------------------------------------sezione 5---------------------------------

%\input bs-sec5.tex

\section{Relation with Teichm\"uller theory} \label{sectionteichsing}

%In the last section it has been proved that the space of MGHF space times with particles containing a uniformly convex surface
%is parameterized by the tangent bundle of the Teichm\"uller space.

Let us fix a topological surface $S$, and a divisor $\bfs{\beta}=\sum\beta_i p_i$, where $\beta_i\in(-1,0)$, 
and put $\puct=\{p_1,\ldots,p_k\}$. Recall we are assuming $\chi(S, \bfs{\beta})<0$.

We denote by $\mathcal H(S,\bfs{\beta})$ the space of hyperbolic metrics on $S$ with cone singularity $\theta_i=2(1+\beta_i)\pi$ at $p_i$.
On the other hand, we consider the space of MGHF structures containing a uniformly convex Cauchy surface orthogonal to the singular locus with bounded second fundamental form, which we denote as follows:
$$\mathcal{S}_+(S, \bfs{\beta})=\left\{\begin{aligned}
&\text{MGHF structures on }S\times\mathbb R\text{ with particles along }\puct\times\R\text{ of angles }\theta_i \\
&\text{containing a convex Cauchy surface orthogonal to }\puct\times\R \\
&\text{having bounded and uniformly positive shape operator}  
\end{aligned}\right\}/{\sim}$$
where two structures are equivalent for the relation $\sim$ if and only if they are related by an isometry of $S\times\R$ isotopic to the identity fixing each particle.

In Theorem \ref{corrispondenza coppie caso singolare}
we have seen that the pairs $(h,b)$, where $h\in \mathcal H(S,\bfs{\beta})$ and $b$ is a bounded and positive $h$-Codazzi tensor on 
$S$, bijectively correspond to immersion data of convex Cauchy surfaces in a MGHF space-time, orthogonal to the singular locus.
Moreover the space-times corresponding to $(h,b)$ and $(h', b')$ are in the same equivalence class in $\mathcal{S}_+(S, \bfs{\beta})$ if and only if 
there is an isometry from $(S,h)$ to $(S,h')$ isotopic to the identity in $\Homeo^+(S, \puct)$
and $\delta(b)=\delta(b')$.

In this section we want to use this characterization to construct a natural bijective map between $\cS_+(S,\bfs{\beta})$ and
the tangent space of Teichm\"uller space of the punctured surface $\mathcal T(S,\puct)$.

Let us recall that elements of $\cT(S,\puct)$ are  complex structures on $S$, say $X=(S,\mathcal A)$, where $\mathcal{A}$ is a complex atlas on $S$, considered up to isotopies of
$S$ which point-wise fix $\puct$.
Dealing with complex structures associated with singular metrics on $S\setminus\puct$ 
makes important to clarify the regularity of the complex
atlas. %Let us fix a base differentiable structure on $S$.
In the classical Teichm\"uller theory one can deal with complex atlas whose charts
are only quasi-conformal with respect to a base smooth complex structure on $S$.
In this framework the group acting on this space of complex structures is the space of quasi-conformal homeomorphisms
of $S$, which does not depend on the complex structure chosen.
We have the following lemma:

\begin{lemma} \label{lemmacomplex}
Let $h$ be a hyperbolic metric with cone singularities on $S$, and $b\in\cC_\infty(S,h)$.
Then there is a complex structure $X$ on $S$ such that the metric $\hat h_t=h((\Id+tb)\bullet, (\Id+tb)\bullet)$ is conformal for $X$.
\end{lemma} 
 \begin{proof}
 Clearly there is a smooth complex structure $\mathcal A^*_t$ 
 on $S\setminus\puct$ for which $\hat h_t$ is conformal.
 We have to prove that $\mathcal A^*_t$ extends to a complex atlas over $S$.
 
 If $t=0$ then this basically follows from the definition of metric with cone singularities.
 
 Consider now the general case. Notice that $\Id:(S\setminus\puct, \mathcal A^*_0)\to
 (S\setminus\puct, \cA^*_t)$ is quasi-conformal. In particular this implies that a neighborhood $U$ of any 
 puncture $p_i$ with the structure inherited by $\cA^*_t$ is quasi-conformal to a punctured disc. So it is biholomorphic to a punctured disc, that
 is there is a biholomorphic map
\[
\zeta:(U\setminus p_i, \cA^*_t)\to \D^*\,,
\] 
 where $\D^*=\{z\in\mathbb C\,|\, 0<|z|<1\}$.
 
 It is not difficult to show that $\zeta$ extends by continuity to a homeomorphism $\zeta:U\to \D$.
 This proves that the atlas $\cA^*_t$ extends to $S$. 
 Finally notice that the function $\zeta$  in general is not smooth at $p_i$, but, as the map
 \[
      \zeta:(U, \cA_0)\to \D
 \]
is quasi-conformal, it has the requested regularity.   
 \end{proof}
 
At a point $[X]\in\cT(S,\puct)$ , the tangent space of $\cT(S, \puct)$ is identified with a quotient of 
the space of Beltrami differentials $\cB(X)$. 
We say that a Beltrami differential $\mu$ is trivial if $\langle q, \mu\rangle=0$ for any holomorphic quadratic differential with poles of order 
at worst $1$ at $p_i$ (equivalently with for any holomorphic  section of $K^2(\puct)$).
We will denote by $\cB(X,\puct)^\perp$ the subspace of trivial Beltrami differentials.
The tangent space  $T_{[X]}\cT(S,\puct)$  is naturally identified with $\cB(X)/\cB(X,\puct)^\perp$.
The identification works as in the case of a closed surface. The main difference is that the derivative of the Beltrami differential
of the map $\Id:X\to X_t$ is well-defined up to this more restrictive relation as we only consider homotopies which point-wise fix
$\puct$.
It turns out that if $\sigma$ is a smooth section on $K^{-1}$, then $\bar\partial\sigma$ is trivial iff $\sigma$ vanishes at punctures.

The main theorem we prove in this section is the analogue of Theorem \ref{theorem parametrization closed surfaces} in the closed case. 
Given a hyperbolic metric $h\in\cH(S,\bfs{\beta})$ and $b\in\cC_\infty(S,h)$, the family of Riemannian metrics 
$\hat h_t=h((\Id+tb)\bullet, (\Id+tb)\bullet)$ induces a smooth family 
$X_t$ of complex structures by Lemma \ref{lemmacomplex}. As in the closed case treated in Section \ref{glopar}, 
the derivative of $X_t$ coincides with the Beltrami differential corresponding to the traceless part $b_0=b-(\tr b/2)\id$. 
This leads again to the definition of a map $\Psi:\E_{\bfs{\beta}}\to T\cT(S,\puct)$.

\begin{thmx}\label{thm:coneteich}
%\label{thm:main4}
Let $\cC_\infty(S, h)$ be the space of bounded Codazzi tensors on $(S, h)$.
The following diagram is commutative 
\begin{equation}\label{eq:commcone}
\begin{CD}
\cC_\infty(S, h) @>\Lambda\circ \delta>>H^1_{\tiny\Ad\circ\tiny\hol}(\pi_1(S\setminus\puct),\so(2,1))\\
@V\Psi VV                                       @A d\Hol AA\\
T_{[X]}\cT(S,\puct) @>>\mathcal J > T_{[X]}\cT(S,\puct)
\end{CD}\,,
\end{equation}
where $\Lambda:H^1_{\tiny\hol}(\pi_1(S\setminus\puct), \R^{2,1})\to H^1_{\tiny\Ad\circ\tiny\hol}(\pi_1(S\setminus\puct),\so(2,1))$ 
is the natural isomorphism, and $\cJ$ is the complex structure on $\cT(S,\puct)$.
\end{thmx} 

%\begin{theorem}\label{thm:coneteich}
%The map
%\[
%\Psi:\E_{\bfs{\beta}}\to T\cT(S,\puct)
%\]
%sending $(h,b)$ to $([X_h], [b_0])$ is surjective. Moreover
%$\Psi(h,b)=\Psi(h',b')$ if and only if $h$ and $h'$ are isotopic metrics and $\delta(b)=\delta(b')$.
%\noindent In particular the map induces to the quotient a bijective map
%\[
%\bar\Psi: \cS_+(S,\bfs{\beta})\to T\cT(S, \puct).
%\] 
%\end{theorem}

As in the closed case the proof of this theorem is based on the computation of the differential of the holonomy map, where
in this case 
\[
\Hol:\cT(S,\puct)\to\cR(\pi_1(S\setminus\puct), \SO(2,1))/\!\!/ \SO(2,1)
\]
is the map sending the marked Riemann surface $[X]$ to the holonomy of the unique hyperbolic conformal metric on
$S\setminus\puct$ with cone singularities $\theta_i$ at $p_i$. The existence of such a metric is a corollary of a more general result
\cite{Troyanov}.

In \cite{schumacher} it has been proved that if $t\mapsto [X_t]$ is a smooth path in $\cT(S, \puct)$, then there is a smooth family of hyperbolic metrics
$h_t$ on $S\setminus\puct$ whose underlying complex structure is isotopic to $X_t$.
By Proposition \ref{pr:deltastar} it turns out that the holonomy map is smooth.
Now we want to compute precisely the differential of $\Hol$.
In particular we prove the analogue of Proposition \ref{pr:holclos}. Then, the proof of Theorem \ref{thm:coneteich} follows exactly as in the closed case.

\begin{prop}\label{prop:hol}
Let $h$ be a hyperbolic metric in $\cH(S, \bfs{\beta})$, $X_h\in\cT(S,\puct)$ its complex structure and let $b\in\cC_\infty(S,h)$.
Let $b=b_q+\Hess u-u\Id$ be the decomposition of $b$ given in Proposition \ref{pr:conedecomp}.
% Let $X_t$ be the complex structure on $S$ induced by the metric $\hat h_t=h(\Id+tb, \Id+tb)$.
Then  $$d\Hol_{X_h}([b_0])=-\Lambda\delta(Jb_q)\,,$$ 
where $\Lambda:H^1_{\tiny\hol}(\pi_1(S\setminus\puct), \R^{2,1})\to
H^1_{\tiny\Ad\circ\tiny\hol}(\pi_1(S\setminus\puct), \mathfrak{so}(2,1))$ is the isomorphism  induced by the $\SO(2,1)$-equivariant isomorphism $\Lambda:\R^{2,1}\rar \so(2,1)$.
\end{prop}

The proof of this proposition follows the same line as in the closed case,
but some technicalities come up.
We consider the hyperbolic metric $h_t$ with cone singularities and  $h_t=e^{2\psi_t}\hat h_t$, where $\hat h_t=h((\Id+tb)\bullet, (\Id+tb)\bullet)$.
By Proposition \ref{pr:deltastar}, 
\[
     \dot\hol=\frac{1}{2}\mathfrak{d}(h^{-1}\dot h)\,.
\]
Since $h^{-1}\dot h=2((\dot \psi-u)\Id+b_q+\Hess u)$, putting $\phi=\dot \psi-u$ one gets
$\Delta \phi-\phi= 0$ as in the closed case.
Then proving that $\phi\equiv 0$, one concludes as in the closed case.

Notice that respect to the closed case there are two technical points.
\begin{itemize}
\item One has to prove that the conformal factor $\psi_t$ smoothly depends on $t$ on $S\setminus\{\puct\}$.
The idea is to use the result of \cite{schumacher}, but we emphasize that we cannot apply directly it, as the conformal structure induced by
$h_t$ is not constant in a neighborhood of the punctures. So we need to use an isotopy to fix the conformal structures around the punctures.
\item
The second point is that the  proof that $\phi$ is zero is not immediate. In fact $\phi$ is defined only on the regular part
$S\setminus\puct$ which is not compact. Actually we will prove that $\phi$ continuously extends to the punctures and conclude by
adapting the maximum principle to the context of surfaces with cone singularity. 
The proof that $\phi$ continuously extends to the disc needs some careful analysis around the singular points.
%of the behavior of the isotopy $F_t$ used
%to get $h_t$ hermitian for the complex structure $X_0$ in a neighborhood of any puncture.
\end{itemize}

In the following lemma we summarize the technical construction of the isotopy $F_t$. 

\begin{lemma}
There is a smooth map $F:[0,\epsilon]\times(S\setminus\puct)\to(S\setminus\puct)$ such that
\begin{itemize}
\item For any $t$ the map $F_t(\bullet)=F(t,\bullet)$ extends to a  quasi-conformal homeomorphism fixing $\puct$.
\item $F_0$ is the identity.
\item There is a neighborhood $U$ of $\puct$ such that $F_t: (U, \cA_0)\to (U, \cA_t)$ is conformal for any $t$.
\item The variation field $Y=\left.\frac{dF_t(x)}{dt}\right|_{t=0}$ extends to a continuous field on $S$ with $Y(p_i)=0$.
\end{itemize}
\end{lemma}
\begin{proof}
First we construct the isotopy $F^{(i)}_t$ on a small neighborhood $U_i$ of a puncture $p_i$.
Notice that the Beltrami coefficient $\mu_t$ of the identity $(S, \cA_0)\to (S, \cA_t)$ corresponds to the symmetric traceless tensor $\frac{tb_0}{1+t (\mathrm{tr} b/2)}$
 under the usual identification, so $\mu_t$ smoothly depends on $t$.
By the classical theory of Beltrami equation we can find on a disc around the puncture a family of quasi-conformal maps
$G^{(i)}_t:(U_i, \cA_0)\to(U_i, \cA_0)$ with Beltrami coefficient $\mu_t$, such that the map $G^{(i)}:(-\epsilon,\epsilon)\times U_i\to U_i$ is smooth in $t$.
We may moreover suppose that $G^{(i)}_t(p_i)=p_i$ for every $t$.
Regarding $G^{(i)}_t$ as a  map $G^{(i)}_t:(U_i,\cA_t)\to (U_i, \cA_0)$, it is holomorphic, so the map defined by $F^{(i)}_t=(G^{(i)}_t)^{-1}$ satisfies
the requirements. Notice that the variation field 
$Y^{(i)}_t=\left.\frac{dF^{(i)}}{dt}\right|_{t=0}$ is defined on the whole $U_i$, is smooth outside the punctures and $Y_t(p_i)=0$.

Now take a neighborhood $U'_i$ of $p_i$ such that $\overline{U}'_i\subset U_i$ and choose a smaller neighborhood $U''_i$ such that
$F^{(i)}_t(U''_i)\subset U'_i$ for every $t$.
Take a smooth function $\chi$ on $S$ which vanishes on $S\setminus(\bigcup U_i)$ and is constantly $1$ over $\bigcup U'_i$.
Let $Y_t$ be the field defined by $Y_t=\chi Y_t^{(i)}$ over $U_i$ and $Y_t\equiv 0$ over $S\setminus(\bigcup U_i)$.
It can be readily shown that $Y_t$ generates a flow of maps $F_t\in \Homeo(S,\puct)\cap \Diffeo(S\setminus\puct)$ and $F_t\equiv F_t^{(i)}$
over $U''_i$. It is easy to check that $F_t$ verifies the requirements we need.
\end{proof}

\begin{proof}[Proof of Proposition \ref{prop:hol}]
We consider the metrics $k_t=F_t^*(h_t)$  conformal to $F_t^*(\hat h_t)$.
By \cite{schumacher} we know that they smoothly depend on the parameter $t$.
It follows that $h_t$ also smoothly depends on $t$. Moreover
we have
\[
    h^{-1}\dot k=h^{-1}\dot h+2\mathbf{S}\nabla Y\,.
\]
As $h^{-1}\dot h=(2\dot \psi-2u)\Id+2b_q+2\Hess u=2\phi\Id+2b_q+2\Hess u$, one deduces that
$\Delta\phi-\phi=0$ as in the closed case.

We claim that $\phi\equiv 0$. From the claim the proof follows exactly as in the closed case.
To prove the claim we first check that $\phi$ continuously extends to the punctures, then we use
 the maximum principle adapted to the case of surfaces with cone points that we prove separately in the next Lemma
 (notice that here $\phi$ solves the equation $\Delta\phi-\phi=0$ on $S\setminus\puct$ but this does not
imply that it is a weak solution of the equation on the closed surface). 

The proof of the continuity of $\phi$ around a puncture $p_i$ is articulated in the following steps:
\begin{itemize}
\item[Step 1] We will show that around  $p_i$ there is a smooth vector field
$V$ such that $h^{-1}\dot k=2\bfs{S}\nabla V$.
\item[Step 2] Writing $b_q=J\hess v-v J$ we have that $$\phi\Id=\bfs{S}\nabla(Y_1)\,,$$
where $Y_1=V-Y-J\grad v-\grad u$. This implies that $\bar\partial Y_1=0$, that is $Y_1$ is
a holomorphic vector field on $U\setminus\{p_i\}$ (see Remark \ref{remarkremark}).  We will prove that $Y_1$ extends at $p_i$ and $Y_1(p_i)=0$.
\item[Step 3] As $2\phi$ is the divergence of $Y_1$, the continuity of $\phi$ will be deduced by an explicit computation 
where we use that $Y_1$ is analytic around the puncture with $Y_1(p_i)=0$ 
(as the Christoffel symbols of the metric $h$ diverge around the puncture,  the condition $Y_1(p_i)=0$ 
will play a key role in the computation). 
\end{itemize}
 \vspace{10pt}
 \par\noindent
\emph{Step 1:}
\vspace{5pt}
\par
As $k_t$ smoothly depends on $t$ we can construct a smooth family of isometric embeddings
\[
s_t:(U, k_t)\to (U',h)
\]
where $U$ and $U'$ are fixed neighborhoods of $p_i$.
Up to shrinking $U$ we may suppose that on  $U$ the metric $k_t$ is conformal to $h$ for every $t$, so from
the conformal point of view we have a  smooth map
\[
s:(-\epsilon, \epsilon)\times U\to U'
\]
such that the restriction $s_t(\bullet)=s(t,\bullet)$ is holomorphic for every $t$.
It follows that  $V=\left.\frac{ds_t}{dt}\right|_{t=0}$ is a holomorphic field defined on the whole $U$.
As $s_t(p_i)=p_i$ we get that $V(p_i)=0$.
Finally notice that as $s_t^*(h)=k_t$ we have that $2\bfs{S}\nabla V=h^{-1}\dot k$.

\vspace{10pt}
 \par\noindent
\emph{Step 2:}
\vspace{5pt}
\par
As we know that $Y(p_i)=0$ and $V(p_i)=0$, it is sufficient to prove that
$\grad v$ and $\grad u$ vanish at $p_i$.

More generally we will prove that if $f$ is a function on $U$ such that
$||\hess f-f\Id||<Cr^{\alpha}$ for some $\alpha>-1$, then $\grad f$ extends to $0$ at $p_i$.
This general fact implies the extendability of $\grad v$, because $\Hess v-v\Id=-b_{iq}$ satisfies this condition as we noted in the proof of
\ref{integrability harmonic codazzi}. 
On the other hand $u$ can be regarded as the difference of the functions $u_1-u_2$ where $b=\Hess u_1-u_1\Id$
and $b_q=\Hess u_2-u_2\Id$. As $b$ is bounded and $b_q$ satisfies the condition above, we conclude that $\grad u$ extends as well. 

Let $g$ be the Klein Euclidean metric on $U$. As in the proof of Lemma \ref{lemma metrica euclidea}, let us put
$\bar f=(\ch r)^{-1}f$, where $r$ is the hyperbolic distance from the cone point.
By Lemma \ref{lm:kleincorr}, $\bar f$ satisfies the equation $D^2_g\bar f(\bullet, \bullet)=(\cosh r)^{-1} h((\Hess f-f\id) \bullet, \bullet)$, so $||D^2_g \bar f||_g<Cr^\alpha$.
Consider on the universal cover the gradient map $\tilde\varphi=(\dev)_* (\grad_g\bar f):\widetilde{U\setminus\{p_i\}}\to\R^{2}$.
If $\rho$ denotes the Euclidean radial coordinate and $\theta$ is the pull-back of the angular coordinate, we have

\begin{align*}
   ||\tilde \varphi(\rho_1, \theta_1)-\tilde \varphi(\rho_2,\theta_2)||&\leq ||\tilde \varphi(\rho_1, \theta_1)-\tilde \varphi(\rho_1,\theta_2)||
   +||\tilde \varphi(\rho_1, \theta_2)-\tilde \varphi(\rho_2,\theta_2)||\\
    &\leq C( \rho_1^{\alpha}\rho_1||\theta_1-\theta_2||+||\rho_2^{\alpha+1}-\rho_1^{\alpha+1}||)\,.
\end{align*}    

This shows that on each radial line there exists the limit $\lim_{\rho\to 0}\varphi(\rho, \theta)=\xi$. Moreover
 this limit does not depend
on $\theta$, and the convergence is uniform as far as $\theta$ lies in some compact interval of $\mathbb R$.
 As in the proof of Lemma \ref{lemma metrica euclidea},
 $\varphi(\rho, \theta+\theta_0)=R_{\theta_0}\varphi(\rho,\theta)$, we deduce that $\xi$ is a fixed point of the rotation, that is
 $\xi=0$. It turns out that $ ||\grad_g\bar f||_g\to 0$ at $p_i$.
As the hyperbolic metric $h$ is equivalent to $g$ we conclude that also $||\grad_h f||_h\to 0$ at the puncture $p_i$.
 \vspace{10pt}
 \par\noindent
\emph{Step 3:}
\vspace{5pt}
\par

Under the identification $K^{-1}=TS$ we have $Y_1=f(z)\frac{\partial}{\partial z}$ and 
$2\phi=\Div Y_1$. As in complex notation the connection form (compare Remark \ref{remarkremark}) is $$\omega=2\frac{\partial \eta}{\partial z}dz\,, $$
where $\eta$ is the conformal factor of the hyperbolic metric $h=e^{2\eta}|dz|^2$. It turns out that 
$\phi=\Re(f'(z)+2\partial_z \eta f(z))$. 
Notice that $\eta=\beta\log |z|+\xi$ where $\xi$ is a $C^1$ function on $D$ such that $||d\xi|||z|\rar 0$ 
(compare the explicit expression in Equation \eqref{eq:hypconf} for the hyperbolic metric in the conformal coordinate). So we get 
\[
\lim_{z\to 0}\phi=(1+\beta)\Re f'(0)\,,
\]
and in particular $\phi$ extends to a continuous function on $S$. 
\end{proof}

\begin{lemma}
Let $h\in\cH(S,\bfs{\beta})$.
If $\phi$ is a continuous function on $S$, and  on $S\setminus\puct$  is a smooth solution of
$\Delta_h \phi-\phi=0$, then $\phi\equiv 0$.
\end{lemma}
\begin{proof}
From the equation we know that if the maximum of $\phi$ is realized at an interior point, then it must be nonpositive.
We claim that the same holds if the maximum is realized at a puncture $p_i$. From the claim we can conclude that the maximum
must be nonpositive and analogously the minimum nonnegative, that is $\phi\equiv 0$.

To prove the claim we consider the function $F:[0,\epsilon)\to\R$ such that $F(r)$ is the average of $\phi$ over the circle centered at
$p_i$ of radius $r$. We fix $\epsilon$ so that all those circles are embedded in $S$.
Notice that by continuity of $\phi$ we have $\lim_{r\to 0} F(r)=\phi(p_i)$ and the assumption that $p_i$ is a maximum point for $\phi$
implies that $F(r)\leq F(0)$ for $r\geq 0$.

Now using  coordinates $r,\theta$ in a neighborhood of $p_i$ we have
\[
F(r)=\frac{1}{\theta_0}\int_0^{\theta_0}\phi(r, \theta)d\theta\,,
\]
so
\[
\begin{split}
\dot F(r)=\frac{1}{\theta_0}\int_0^{\theta_0}h(\grad\phi(r,\theta), \nu)d\theta
=\frac{1}{\theta_0\sh r}\int_{\partial B_r}h(\grad\phi(r,\theta), \nu)d\ell_r~,
\end{split}
\]
where $\nu$ is the normal field on $\partial B_r$ pointing outside.

Putting $G(r)=\int_{\partial B_r}h(\grad\phi(r,\theta), \nu)d\ell_r$, the Divergence  Theorem implies that
for $s<r$
\[
    G(r)-G(s)=\int_{B_r\setminus B_s}\Delta_h \phi \da_h=\int_{B_r\setminus B_s}\phi \da_h\,.
\]
As $\phi$ is bounded we have $|G(r)-G(s)|\leq K(r^2-s^2)$ for some constant $K$.
This implies that $G$ extends to $0$ and, putting $C_0=G(0)$ 
\begin{equation}\label{eq:dvf}
|G(r)-C_0|\leq Kr^2\,.
\end{equation}
Let us show that $C_0=0$. If $C_0\neq 0$, up to changing the sign of $\phi$ we may suppose $C_0>0$. 
Then by \eqref{eq:dvf} we get
\[
\left|\theta_0\dot F(r)-\frac{C_0}{\sh r}\right|\leq K'r\,.
\]
This implies that $\theta_0\dot F\geq\frac{C_0}{\sh r}-K'r$, but this contradicts the fact that $\dot F$ is integrable on $[0,\epsilon)$.
Thus $C_0=0$ so $|\dot F(r)|<K' r$, that  implies that $\dot F(r)\to 0$ as $r\to 0$.

Now 
\[
\begin{split}
\theta_0(\dot F(r)-\dot F(s))=\frac{1}{\sh r}(G(r)-G(s))+\left(\frac{1}{\sh r}-\frac{1}{\sh s}\right)G(s)=\\
=\frac{1}{\sh r}\int_{B_r\setminus B_s}\Delta_h \phi\da_h+\frac{\sh s-\sh r}{\sh s\,\sh r}G(s)~.
\end{split}
\]
Notice that the last addend tends to $0$ as $s\to 0$, so we deduce
\[
\dot F(r)=\frac{1}{\theta_0\sh r}\int_{B_r}\Delta_h \phi\da_h=\frac{1}{\theta_0\sh r}\int_{B_r} \phi\da_h~.
\]
Now as $F(r)\leq F(0)$ we get that $\dot F(r)$ must be nonpositive for small $r$, and this implies that $\phi(0)$ 
cannot be positive. Analogously one shows that if the minimum is achieved at a puncture, then it must be nonnegative and this concludes that $\phi\equiv 0$.
\end{proof}

\begin{corx}\label{cor:main2}
Two embedding data $(I,s)$ and $(I', s')$ in $\E_{\bfs{\beta}}$
 correspond to Cauchy surfaces contained in the same space-time with particles
if and only if 
\begin{itemize}
\item the third fundamental forms $h$ and $h'$ are isotopic;
\item the infinitesimal variation of $h$ induced by $b$ is Teichm\"uller equivalent to the infinitesimal variation of $h'$ induced by $b'$.
\end{itemize}
The map $\Psi$ induces to the quotient a bijective map
\[
\bar\Psi: \cS_+(S,\bfs{\beta})\to T\cT(S, \puct).
\] 
%\end{theorem}
\end{corx}
\begin{proof}
The first part directly follows by Proposition \ref{existence and uniqueness with singularities}.
The fact that $\bar\Psi$ is well-defined and injective is then a consequence of commutativity of \eqref{eq:commcone}.
Notice that  $\Psi:\cC(S, h)\to T_{[X]}(S, \puct)$ is surjective by a simple dimensional argument.
As for any $b\in\cC_\infty(S, h)$ we may find a constant $M$ so that $b+M\Id$ is positive. Like in the closed case we conclude that
$\bar\Psi$ is surjective.
\end{proof}

\subsection{Symplectic structures in the singular case}
In the singular case it is also possible to construct a Goldman intersection form $\omega^B$ 
on the image of $d\Hol$ in $H^1_{\tiny\Ad\circ\hol}(\pi_1(S), \so(2,1))$.
Mondello \cite{mondelloPoisson} proved that the map $d\Hol$ is symplectic up to a factor.
We will give here a different proof of that result in the analogy of the proof of Corollary \ref{corollarytheoremgoldman} given in Subsection \ref{Symplectic forms}.

First let us recall some basic facts on the construction of $\omega^B$.
We denote by $H^\bullet_{\mathrm c}(S, F_{\so(2,1)})$ the de Rham cohomology of the complex of $F_{\so(2,1)}$-valued 
forms on $S$ with compact support, and let 
\[
    I_*:H^1_{\mathrm c}(S, F_{\so(2,1)})\to H^1_{\mathrm{dR}}(S, F_{\so(2,1)})
\]
be the map induced by the inclusion.
The image of $I_*$ will be denoted by $H^1_{0}(S, F_{\so(2,1)})$ and contains the cohomology classes in
$H^1_{\mathrm{dR}}(S, F_{\so(2,1)})$ which admit a representative with compact support.
Under the isomorphism $H^1_{\tiny\Ad\circ\tiny\hol}(\pi_1(S), \so(2,1))\cong H^1_{\mathrm{dR}}(S, F_{\so(2,1)})$,  elements of
$H^1_0(S, F_{\so(2,1)})$ correspond to cocycles which are trivial around the punctures.

Let $B$ be the $\Ad$-invariant product on $\so(2,1)$ defined in Subsection \ref{Symplectic forms}. It induces a well-defined pairing
\[
   \bar\omega^B:H^1_{\mathrm c}(S, F_{\so(2,1)})\times H^1_{\mathrm{dR}}(S, F_{\so(2,1)})\to\R
\]
given as in the closed case by
\[
   \bar\omega^B([\varsigma], [\sigma])=\int_{S}B(\varsigma\wedge\sigma)\,.
\]
%By definition it follows that if $\varsigma$ is a closed $1$-form with compact support and $\sigma$ is a closed form.
This pairing is nondegenerate by Poincar\'e duality.
 Notice that if $\varsigma, \varsigma'$ are forms with compact support
$\bar\omega^B([\varsigma], I_*[\varsigma'])=-\bar\omega^B([\varsigma'], I_*([\varsigma]))$, showing
that $\ker I_*$ coincides with the orthogonal subspace of $H^1_0(S, F_{\so(2,1)})$.

\noindent Thus the form $\bar\omega^B$ induces a symplectic form on $H^1_0(S, F_{\so(2,1)})$, defined by
\[
  \omega^B([\sigma], [\sigma'])=\int_SB(\sigma\wedge\sigma')~,
\]
where $\sigma$ and $\sigma'$ are representatives with compact support.

In a similar  way we can define the subspace $H^1_0(S, F)$ and a symplectic form $\bar\omega^F$ on it, 
analogous to the one constructed  in Subsection \ref{Symplectic forms}

By Proposition \ref{pr:trivial}, the coboundary operator $\delta:\cC_2(S,h)\to H^1_{\mathrm{dR}}(S, F)$
takes values in $H^1_0(S, F_{\so(2,1)})$.
The following proposition computes $\bar\omega^B(\Lambda\delta b, \Lambda\delta b')$ in analogy with Proposition \ref{pr:omegaB}.
\begin{prop}
Let $\delta:\cC_2(S,h)\to H^1_0(S,F)$. Then 
\begin{equation}\label{eq:omegaFsing}
  \bar\omega^F(\delta (b), \delta (b'))=\frac{1}{2}\int_S \tr(Jbb')\da_h\,.
\end{equation}
or analogously 
\begin{equation}\label{eq:omegaBsing}
\bar\omega^B(\Lambda(\delta (b)), \Lambda(\delta (b')))=\frac{1}{4}\int_S \tr(Jbb')\da_h\,.
\end{equation}
%If moreover $b,b'\in\cC_\infty(S,h)$, then in the RHS of the above expressions $b_{q}$ and $b_{q'}$ can be replaced by $b$ and $b'$.
\end{prop}
\begin{proof}
%Notice that $\delta:\cC_2(S, h)\to  H^1_{\mathrm{dR}}(S, F)$ is continuous for the $L^2$ norm. 
%This follows from the fact that the projection of $\cC(S,h)$  on the finite dimensional subspace
%$$V=\{b_q\,|\,q\text{ is a holomorphic quadratic differential with at worst simple poles at }\puct\}$$ 
%$V=\{b_q\}$ of traceless Codazzi tensors
%given in Proposition \ref{pr:conedecomp} is orthogonal and thus continuous.

Since $\delta$ is continuous for the $L^2$-norm by Proposition \ref{pr:contdelta}, both LHS and RHS in \eqref{eq:omegaFsing} and \eqref{eq:omegaBsing}
 are continuous on $\cC_2(S, h)\times\cC_2(S, h)$.
So by density it is sufficient to prove that \eqref{eq:omegaFsing}  holds for $b$ and $b'$ with compact support.
But in that case, the proof is the same given in the closed case, recalling that $\delta b=[\bfs{\iota}_*b]$.
\end{proof}

Now we  compute the Weil-Petersson form in terms of the map $\Phi$.
Let us denote by $K^2(\puct)$ the space of holomorphic quadratic differentials with at worst simple poles at $\puct$. Recall that if $q,q'\in K^2(\puct)$
then as in the closed case one can define $g_{\WP}(q, q')$ by integrating the form that in a local chart is
\[
      \frac{f\bar g}{e^{2\eta}}dx\wedge dy\,,
\]
for $q=f(z)dz^2$, $q'(z)=g(z)dz^2$ and $h=e^{2\eta}|dz|^2$.
In fact the integrability of that form relies on the fact that $q$ and $q'$ have at worst simple poles.
We have the same result as in the closed case:
\begin{prop}\label{omegaWPsing}
Given $b,b'\in\cC_\infty(S,h)$, the following formula holds:
\begin{equation}\label{eq:omegaWPsing}
   \omega_{\WP}(\Psi(b), \Psi(b'))=2 \int_S \tr (Jbb')\da_h\,.
\end{equation}
\end{prop}

The computation is done as in the closed case up to a simple technical difficulty.
The point is that strictly speaking if $b\in\cC_\infty(S,h)$, then $b_q$ is only $\cC_2(S, h)$, 
so in order to use the splitting $b=b_q+\hess u-u\Id$ we need to extend the map $\Psi$  to $\cC_2(S, h)$.

The reason this is possible is that the pairing 
\[
   \cC_\infty(S,h)\times \mathcal K^2(\puct)\to\mathbb C\,,\quad \langle q,\Psi(b)\rangle=\int_S q\bullet \Psi(b)
   \]
continuously extends to a pairing $\cC_2(S,h)\times\mathcal K^2(\puct)\to\mathbb C$. 
In fact as the proof of Equation \eqref{eq:cont} was only local, it holds also in this singular case and we have
\[
       \langle q,\Psi(b)\rangle=-\int_S(\tr(Jb_0 b_q)+i\tr(b_0 b_q))\da_h\,.
\]
Moreover, the antilinear map $\mathcal K^2(\puct)\to T_{X_h}\cT(S,\puct)$ given by the Weil-Petersson product is 
$$q\mapsto=\Psi(\frac{Jb_q}{2})$$
as in the closed case, which allows to recover the result. This concludes also the alternative proof of the result of Mondello (\cite{mondelloPoisson}).

\begin{cor} \label{corollarytheoremgoldmansing}
The Weil-Petersson symplectic form $\bar\omega_{\WP}$ and the Goldman symplectic form $\bar\omega^{B}$ for hyperbolic surfaces with cone points are related by:
$$\Hol^*(\bar\omega^B)=\frac{1}{8}\bar\omega_{\WP}\,.$$
\end{cor}

%----------------------------------------------sezione 6---------------------------------

%\input bs-exotic.tex

\section{Exotic structures} \label{exotic}
We now want to discuss flat Lorentzian manifolds which do not satisfy the hypothesis we considered so far.

\subsection{The existence of a strictly convex Cauchy surface}
Recall we are considering MGHF space-times on $S\times \R$, where $S$ is a possibly singular surface of genus $g$ and the cone angles correspond to a divisor $\bfs{\beta}=\sum_{i=1}^k \beta_i p_i$ with
$\beta_i\in(-1,0)$. 
We have always  assumed $\chi(S, \bfs\beta)<0$. 
In \cite{Mess}, Mess proved that in the case of a closed surface (i.e. no cone points), every MGHF space-time contains a strictly convex Cauchy surface provided $\chi(S)<0$. 
We now show that this is not true in general if cone singularities are allowed.

\begin{example}
Let $M$ be a MGHF space-time with a particle $s$ of angle $\theta<\pi$ which contains a uniformly convex Cauchy embedding orthogonal to the singular lines. Such space-times are classified in Theorem \ref{thm:coneteich}. It is easy to find another Cauchy embedding $\sigma:S\rar M$ which is is flat (and thus not strictly convex) in a neighborhood of its intersection with the particle $s$. Hence there is a neighborhood $U$ of the singular point $p$ such that $\sigma(U)$ lies in the orthogonal plane to $s$ at $\sigma(x)$, so that the induced metric is Euclidean. By taking $U$ sufficiently small, we suppose $\sigma(U)$ does not intersect any other particle.

%We can work in a local model where $g$ is developed to a vertical line in $\R^{2,1}$ and $\sigma(U)$ lies in a horizontal plane. We now want to cut a cylinder $C$ with basis $\sigma(U)$ and replace it with a new cylinder $C'$ containing two particles $g'$ and $g''$. It also suffices to perform this operation in a neighborhood of the Cauchy surface $\sigma(S)$; our new spacetime $M'$ will be the maximal extension. For this purpose, we will need the holonomy of $\gamma'\gamma''$ for $M'$ to coincide with the holonomy of $\gamma$ for $M$, where $\gamma,\gamma',\gamma''$ are peripheral paths around $g,g',g''$ respectively.

Hence we find a neighborhood $\sigma(U)\times(-\epsilon,\epsilon)$ where the metric takes the form $g_U-dt^2$, for $g_U$ is a Euclidean metric on $U$ with a cone point $p$ of angle $\theta$. We are now going to cut this neighborhood of $\sigma(U)$ and glue a germ of flat space-time containing two cone points.

\begin{figure}[b]
\centering
\begin{minipage}[c]{.45\textwidth}
\centering
\includegraphics[height=5.5cm]{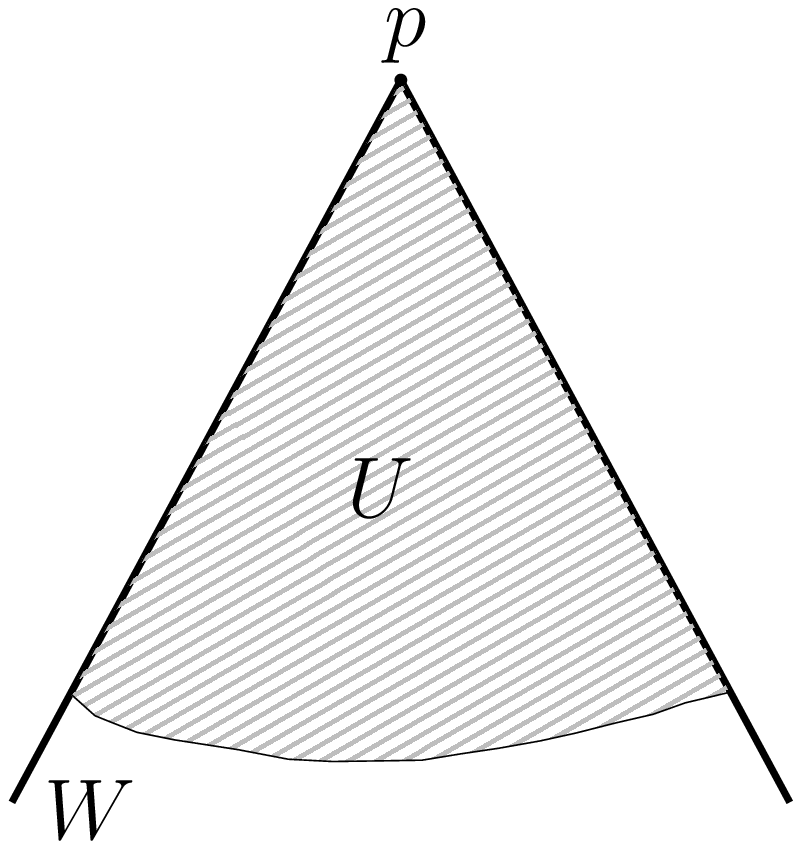} 
%\captionsetup{labelformat=empty}
%\caption{The lightcone of future null geodesic rays from a point and a totally geodesic plane $P$.} \label{fig:lightcone}
\end{minipage}%
\hspace{5mm}
\begin{minipage}[c]{.45\textwidth}
\centering
\includegraphics[height=5.5cm]{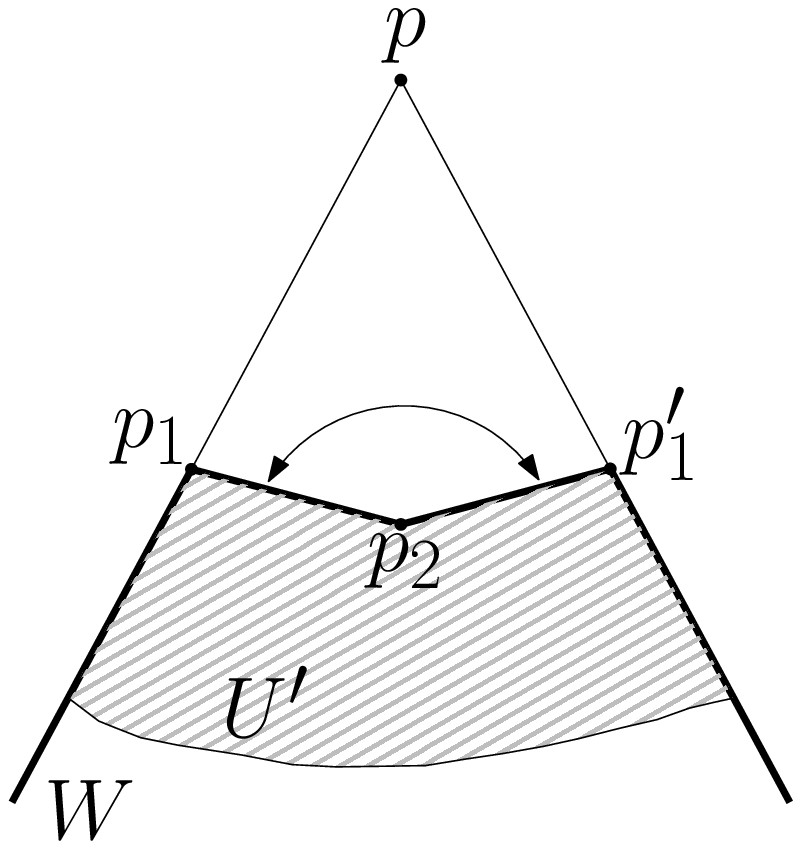}
%\captionsetup{labelformat=empty}
%\caption{Left and right projection from a point $\xi\in\partial_\infty\AdS^3$ to the plane $P=\left\{x_3=0\right\}$} \label{fig:projections}
\end{minipage}
\caption{Replacing a neighborhood of a cone point on a Euclidean structure by a disc with two cone points. The cone angles satisfy the relation $\theta_1+\theta_2=2\pi+\theta$.} 
\label{fig:wedge}
\end{figure}

Our construction is two-dimensional Euclidean geometry; see also Figure \ref{fig:wedge}. Consider a wedge $W$ in $\R^2$ of angle $\theta$, which represents a model of the cone point. We choose two points $p_1$ and $p_2$ in $U$ (which will represent the new cone points), $p_1$ in the boundary of the wedge and $p_2$ in the line bisecting $W$. Let $p_1'$ be the image of $p_1$ by the rotation of angle $\theta$, namely the point identified to $p_1$ on the other edge of $W$. Connect $p_2$ to $p_1$ and $p_1'$ by geodesic segments. We remove the quadrilateral $Q=p_1 p p_1' p_2$ from $W$ and glue the segments $\ol{p_1 p_2}$ and $\ol{p_1' p_2}$ by a rotation around $p_2$. We keep the same gluing as before between the two edges of $W$ outside $Q$. This gives a Euclidean structure on a disc $U'$, obtained from $W\setminus Q$ by the gluing we defined, containing two cone points of angle $\theta_1$ and $\theta_2$. Observe that at least one of $\theta_1$ and $\theta_2$ has angle between $\pi$ and $2\pi$. By some simple Euclidean geometry we have that $\frac{\theta}{2\pi}-1=\frac{\theta_1}{2\pi}-1+\frac{\theta_2}{2\pi}-1$.

%From this point, we only deal with the Euclidean structure on $U\subseteq S$, since both $C$ and $C'$ are obtained by taking the product of $\sigma(U)$ (endowed with a Euclidean singular metric) with a real interval. Consider a Euclidean structure on a disc $U$ having two cone singularities at distance $d<r$ of angles $\theta'$ and $\theta''$, where $\theta=\theta'=\theta''-2\pi$. Note that this implies that at least one of $\theta'$ and $\theta''$ is greater than $\pi$. Such a structure can be modified in the following way. Cut along the Euclidean geodesic segment $s$ connecting the two cone points, so as to create two geodesic segments $s_1$ and $s_2$ of lenght $d$. Then a Euclidean triangle with vertices $a_1,b_1,c_1$ of angles $\pi-\frac{\theta'}{2}$, $\pi-\frac{\theta''}{2}$ and $\frac{\theta}{2}$ respectively and having the edge $\ol{a_1 b_1}$ of lenght $d$; take also a second isometric triangle whose vertices we denote $a_2,b_2,c_2$. Glue these two triangles in by identifying $\ol{a_1 b_1}$ to $l_1$, $\ol{a_2 b_2}$ to $l_2$, $\ol{a_1 c_1}$ to $\ol{a_2 c_2}$ and $\ol{b_1 c_1}$ to $\ol{b_2 c_2}$. In this way we have obtained a new Euclidean structure with a single cone point of angle $\theta$ (where the vertices $c_1$ and $c_2$ are glued). (METTERE FIGURE)

We extend this operation to $U' \times (-\epsilon,\epsilon)$ in the obvious way and glue the new structure to a tubular neighborhood of $\sigma(S)\setminus \sigma(U)$ in $M\setminus (\sigma(U) \times (-\epsilon,\epsilon))$. By taking the maximal extension, we obtain a space-time $M'$ with two particles $s_1$ and $s_2$ of angles $\theta_1$ and $\theta_2$. Notice that the Euler characteristic of $M'$ equals that of $M$ so it is negative.
It is also clear that $M'$ cannot contain any strictly convex Cauchy surface, as the requirement of being orthogonal to the singularities forces a spacelike surface to be flat (i.e. the shape operator has a null eigenvalue) at some points. More precisely, $s_1$ and $s_2$ are connected by a geodesic segment entirely contained in $S$.

Observe that the holonomy of $M'$ of a path winding around both $s_1$ and $s_2$ is the same as the holonomy of $M$ of a path around $s$. Moreover, on peripheral paths around $s_1$ and $s_2$, the linear part of the holonomy of $M'$ fixes the same point in $\Hyp^2$.
\end{example}

%From the above example, we can in general state the following:
%\begin{prop}
%Given any surface $S$ and any choice of cone angles $\theta_1,\ldots,\theta_n$ such that $2g-2-\sum_{i=1}^n(2\pi-\theta_i)<0$ and $\theta_i>\pi$ for at least one index $i$, there exists a MGH flat spacetime $M$ homeomorphic to $S\times \R$ which does not contain any strictly convex spacelike Cauchy surface.
%\end{prop}

%\begin{remark}
%It is also clear that the spacetimes so obtained do not contain any strictly convex Cauchy surface, since the linear part of the holonomy is not the holonomy of any hyperbolic singular metric on $S$. Indeed, the holonomy image of peripheral paths $\gamma'$ and $\gamma''$ fix the same point in $\Hyp^2$.
%\end{remark}

%\nocite{*}

%\cleardoublepage
\bibliographystyle{alpha}
\bibliography{bs-bibliography}

\end{document}